\documentclass[11pt, a4paper]{article}
\usepackage{graphicx} % Required for inserting images

\usepackage[algo2e,ruled,vlined]{algorithm2e}
\usepackage[noadjust]{cite}
\usepackage{amsmath, amsfonts, amsthm}
\usepackage{graphicx,soul}
\usepackage[colorlinks=true, allcolors=blue]{hyperref}
\usepackage{amsfonts}
\usepackage{color,xcolor}
\usepackage{multirow}
\usepackage{soul}

\usepackage{amssymb,latexsym}

\usepackage{algorithm}
\usepackage{algorithmicx}
\usepackage{algpseudocode}
\usepackage{mathtools}
\newif\ifPDF
\ifx\pdfoutput\undefined
\PDFfalse
\else
\ifnum\pdfoutput > 0
\PDFtrue
\else
\PDFfalse
\fi
\fi

\ifPDF
\usepackage{pdftricks}
\begin{psinputs}
	\usepackage{pstricks}
	\usepackage{pstcol}
	\usepackage{pst-plot}
	\usepackage{pst-tree}
	\usepackage{pst-eps}
	\usepackage{multido}
	\usepackage{pst-node}
	\usepackage{pst-eps}
\end{psinputs}
\else
\usepackage{pstricks}
\fi

\usepackage[openbib]{currvita}

\usepackage{fancyhdr}

\newtheorem{theorem}{Theorem}[section]
\newtheorem{lemma}[theorem]{Lemma}

\usepackage{nicematrix}
\newcommand{\dint}{\displaystyle\int}
\newcommand{\eps}{\varepsilon}

\newcommand{\bbR}{\mathbb R} \newcommand{\bbS}{\mathbb S}

\newcommand{\cA}{\mathcal A}

 \newcommand{\cH}{\mathcal H}
 
 \newcommand{\cL}{\mathcal L}
 
\newcommand{\cO}{\mathcal O}  
 
 \newcommand{\cT}{\mathcal T}

\newenvironment{keywords}
{\noindent{\bf Key words.}\small}{\par\vspace{1ex}}

\makeatletter
\newcommand{\chapterauthor}[1]{%
	{\parindent0pt\vspace*{-25pt}%
		\linespread{1.1}\large\scshape#1%
		\par\nobreak\vspace*{35pt}}
	\@afterheading%
}
\makeatother

\usepackage[table]{xcolor}% http://ctan.org/pkg/xcolor
\usepackage{graphicx}
\usepackage{subcaption}
\newcommand{\G}{\mathcal{G}}

\newcommand{\ReLU}{\text{ReLU}}

\title{From Frequency Bias to Spectral Balance: \\ Operator-Aware Preconditioners for PINNs}
\author{Roy Y. He~\thanks{Department of Mathematics, City University of Hong Kong, Kowloon Tong, Hong Kong}\and Ying Liang~\thanks{Department of Mathematics, Duke University, Durham, NC, USA}\and Hongkai Zhao~\thanks{Department of Mathematics, Duke University, Durham, NC, USA}\and Yimin Zhong~\thanks{Department of Mathematics and Statistics, Auburn University, Auburn, AL, USA}}
\date{\today}
\begin{document}
\maketitle
\begin{abstract}
% When neural networks (NNs), a special form of nonlinear parametrization for approximating functions, are used to solve partial differential equation (PDE), while enjoying mesh-free, easy implementations, scalable to high dimensions, and the potential of adaptivity, this nonlinear representation may lead to complications of the formulation and, more importantly, a significant computational hurdle: the necessity of solving a large-scale non-convex optimization problem, for which only gradient-based methods are practically feasible. 
% Consequently, a basic computational question arises: what can we expect regarding the accuracy, stability, convergence, and computational cost when using NNs to solve PDEs? 
% In this work, we will mathematically and numerically investigate some of the above issues when NNs are used to solve elliptic PDEs. In particular, we show that the presence of differential operators in the loss function can result in ill-conditioning and high frequency bias, opposite to the typical low frequency bias in direct NN representation and training, in the training process which leads to slow convergence and poor accuracy when gradient-based methods are used. We propose a simple and effective operator-aware preconditioner, based on an integral operator, that results in a rebalance of the optimization landscape and much improved convergence.  We use extensive numerical experiments, including multiscale and variable-coefficient problems, to corroborate our study.

When neural networks (NNs) are used as a type of nonlinear parametric representation to solve partial differential equations (PDEs), they often display frequency-dependent learning dynamics that can differ from those seen in direct function approximation tasks, resulting from a balance between the frequency bias of the NN representation and that of the underlying differential operator. Although many commonly used NNs exhibit a bias towards low-frequency modes in representation, the presence of differential operators in the loss function, which amplifies high-frequency components, can lead to high frequency bias. In this work, using second order elliptic PDEs as an example, we show how these two factors compete and lead to an overall frequency bias in different situations. Once the balance is determined, it is important to design computational strategies to counter the resulting bias to improve training efficiency. We propose a simple operator-aware preconditioning strategy that rebalances the optimization landscape and the learning dynamics by applying an auxiliary integral operator to the residual. The integral kernel can be the Green's function of a reference elliptic operator or an approximation, and integrates easily with common NN solvers for PDEs. Extensive experiments, including multiscale and variable-coefficient problems, show that the approach restores more balanced learning dynamics across modes and substantially improves both convergency and accuracy.
\end{abstract}
\section{Introduction}

Neural networks (NNs) provide a special form of nonlinear parametric representation for approximating functions. There are extensive results in terms of mathematical approximation theory for NNs. One of the main themes is showing universal approximation property (UAP) of NNs~\cite{cybenko1989approximations,barron2002universal,devore2021neural,kidger2020universal, lu2017expressive, eldan2016power, zhang2022deep}. However, these existing mathematical results do not address the key challenges of how close and how costly one can compute the solution by solving a large-scale non-convex optimization problem using gradient-based techniques, which are sensitive to the conditioning of the parametrized representation and often become the computational bottleneck in practice. 

A common feature in NN representation is the low-frequency bias. The frequency bias refers to the network's tendency to learn low-frequency components of a target function more quickly than high-frequency ones during training, which makes it computationally more difficult and costly for networks to capture high-frequency components and fine features of target functions. This phenomenon was first systematically demonstrated by \cite{rahaman2019spectral}, who analyzed the Fourier spectra of neural network outputs and showed that standard architectures with ReLU or Tanh activation functions prioritize smoother, low-frequency features, even without explicit regularization. Frequency bias can be interpreted through the lens of Neural Tangent Kernel (NTK) theory, in which training dynamics are decomposed into eigenspaces of the NTK, each associated with distinct kernel eigenvalues \cite{jacot2018neural, cao2021towards}. 
The NTK eigenvalues corresponding to high-frequency eigenfunctions decay polynomially with frequency, leading to slower learning of high-frequency components. While this implicit regularization can be beneficial for promoting smoothness, it becomes problematic when fine details or multiscale features are essential for accurate approximation. An explicit theoretical and numerical analysis~\cite{zhang2025shallow} showed ill-conditioning and frequency bias for two-layer NNs in both representation and training, due to strong correlations among parametrized global activations.

Naturally, neural networks (NNs) can be used to represent/approximate the solutions of PDEs. Numerous works have explored this approach, including Physics-Informed Neural Networks (PINNs)~\cite{raissi2019physics,raissi2020hidden}, the Deep Ritz Method (DRM)~\cite{weinan2018deep},  Weak Adversarial Networks (WAN)~\cite{zang2020weak}, and among others~\cite{lu2021learning,li2020fourier,de2024wpinns}. Significant efforts have been devoted to the theoretical justification of these methods. For instance, the consistency of PINNs for linear second-order elliptic and parabolic PDEs was established in~\cite{shin2020convergence}, while the generalization error of PINNs for data assimilation was analyzed in~\cite{mishra2023estimates} by examples. Analogous results for the DRM are also available~\cite{muller2019deep, duan2022convergence, lu2021priori, muller2022error}. The general conclusion from these works is that, given sufficient data and appropriate handling of boundary conditions, the minimizer of the associated optimization problem converges to the solution of the PDE.

However, these theoretical guarantees often do not translate easily into practice. Although solving PDEs using NN representation can potentially offer significant advantages, such as being mesh-free, easy implementation, scalable to high dimensions, and the potential of adaptivity, the nonlinear representation has to be solved through large-scale non-convex optimization using gradient-based methods in real practice.
%the formulation of which can be subtle, e.g., how to enforce the differential equation along with its boundary and/or initial conditions correctly. Furthermore, it leads to a significant computational hurdle: the necessity of solving a large-scale, non-convex optimization problem, for which only gradient-based methods are practically feasible.  
Again, ill-conditioning and frequency bias can become problematic for numerical computations, particularly for those PDEs with solutions containing multiple frequency scales or sharp features.  The presence of differential operators introduces additional factors that affect the overall ill-conditioning and frequency bias in the training process. \cite{wang2021eigenvector} recasts spectral bias in PDE solving tasks as an ``NTK eigenvector bias'', and also shows that incorporating Fourier feature mappings can reshape a network's spectral preferences, enabling better learning of the high-frequency components. 
%In a related contribution, \cite{yutuning} demonstrates that adopting Sobolev norms as loss functions for approximating functions can systematically reverse the typical low-frequency bias. Although the task in \cite{yutuning} is not solving PDEs, the Sobolev norm introduced derivative information, and their analysis of training dynamics reveals that penalizing solutions in a Sobolev space alters the emphasis on different Fourier modes, offering a principled way to control or invert spectral bias during training. This has direct implications for neural solvers of PDEs, where the ability to tune frequency sensitivity is often essential for capturing fine-scale solution structures.
In our earlier work \cite{he2025can}, we showed that ill-conditioning and low-frequency bias for two-layer NNs still dominate the opposite effect of a differential operator in the training dynamics. This consequence leads to numerical challenges for efficiency and accuracy compared to linear representations, such as the finite element method, with well-developed preconditioners. %Hence, using two-layer NNs can not compete against well-developed linear representations such as finite element methods. 

% Our study in this work will focus on the frequency bias when solving elliptic PDEs using multi-layer NNs.
In this work, we will focus on the multi-layer NNs. The story is quite different from the one for two-layer NNs. First of all, the representation capability and flexibility for multi-layer NNs are much better than two-layer NNs. In particular, its low-frequency bias can be influenced by several factors, such as network architecture, activation functions, weight initialization, and the trajectory followed by gradient descent during optimization. When trained for approximation,  it is demonstrated in \cite{basri2020frequency} and \cite{tancik2020fourier} that input encoding schemes and network depth significantly influence the frequency content learned during training.  In \cite{hong2022activation}, it is shown that the choice of activation function also plays a crucial role, while networks with {ReLU} as activation function exhibit spectral bias connected to finite element theory, alternative activations like piecewise linear B-splines can remove this bias. Recent theoretical advances in \cite{zhang2025shallow} demonstrate that two-layer neural networks act as `low-pass filters' due to ill-conditioning of the Gram matrix, and in follow-up works~\cite{zhang2025structured, zhang2025fourier},  it is further explored that with a multi-component and multi-layer neural network (MMNN) structure,  Sine activation function, and properly scaled random initialization of the first layer, NNs can be trained to approximate high-frequency modes effectively.
Furthermore, properties of the optimization process, such as the evolution of sharpness \cite{barrett2020implicit} and the path traversed through the loss landscape, also modulate mode dynamics. 

When using multi-layer NN representation to solve PDEs, the involvement of differential operators in the loss function, such as in PINN, which inherently introduce high frequency amplification/bias, can be a potential game changer of the overall ill-conditioning and frequency bias of the formulation and the learning dynamics. In particular, we show that, due to the flexibility, capacity, and mild frequency bias of well-designed multi-layer NNs, such as Fourier MMNNs using Sine activation function with proper initial scaling, ill-conditioning and high-frequency bias of an elliptic operator will dominate the overall formulation and the corresponding learning dynamics. 
As a comparison, when the finite element method is used for an elliptic PDE, since the finite element representation has no frequency bias, the ill-conditioning and high frequency bias of the resulting linear system are inherited from the differential operator. As a result, when simple iterative methods are used to solve the linear system, errors in low frequencies decay very slowly, leading to slow convergence even for smooth solutions. This is unacceptable in practice and requires the design of effective preconditioning strategies. Similarly, motivated by the observation that the underlying differential operator creates the overall frequency bias in the formulation, we propose an operator-aware preconditioning strategy that rebalances the spectral weights and training dynamics accordingly.  
%Our strategy decouples two sources of frequency bias: the neural network architecture and the differential operator. Following insights from recent works on neural network resolution, we first minimize the architectural bias through proper weight scaling. 
% Specifically, we utilize {Sine} activation function and initialize first-layer weights with variance proportional to the network width raised to a power, which ensures the network has sufficient representation capability to represent all relevant frequencies with minimal inherent bias.
%We then address the operator-induced bias separately through preconditioning. 
The preconditioner is a linear operator based on a reference elliptic operator that modifies the optimization landscape, effectively reshaping the loss function to balance frequency components. This preconditioner transforms the residual before computing the loss, counteracting the frequency amplification caused by the differential operator. The preconditioner acts as a spectral filter that attenuates the differential operator-induced bias while preserving the physics constraints, accelerating the convergence to an accurate solution.

% However, the involvement of differential operators, which automatically introduces high frequency amplification/bias, is a potential game changer. 
% Here we use the two most popular formulations: Physics Informed Neural Networks (PINNs) and the Deep Ritz Method (DRM), as examples for our study. First, we identify the potential issues, from the formulation to its computational implications. For example, we first show that the weak formulation of DRM makes it more sensitive to the balance of the enforcement of Dirichlet boundary condition and the use of mini-batch in a not well over-sampled regime. For PINNs, due to the use of strong formulation, it leads to ill-conditioning and frequency bias, which poses serious computational challenge to any gradient-based optimization. In particular, we show that the presence of differential operators in the loss function can result in high frequency bias, opposite to the typical low frequency bias in NN representation, which leads to slow convergence and poor accuracy especially when the solution containing multiple scales or sharp features. We will design a simple and effective operator-aware preconditioner based on an integral operator corresponding to the inversion of the free space Laplace equation. We show that this results in a rebalance of the optimization landscape and much improved convergence. 

Several other approaches have been proposed to alleviate the ill-conditioning and frequency bias of PINNs. Sobolev-norm losses replace the standard squared residual with norms that include derivative information, effectively giving more weight to high-frequency components; see, for example, Sobolev-PINNs \cite{son2021sobolev} and the analysis of Sobolev acceleration \cite{lu2022sobolev}. More recently, \cite{de2023operator} analyzes PINN training through an operator-conditioning lens and proposes explicit preconditioners that reduce the condition number of the composite neural tangent kernel and PDE operator. Related strategies include weighted-residual methods that use spatially-varying weight functions to balance the loss, learned preconditioners that adaptively reweight different spectral modes during training, and residual-based attention mechanisms \cite{anagnostopoulos2024residual}  that dynamically emphasize high-residual collocation points. 
These strategies often require additional computational overhead, such as computing high-order derivatives, solving auxiliary eigenvalue problems, or extensive hyperparameter tuning. Different from the works above, our approach requires only the Green's function of a simple reference operator, often available in closed form, 
%and it remains effective even with approximate Green's functions, as demonstrated by our order-of-magnitude accuracy improvements using free-space kernels for bounded domain problems, 
offering a simple, effective, and interpretable mechanism to counteract the frequency bias in training process.

 Our contributions include:
\begin{itemize}
\item Systematic study of the balance of frequency biases between a NN representation and a differential operator when solving elliptic PDEs using NNs, and show 
how differential operators can dominate and result in undesirable high frequency bias and slow convergence in the training process.
\item
An operator-aware preconditioning strategy that reweighs frequency modes in the formulation and leads to balanced and effective learning across frequency modes.
\item 
Comprehensive numerical experiments demonstrate that our preconditioning strategy can improve both efficiency and accuracy significantly, especially for multiscale problems.
\end{itemize}

The remainder of this paper is organized as follows. Section~\ref{Sec: Formulation} presents the problem formulation and establishes notation for both PDE solving and function approximation tasks. Section~\ref{Sec: Dynamics} provides detailed experimental motivation, demonstrating how frequency bias manifests differently in approximation versus PDE solving, and how factors such as activation functions, network depth, and initial scaling affect spectral learning dynamics. Section \ref{Sec: Preconditioner} proposes our operator-aware preconditioner, elaborating how it modifies the gradient flow to rebalance frequency modes. Section~\ref{Sec:Experiments} presents comprehensive numerical experiments, progressing from simple 1D cases to challenging 2D multiscale problems, demonstrating the effectiveness of our preconditioning approach. Section~\ref{Sec:Conclusion} concludes with a discussion of the implications and potential extensions of this work.

\section{Problem formulation}\label{Sec: Formulation}

We take second-order elliptic equations with Dirichlet boundary conditions as an example. Given a bounded domain $\Omega\subset\mathbb{R}^d$ with sufficiently smooth boundary, we consider the following partial differential equation 
\begin{equation}\label{eq_PDE_problem}
\begin{cases}
    Lu=f&\text{in}~\Omega\\
    u = g&\text{on}~\partial\Omega
\end{cases}\;.
\end{equation}
Here, $L$ is a second-order elliptic differential operator of the form
\begin{equation}\label{eq_elliptic_operator}
Lu = -\sum_{i,j=1}^d \frac{\partial}{\partial x_i}\left(a_{ij}(x)\frac{\partial u}{\partial x_j}\right) + c(x)u,
\end{equation}
that is, there exists a constant $\alpha > 0$ such that 
$$\sum_{i,j=1}^d a_{ij}(x)\xi_i\xi_j \geq \alpha|\xi|^2 \quad \forall x \in \Omega, \xi \in \mathbb{R}^d.$$
The function $u$ is the unknown, and $f\in L^{\infty}(\Omega),g \in L^{\infty}(\partial\Omega)$ are given.

The main idea here is using a NN, denoted by $u_\theta$, to represent the solution and learn the parameters $\theta$ to approximate the PDE and the boundary condition as best as possible. Although using NN representation can potentially offer a few advantages,
such as being mesh-free, easy implementation, scalable to high dimensions, and the key potential
of adaptivity as a nonlinear representation, the parameters have to be learned/trained through a large-scale non-convex
optimization using gradient-based methods in real practice. Moreover, due to the use of nonlinear representation, an appropriate formulation of the optimization problem can be subtle, e.g., whether a strong form or weak form of the elliptic PDE problem~\eqref{eq_PDE_problem} should be used and how to enforce the PDE constraint in the domain interior together with the boundary condition.

%We consider a fully connected neural network (FCNN) with $l$ layers and width $m$, using activation function $\sigma:\mathbb{R} \to \mathbb{R}$ applied element-wise. 

In this work, we will study the simplest and most flexible approach using the least squares formulation to enforce both the PDE and the boundary condition with a balance parameter, termed as the PINN model.  The loss function $\mathcal{L}_{\text{PINN}}(\theta)$ is of the following form
\begin{equation}\label{eq_loss_PINN}
\mathcal{L}_{\text{PINN}}(\theta) = \mathcal{L}_{\text{PINN},r}(\theta)+\lambda_b \mathcal{L}_b(\theta),
\end{equation}
where 
\begin{align}\label{eq_loss_components_PINN}
\mathcal{L}_{ \text{PINN},r}(\theta) &= \frac{1}{2}\int_\Omega |Lu_\theta(x)-f(x)|^2dx=\frac{1}{2}\int_\Omega |L[u_\theta(x)-u(x)]|^2dx,
\\
\mathcal{L}_b(\theta) &= \frac{1}{2}\int_{\partial\Omega}|u_\theta(x)-g(x)|^2ds(x)=\frac{1}{2}\int_{\partial\Omega}|u_\theta(x)-u(x)|^2ds(x).
\end{align}
The loss function is composed of the PDE residual term $\mathcal{L}_{\text{PINN},r}$ and the boundary fitting term $\mathcal{L}_b$, which is equivalent to the square of the $H^2$-norm of the approximation error for $L$ defined in~\eqref{eq_elliptic_operator}. Here $\lambda_b > 0$ is a weighting hyperparameter balancing the two loss components.

% The DRM is inspired by the variational formulation of the equation.  Denote $T:H^1(\Omega)\to L^2(\partial\Omega)$ as the trace operator (See e.g.~\cite{evans2022partial} Section 5.5), then $u\in H_g^1(\Omega)=\{h \in H^1(\Omega)~:~Th = g\}$  is called a \textit{weak solution} for~\eqref{eq_PDE_problem} if it satisfies 
% \begin{equation}\label{eq_PDE_weak}
% 	B[u,v] = \langle f, v\rangle\;, ~\text{for any}~v\in H_0^1(\Omega)\;.
% \end{equation}
% Here the bilinear mapping 
% \begin{equation}\label{eq:var}
% 	B[u,v] :=\sum_{i,j=1}^d\int_{\Omega} a^{ij}(x) \partial_i u\partial_j v\,dx + \int_{\Omega} c(x) uv\,dx 
% \end{equation}
% and $\langle \cdot,\cdot\rangle$ is the pairing of $H^{-1}(\Omega)$ and $H_0^1(\Omega)$. 

% With calculus of variations,  it can be shown that~\eqref{eq_PDE_weak} is equivalent to minimizing
% \begin{equation}
% 	J(u) = \frac{1}{2}B[u,u] -\int_{\Omega} fu\,dx
% \end{equation}
% with the constraint that $u\in H_g^1(\Omega)$. Then the formulation with boundary regularization for DRM is
% \begin{equation}\label{eqn:DRMform}
%     \mathcal{L}_{\text{DRM}}(\theta) =\mathcal{L}_{r, \text{DRM}}(\theta) + \lambda_b \mathcal{L}_b(\theta),
% \end{equation}
% where
% \begin{align}\label{eq_loss_components_DRM}
% % \mathcal{L}_r(\theta) &= \frac{1}{N_r}\sum_{i=1}^{N_r}|Lu_\theta(x_i^r)-f(x_i^r)|^2,\\
% \mathcal{L}_{r, \text{DRM}}(\theta) &=J(u_\theta),\\
% \mathcal{L}_b(\theta) &= \int_{\partial\Omega}(u_\theta(x)-g(x))^2ds(x).
% \end{align}
% Other formulations for enforcing boundary conditions will be discussed in the next section.

As a comparison in our study, the most commonly used least squares formulation for direct function approximation using a NN representation is of the form
\begin{equation}\label{eq_loss_approx}
\mathcal{L}_{\text{approx}}(\theta) = \frac{1}{2}\int_\Omega |u_\theta(x)-u(x)|^2 dx,
\end{equation}
which is the square of the $L^2$ norm of the approximation error.

In practice, these integrals are approximated by discrete sampling. For example, interior points $\{x_i^r\}_{i=1}^{N_r} \subset \Omega$ and boundary points $\{x_i^b\}_{i=1}^{N_b} \subset \partial\Omega$ are drawn randomly from uniform distributions over the domain and its boundary, respectively. The loss is then evaluated at these sampled points to approximate the continuous integral. The optimization is typically performed using gradient-based methods such as Adam~\cite{2015-kingma} or L-BFGS, seeking
\begin{equation}\label{eq_optimization}
\theta^* = \arg\min_{\theta \in \mathbb{R}^p} \mathcal{L}_{\text{F}}(\theta),
\end{equation}
for F$\in \{ \text{PINN}, \text{approx}\}$.

\section{Training Dynamics}\label{Sec: Dynamics}
Once the optimization formulation is determined, the most challenging computational task is the training process for the parameters to reach a satisfactory result. To solve this large-scale non-convex optimization problem with no clear interpretation or special structures in a very high dimensional parameter space, the training process typically relies on gradient-based optimization methods. A key issue is to understand the potential ill-conditioning in the optimization problem, which will lead to both inefficiency and bias for gradient-based methods. 

We start with gradient flow using the continuous formulation. For the function approximation problem~\eqref{eq_loss_approx}, the gradient flow in the parameter space is

\begin{align}
 \frac{d\theta}{dt} = -\int_{\Omega} \frac{\partial u_\theta}{\partial \theta}  ( u_\theta - u)dx,
    \\
    \frac{d\mathcal{L}_{\text{approx}}(\theta)}{dt}=-\left\|\int_{\Omega}  \frac{\partial u_\theta}{\partial \theta}( u_\theta - u) dx\right\|^2.
\end{align}

For function approximation using two layer NNs, it was shown in~\cite{zhang2025shallow} that strong correlations among the set of parametrized activation functions lead to ill-conditioning of the NN representation and result in low frequency basis in training and convergence. For multi-layer NNs, the situation is much more difficult to explicitly characterize and analyze since the answer depends on a more complicated composition structure as well as the choice of activation function and the initial scaling of the parameters. For example, a multi-component and multi-layer neural network structure~\cite{zhang2025structured} using Sine activation function, called FMMNN~\cite{zhang2025fourier}, with proper initial scaling was shown to have strong representation capability and flexibility through smooth composition which also leads to weak frequency bias in training.  

When solving the PDE~\eqref{eq_PDE_problem} using the PINN formulation~\eqref{eq_loss_PINN}, we have
\begin{align*}
    \frac{d\theta}{dt} =-\int_\Omega [L\frac{\partial u_\theta}{\partial \theta}][L  ( u_\theta - u )]dx -\lambda_b\int_{\partial\Omega}\frac{\partial u_\theta}{\partial \theta}  ( u_\theta - u)ds(x),
    \\
    \frac{d\mathcal{L}_{\text{PINN}}}{dt}= -\left\Vert \int_\Omega [L\frac{\partial u_\theta}{\partial \theta}] [L  ( u_\theta - u )]dx +\lambda_b\int_{\partial\Omega}\frac{\partial u_\theta}{\partial \theta}  ( u_\theta - u)ds(x)  \right\Vert^2.
\end{align*}
We see that the involvement of a differential operator $L$ in the loss function~\eqref{eq_loss_PINN} changes the frequency weights in the approximation error with a high frequency bias (in terms of $H^2$ norm). As a consequence, it changes the training dynamics as above. However, there is a subtle competition between the typical low frequency bias of a NN representation and the high frequency bias of a differential operator. The intriguing and important question is the balance between the two and the overall conditioning and frequency bias for the training dynamics when solving PDEs using a NN representation. More importantly, how to counter the frequency bias or precondition the formulation is crucial for both training efficiency and resulting accuracy. 

In the recent work~\cite{he2025can}, it was shown that, when solving elliptic PDE with a two-layer Neural network, the low-frequency bias of the NN representation still dominates. As a result,
low-frequency or smooth components of the solution are captured first during the training process, while high-frequency or fine features are more costly and difficult to compute. As a conclusion, due to the low-frequency bias, using two-layer NNs is difficult to achieve the key potential of adaptivity as a nonlinear representation and is not competitive with well-developed linear method representations such as the finite element method. 

In this work, we will focus on the study of using multi-layer NNs to solve elliptic PDEs of the form~\eqref{eq_PDE_problem}. In particular, we show that, for well-designed multi-layer NNs, which have strong representation capability and mild frequency bias, the ill-conditioning and high frequency bias of differential operators will dominate the overall formulation and lead to a nightmare for the gradient-based training process, capturing low frequency and smooth components of the solution very slowly, just as using iterative methods without preconditioning for finite element methods. To address this issue, we propose a simple and effective \textbf{operator-aware preconditioning} strategy that can improve both the efficiency and accuracy notably.

In the following, we first illustrate the behavior of frequency bias in different settings with numerical experiments. For each example, we track the evolution of the coefficients of different frequency modes in the solution error during training iterations. We then analyze the learning dynamics of these frequency modes in the neural network representation $u_\theta$ and examine their dependence on factors such as the activation function, network depth structure, initial scaling, and the nature of the task (e.g., function approximation or PDE solving). In the experiments presented in this section, we focus on the frequency dynamics during the early stage of training (the first 2,000 iterations). The purpose is not to analyze the plateau phase or compare architectures in terms of final accuracy, but rather to illustrate the presence of frequency bias and how frequency bias affects the rate at which the $L^2$ error decreases during training. Furthermore, our experiments are conducted with over-sampled data, meaning that the sample points sufficiently resolve the target function.

\subsection{Frequency bias for function approximation}

In Examples 1 and 2, we examine the effects of various factors, such as activation function, initial scaling, and depth, on frequency bias in function approximation and demonstrate that reduced frequency bias generally correlates with improved convergence and accuracy.

\textbf{Example 1}: The goal of this example is to examine how the choice of activation functions and the initial scaling of the first layer slope parameters affect frequency bias in function approximation.

Following \cite{zhang2025shallow}, we apply an initial first-layer scaling strategy to reduce frequency bias and improve learning efficiency. After standard initialization (e.g., PyTorch's default), the slope weights $w$ in the activation function $\sigma(w(x-b))$ of the first layer are scaled by a factor that is proportional to $(n_1)^{\frac{1}{d}}$, where 
$n_1$ is the first layer network width and $d$ is the input dimension. This scaling provides the largest frequency range supported by the number of neurons in the first layer, capturing diverse frequency modes in the data at the beginning of training. Subsequent layers keep standard initialization to avoid instability due to amplification across layers.

We first consider the standard fully connected neural networks (FCNN) with width 50 and depth 3, corresponding to 5,201 learnable parameters, trained for the approximation of the target function  $$u_1(x) = \sin(2\pi x) + \sin (5\pi x) + \sin(9\pi x)$$ in $L^2$ norm over the interval $[-1,1]$ with 2,000 sample points. The optimization uses the Adam algorithm with a cyclic learning rate schedule, where the learning rate oscillates between a base value of $\eta_{min} =1\times 10^{-5} $
 and a maximum value of $\eta_{max} = 2\times 10^{-3}$. An exponential decay factor $\gamma = 0.999$ is applied to both bounds, gradually reducing the overall learning rate range throughout training.
% The total training time for this experiment is approximately 17 seconds.

To investigate the spectral learning behavior of the network, we study the error in different Fourier modes during training process. 
% Using the pointwise error evaluated on a uniform grid over the domain 
% $[-1,1]$, We compute the discrete Fourier transform of the error using fast Fourier transform and examine the magnitudes of the resulting Fourier coefficients as training progresses.
In this work, a frequency mode labeled by $m$ refers to the Fourier component of the error at frequency 
$m$, measured in cycles per unit length. The analytic components of the form 
$\sin(k\pi x)$ correspond to frequency mode $m=k/2$. thus, the target solution components in the exact solution $u_1$ are associated with frequency modes $1$, $2.5$ and $4.5$, respectively.
The corresponding curve shows the evolution of the magnitude of this Fourier component over training iterations.
\begin{figure}
\centering
\begin{tabular}{ccc}
(a)&(b)&(c)\\
\includegraphics[width=0.28\textwidth]{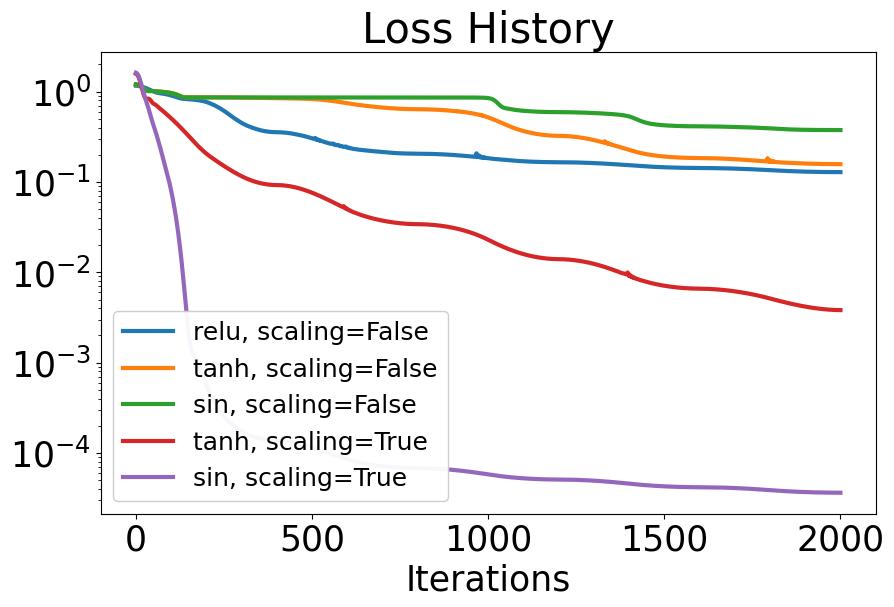}&
\includegraphics[width=0.28\textwidth]{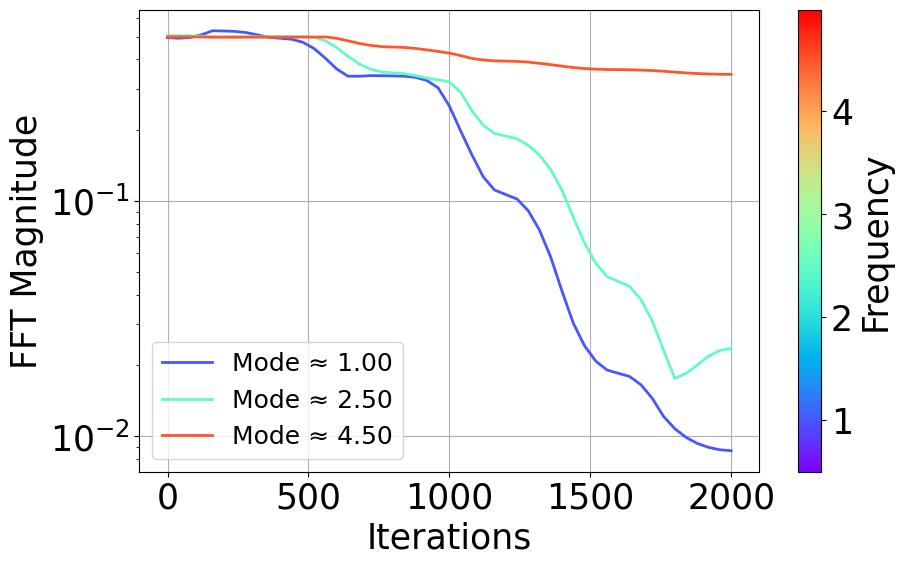}&
\includegraphics[width=0.28\textwidth]{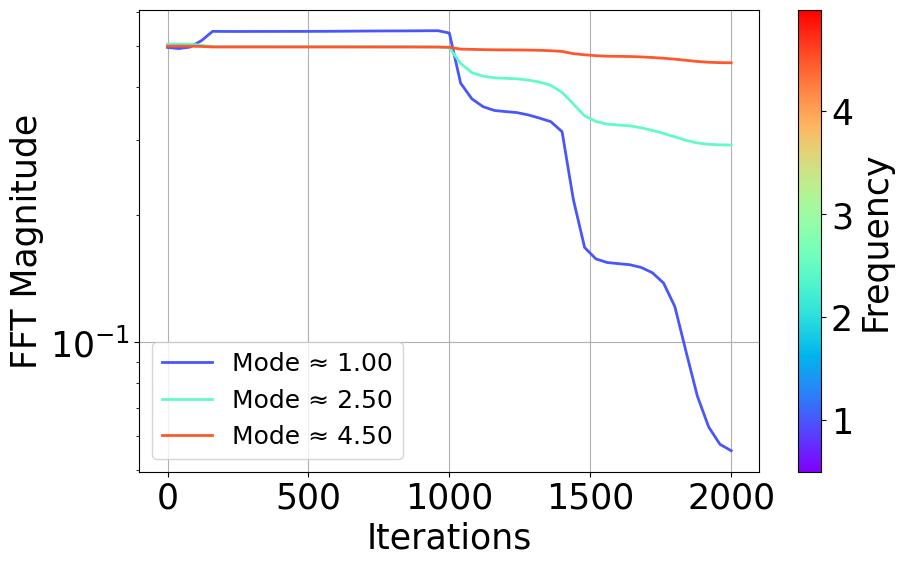}\\
(d)&(e)&(f)\\
\includegraphics[width=0.28\textwidth]{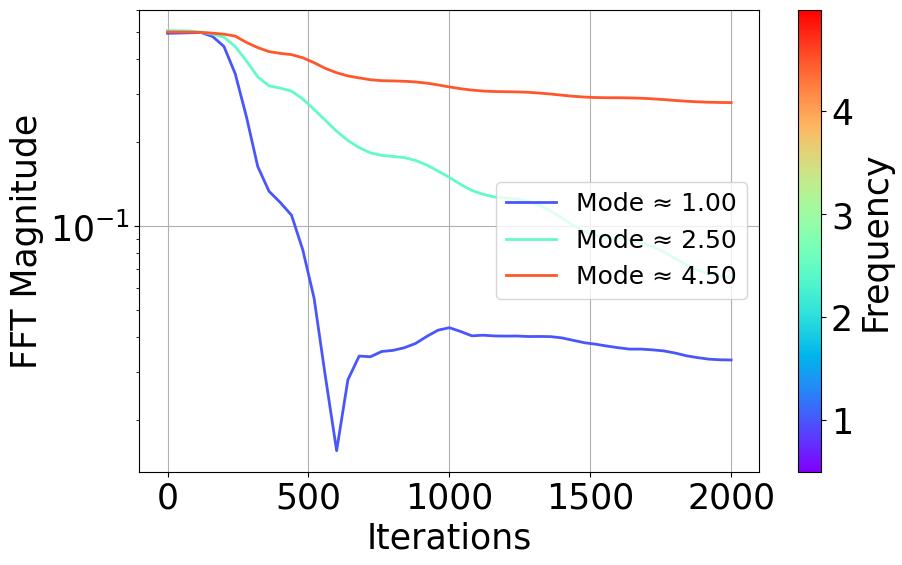}&
\includegraphics[width=0.28\textwidth]{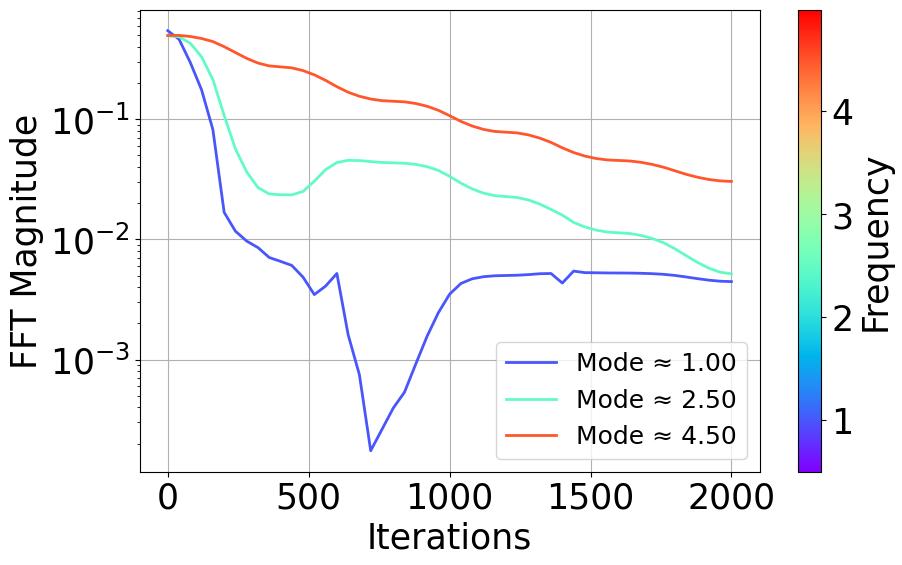}&
\includegraphics[width=0.28\textwidth]{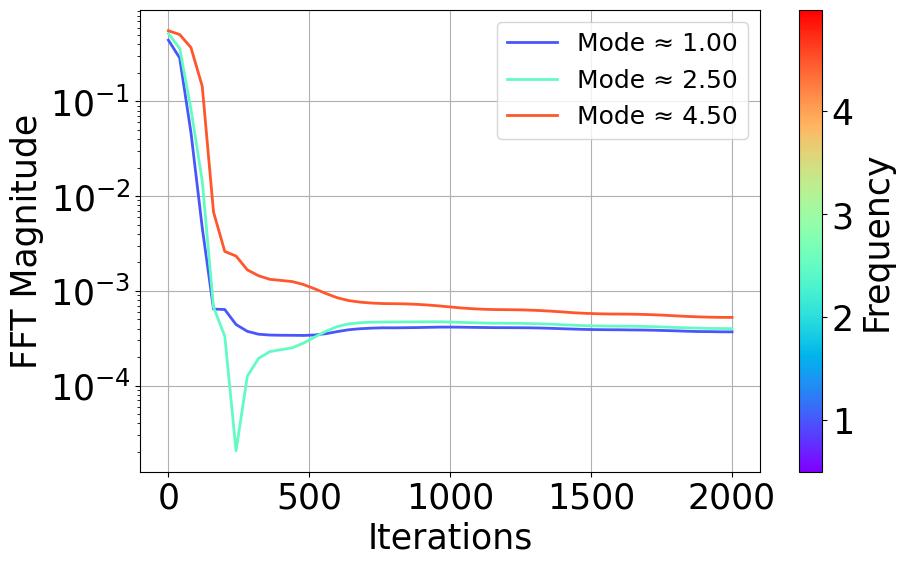}
\end{tabular}
\caption{Example 1: Relative $L^2$ error and frequency bias in function approximation using different activation functions, with and without scaling. (a) Relative $L^2$ error during training; (b)-(f) Errors in different modes during training corresponding to activation functions Tanh, Sine, ReLU,  scaled Tanh, and scaled Sine, respectively.} \label{fig:example1_freq_comparison}
\end{figure}

The results are shown in Fig.~\ref{fig:example1_freq_comparison}. We observe that, without initial scaling, the error decay is slow and Sine or Tanh as activation function does not perform as well as \ReLU as activation function. This can be explained by the spectral analysis in~\cite{zhang2025shallow} for shallow NNs, i.e., the spectral decay of the Gram matrix for a family of randomly parameterized activation functions. Without scaling, the spectral decay is exponential for smooth activation functions while polynomial for \ReLU. However, with a proper scaling of the slope parameters, the family of randomly parameterized activation functions provides a more diverse and less correlated basis whose Gram matrix has a slow spectral decay for a range of frequency modes that are proportional to the scaling and hence can capture more features of the underlying data as shown in~\cite{zhang2025shallow}. In particular, using properly scaled Sine in the first layer is the most efficient, like using a family of diverse random Fourier bases that are the least correlated global functions, as shown in~\cite{zhang2025fourier,tang2025structured} as well as observed in this test. Note that ReLU and its powers are invariant under scaling. 
%coincides with the observation in \cite{tang2025structured} and \cite{zhang2025fourier} that scaling the first layer weights is helpful, especially for {Sine} activation function.

 As for the frequency bias, we have all three NNs without scaling first layer bias towards low frequency; that is, the errors in the lower frequency modes decay faster in the training process. 
With scaling, NN with Sine activation function exhibits little bias towards low frequency. This also leads to a more efficient training process.

\textbf{Example 2}: The second test investigates how network depth influences frequency bias under otherwise similar conditions.

We use the same target function 
$u_1$ and compare three architectures: a shallow (2-layer), a 4-layer, and a 6-layer FCNN, each with a width of 100, corresponding to 301, 20,501, and 40,701 learnable parameters, respectively. For networks using Tanh or Sine activations, the same first-layer scaling is applied. The model is optimized using the Adam algorithm with a cyclic learning rate schedule, where the learning rate oscillates between a base value of $\eta_{min} =1\times 10^{-5} $
 and a maximum value of $\eta_{max} = 1\times 10^{-2}, 5\times 10^{-3}$, and $3\times 10^{-3}$ for 2-layer, 4-layer, and 6-layer NNs respectively with an exponential decay factor $\gamma = 0.999$.
the total training time is approximately 16 seconds for all 2-layer networks, 20 seconds for all 4-layer networks, and 22 seconds for all 6-layer networks.

% \begin{figure}
% \centering
% \begin{tabular}{ccc}
% (a)&(b)&(c)\\
% \includegraphics[width=0.28\textwidth]{figures/example2_selected_freq_Approx_relu_shallow.jpg}&
% \includegraphics[width=0.28\textwidth]{figures/example2_selected_freq_Approx_tanh_shallow.jpg}&
% \includegraphics[width=0.28\textwidth]{figures/example2_selected_freq_Approx_sin_shallow.jpg}\\
% (d)&(e)&(f)\\
% \includegraphics[width=0.28\textwidth]{figures/example2_selected_freq_Approx_relu_deep.jpg}&
% \includegraphics[width=0.28\textwidth]{figures/example2_selected_freq_Approx_tanh_deep.jpg}&
% \includegraphics[width=0.28\textwidth]{figures/example2_selected_freq_Approx_sin_deep.jpg}
% \end{tabular}
% \caption{Example 2: Selected dynamic mode behavior in the error for shallow (top row) and deep (bottom row) models using $\ReLU$, Tanh, and Sine activations.} \label{fig:example2_freq_comparison}
% \end{figure}

\begin{figure}[H]
\centering
\begin{tabular}{ccc}
(a) & (b) & (c)\\
\includegraphics[width=0.28\textwidth]{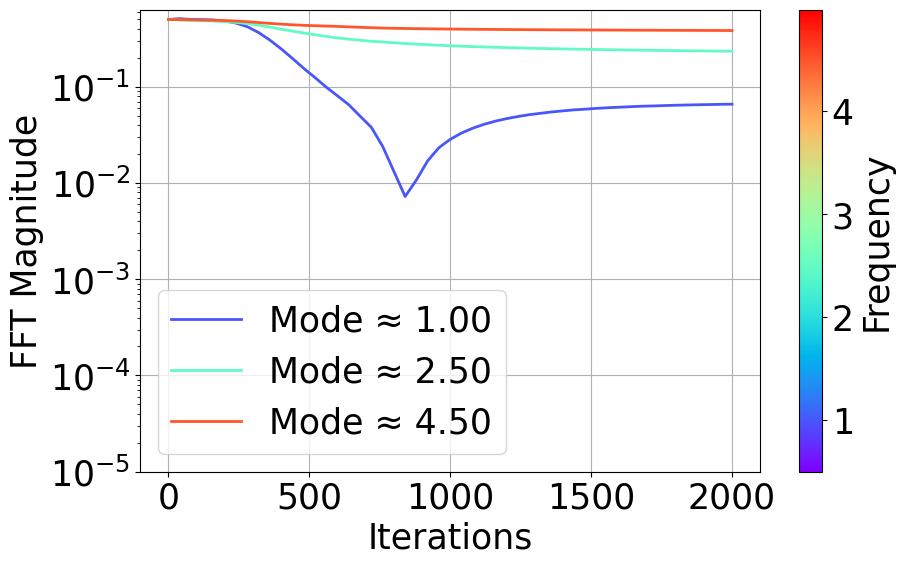} &
\includegraphics[width=0.28\textwidth]{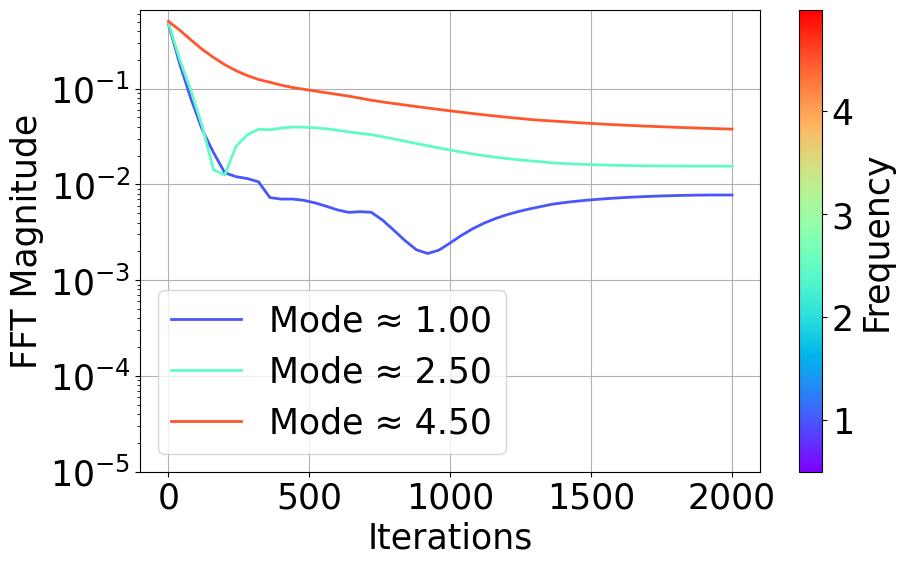} &
\includegraphics[width=0.28\textwidth]{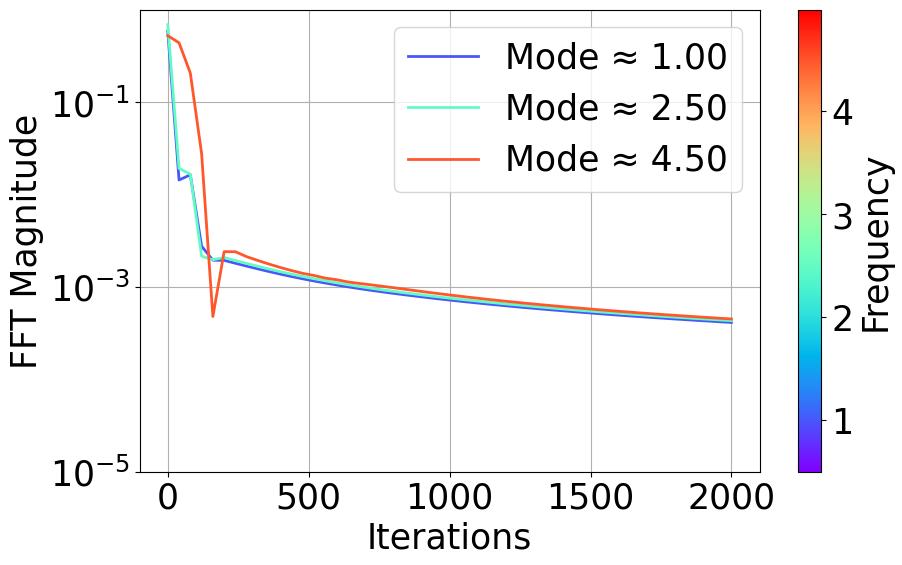}\\
(d) & (e) & (f)\\
\includegraphics[width=0.28\textwidth]{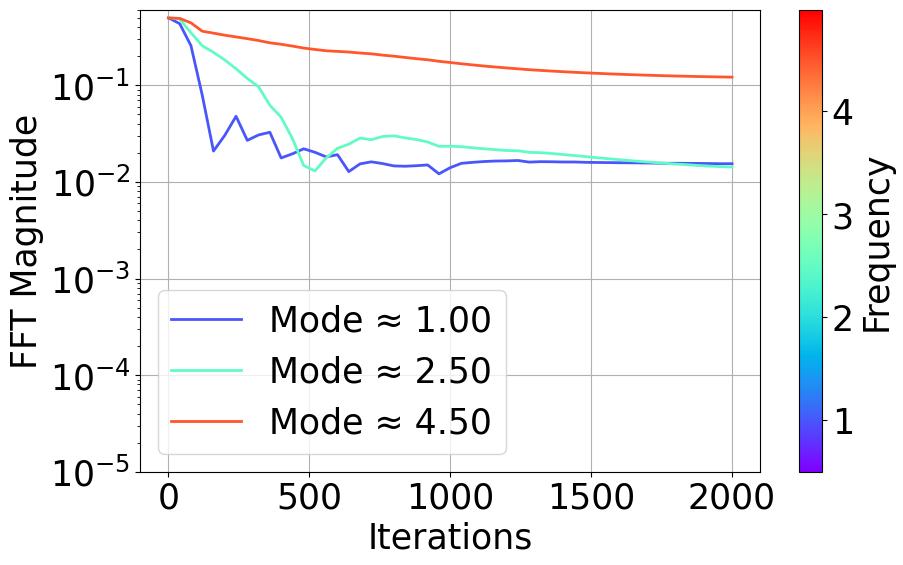} &
\includegraphics[width=0.28\textwidth]{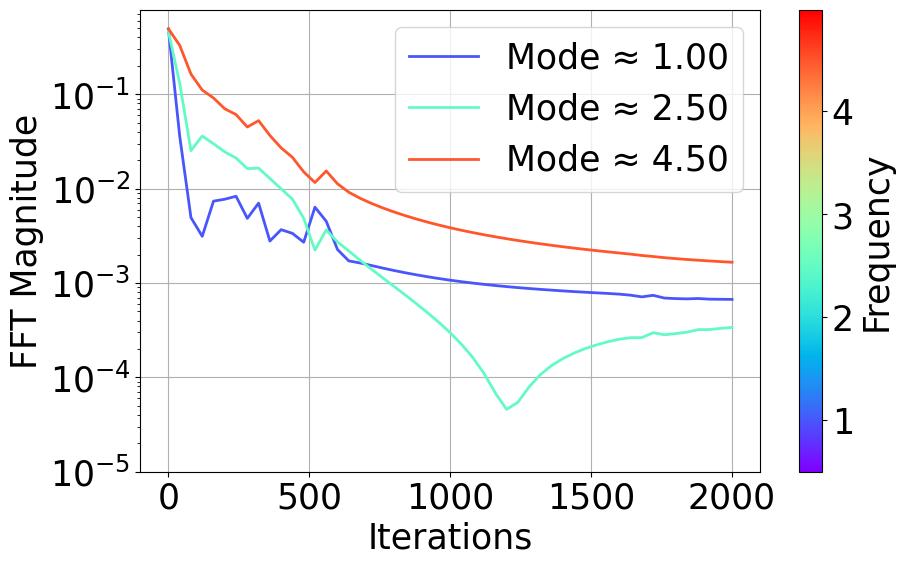} &
\includegraphics[width=0.28\textwidth]{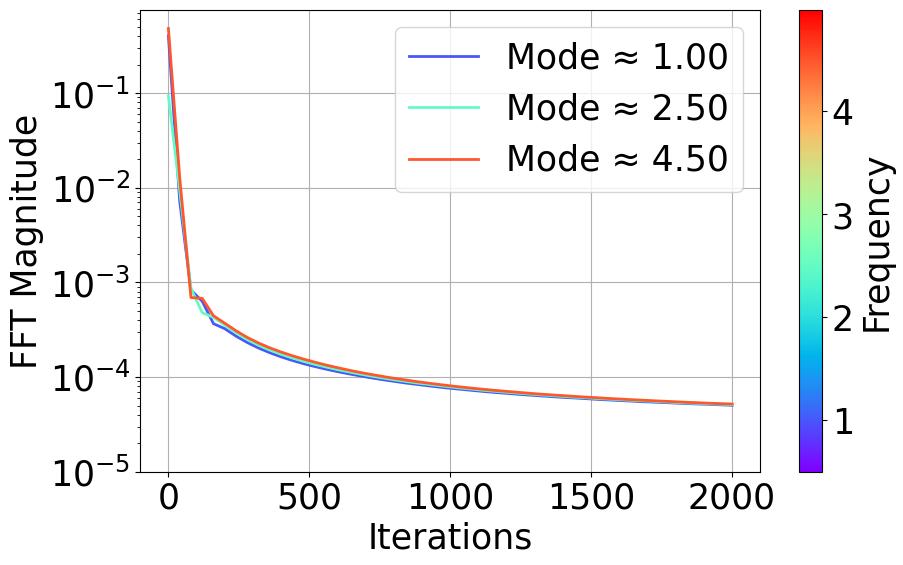}\\
(g) & (h) & (i)\\
\includegraphics[width=0.28\textwidth]{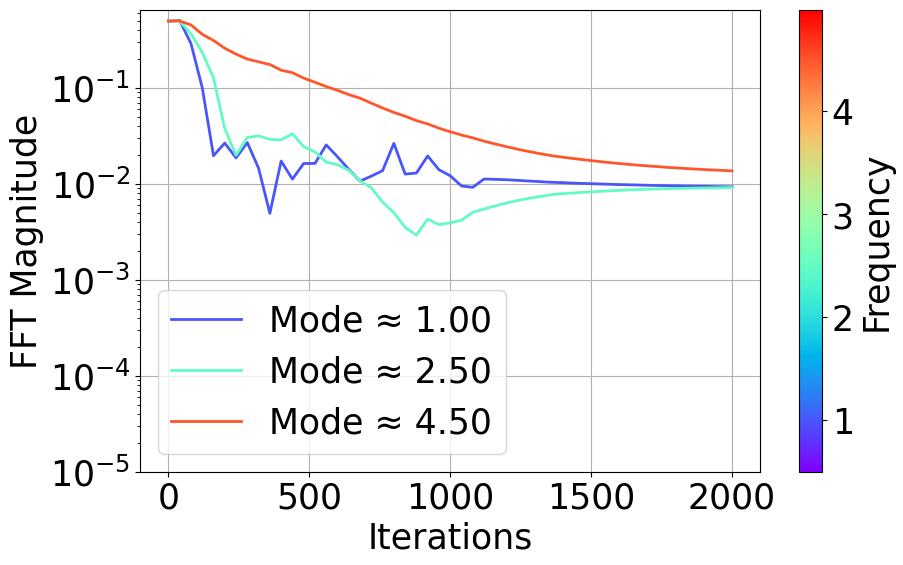} &
\includegraphics[width=0.28\textwidth]{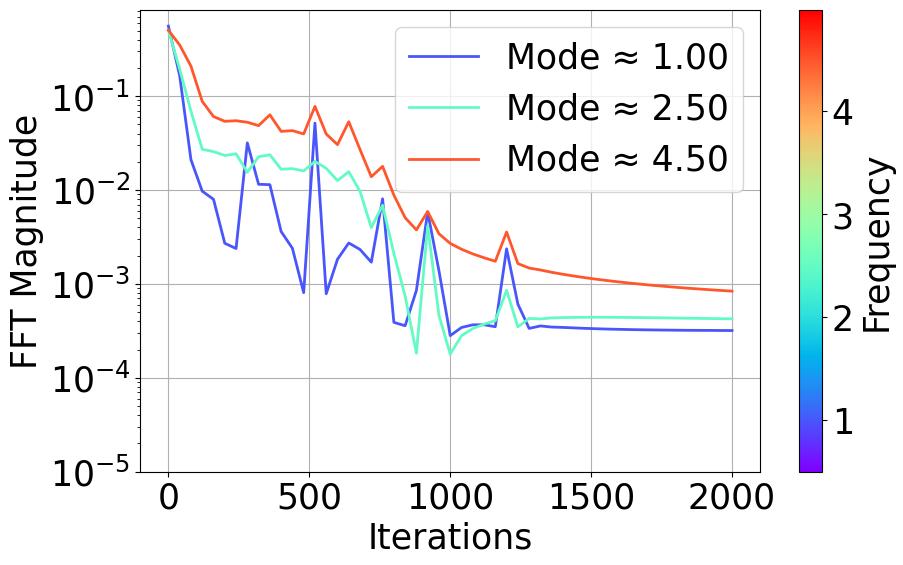} &
\includegraphics[width=0.28\textwidth]{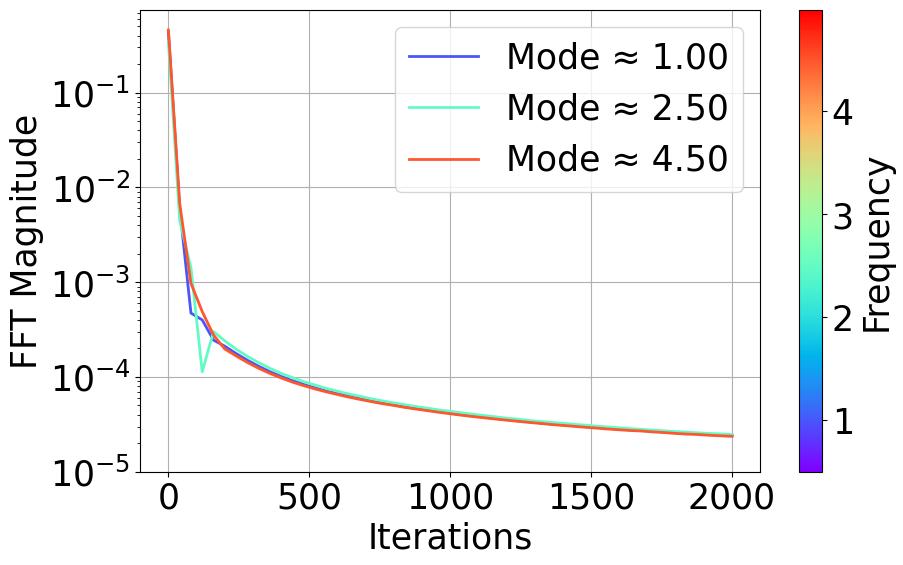}
\end{tabular}
\caption{Example 2: Errors in different modes for a 2-layer (top row), 4-layer (middle row), and 6-layer (bottom row) neural network using $\ReLU$ (left column), Tanh (middle column), and Sine (right column) activation functions.}
\label{fig:example2_freq_comparison}
\end{figure}

As the number of layers increases, the neural networks exhibit reduced frequency bias and improved overall $L^2$ accuracy in the early training stage.
For $\ReLU$ activations (Figs.~\ref{fig:example2_freq_comparison}(a), \ref{fig:example2_freq_comparison}(d), and \ref{fig:example2_freq_comparison}(g)), networks with more layers show a marked reduction in frequency bias. The shallow network displays a strong bias toward low-frequency modes; that is, Mode 1 converges much faster than Mode 4.5. The 4-layer and 6-layer networks achieve more balanced learning, with all modes decaying to around $10^{-2}$ error within 2,000 iterations.
For Tanh activations with scaling (Figs.~\ref{fig:example2_freq_comparison}(b), \ref{fig:example2_freq_comparison}(e), and \ref{fig:example2_freq_comparison}(h)), all models exhibit some frequency bias, but increasing the number of layers improves accuracy by roughly an order of magnitude (reaching $10^{-3}$ vs. $10^{-2}$) and yields more uniform convergence across modes.
For {Sine} with scaling (Figs.~\ref{fig:example2_freq_comparison}(c), \ref{fig:example2_freq_comparison}(f), and \ref{fig:example2_freq_comparison}(i)), the shallow network already shows good performance with minimal frequency bias, but the deep network achieves exceptional results with all modes converging to $10^{-4}$
 error or better. 

 Here is the summary of a few key observations from our function approximation tests above.
 \begin{itemize}
\item Less frequency bias leads to improved efficiency and accuracy of training process.
\item The level of frequency bias depends on NN structure, activation function, and initial scaling. The increase of depth, proper initial scaling, and use of Sine activation function all lead to lesser low frequency bias and improved efficiency for function approximation.
\end{itemize}

 \subsection{Frequency bias in solving PDEs}
We now investigate how frequency bias behaves when NNs are trained to solve PDEs, and how this behavior differs from that observed in direct function approximation. Using the Poisson equation with zero Dirichlet boundary conditions as a test case, we study the overall frequency bias for the learning dynamics in different settings. 

The training dynamics in solving PDE can be significantly different from those observed in function approximation, depending on the balance between the opposite frequency biases of the differential operator and the NN representation. In other words, the presence of the differential operator alters weights for different frequency components in the error and may lead to different frequency bias during training, resulting in distinct patterns of convergence in $L^2$ error. When using two-layer NNs, it was shown~\cite{he2025can} that low frequency bias of NN representation still dominates and results in slow convergence and low accuracy of the training process when the solution contains high frequency components. However, we show below that the situation for multi-layer networks can be quite subtle and different. 
% The more capable and less frequency bias a multi-layer NN is, which depends on the depth, activation function, and initial scaling as shown in the previous section, the more dominant the differential operator and its frequency bias are in the formulation and training process.
The more expressive a multi-layer NN is and the weaker the intrinsic frequency bias, which depends on the depth, activation function, and initial scaling as shown in the previous section, the more dominant the differential operator and its frequency bias are in the formulation and training process.
Nevertheless, our tests show that, although frequency bias may vary depending on the settings, a weaker overall frequency bias still leads to improved convergence and accuracy of the training process.

In the function approximation setting, we previously examined $\ReLU$ as the activation function. For PINN, we need to use at least $\ReLU^3$ as the activation function, which has a second derivative proportional to $\ReLU$. Thus, we compare using $\ReLU^3$ for both solving PDEs and function approximation.
%  \textbf{Example 3(a)}:
% We use the FCNN with width 200 and depth 3, and  4000 points for training. The PDE is $-\Delta u = f$ in $\Omega =(-1,1)$ with $u=0$ on $\partial \Omega$, where $f$ is chosen such that $u_1$ is the exact solution.

% \begin{figure}
% \centering
% \begin{tabular}{ccc}
% (a)&(b)&(c)\\
% \includegraphics[width=0.28\textwidth]{figures/example3_approximation_all_models.jpg}&
% \includegraphics[width=0.28\textwidth]{figures/example3_selected_freq_Approx_relucub.jpg}&
% \includegraphics[width=0.28\textwidth]{figures/example3_selected_freq_Approx_sin.jpg}\\
% (d)&(e)&(f)\\
% \includegraphics[width=0.28\textwidth]{figures/example3_training_L2_loss_all_models.jpg}&
% \includegraphics[width=0.28\textwidth]{figures/example3_selected_freq_PINN_relucub.jpg}&
% \includegraphics[width=0.28\textwidth]{figures/example3_selected_freq_PINN_sin.jpg}
% \end{tabular}
% \caption{Example 3(a): Comparison of training, approximation, and selected frequency responses. (a) NN output plot for different tasks; (b) Approximation $\ReLU^3$; (c) Approximation with Sine activation function;(d)Relative $L^2$ error history; (e) PINN $\ReLU^3$; (f) PINN Sine. }
%     \label{fig:example3_overview}
% \end{figure}

  \textbf{Example 3}: 
We employ 4-layer FCNNs with a width of 200 and 4,000 sample points for this example. The network has 
81,001 learnable parameters. The exact solution $u_2$ to the Poisson equation is
\[
\begin{aligned}
u_2(x) = 1000 \, \chi(x) \big[ &
\sin\!\left(60(x + 0.2)^2\right) e^{-80(x + 0.2)^2} \\
& + 0.8\, \sin\!\left(300x^2\right) e^{-50x^2} \\
& + 0.6\, \sin\!\left(40(x - 0.2)^2\right) e^{-60(x - 0.2)^2} 
\big],
\end{aligned}
\]
where the smooth cutoff function \( \chi(x) \) is defined as:
\[
\chi(x) =
\begin{cases}
\exp\left( -\dfrac{1}{(x + 0.6)^2 (0.6 - x)^2} \right), & \text{if } -0.6 < x < 0.6 \\
0, & \text{otherwise.}
\end{cases}
\]
% For the first 2,000 epochs, the total training time is 25 seconds for direct approximation with 
% $\ReLU$ activation, 22 seconds for direct approximation with Sine 
% % \ROY{The format is not consistent \texttt{Sine} versus Sine} 
% activation, 43 seconds for PINN with 
% $\ReLU^3$, and 38 seconds for PINN with Sine activation. This solution contains significantly higher frequency components compared to $u_1$, making it a more challenging test for approximation and PDE solving tasks.

\begin{figure}
\centering
\begin{tabular}{ccc}
(a)&(b)&(c)\\
\includegraphics[width=0.28\textwidth]{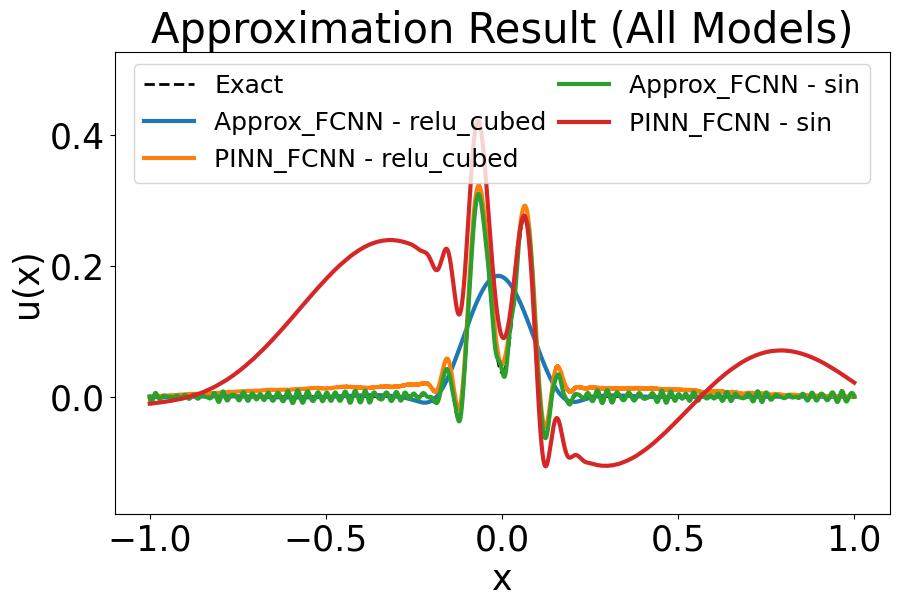}&
\includegraphics[width=0.28\textwidth]{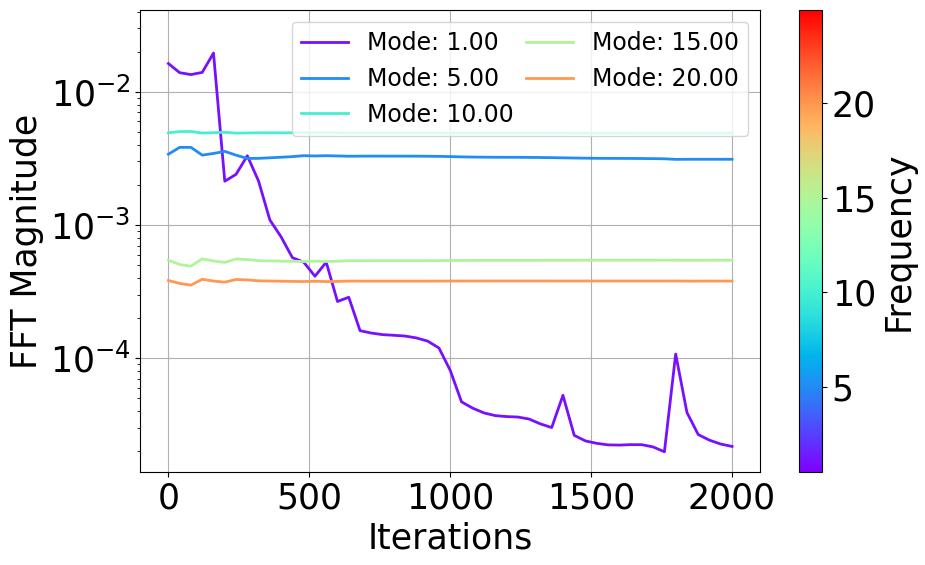}&
\includegraphics[width=0.28\textwidth]{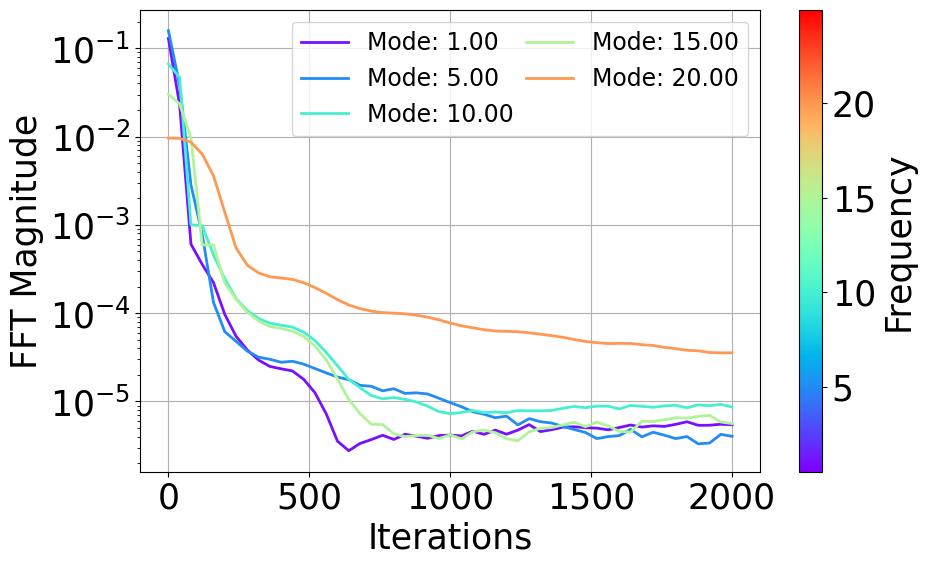}\\
(d)&(e)&(f)\\
\includegraphics[width=0.28\textwidth]{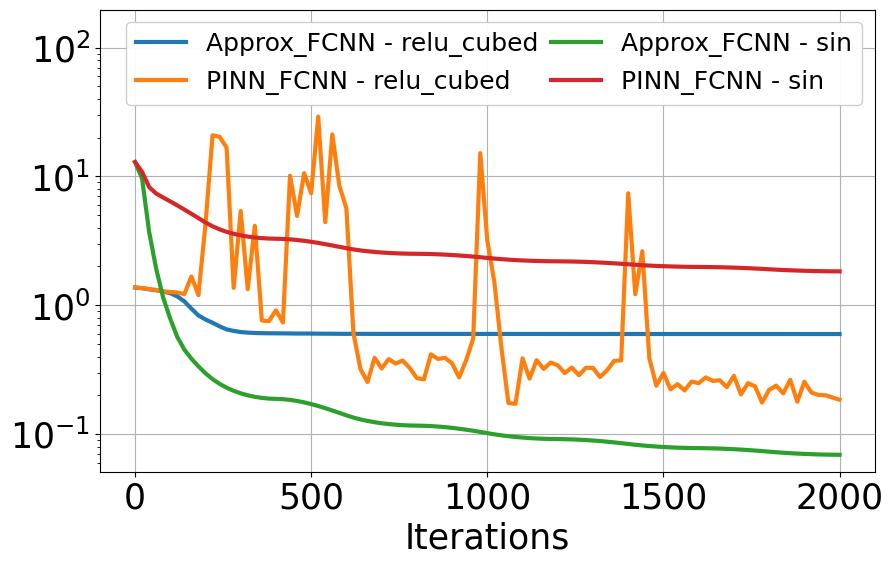}&
\includegraphics[width=0.28\textwidth]{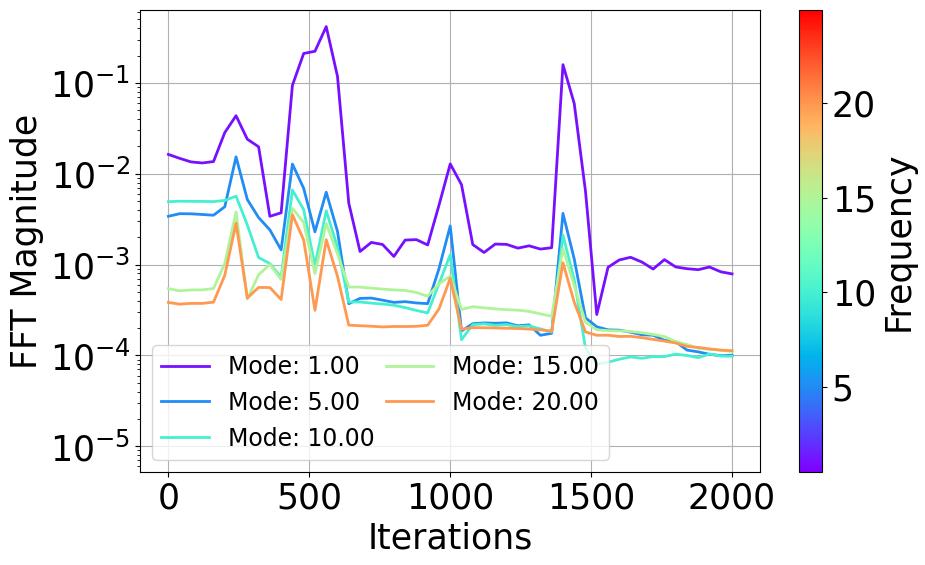}&
\includegraphics[width=0.28\textwidth]{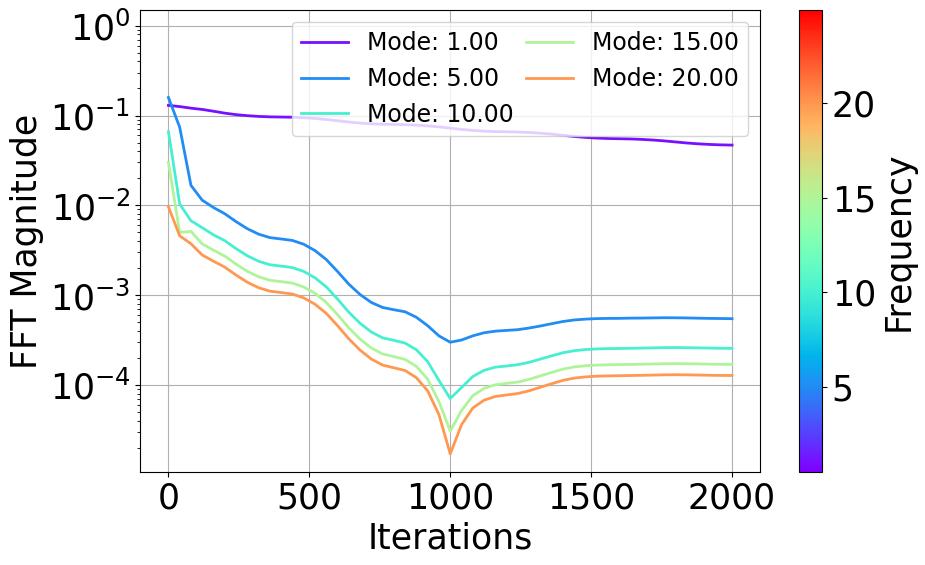}
\end{tabular}
    \caption{Example 3: Comparison of training performance and selected frequency responses.
(a) Neural network outputs at 2,000 epoch for different settings;
(b) Approximation using $\ReLU^3$ activation;
(c) Approximation using Sine activation;
(d) Relative $L^2$ error history;
(e) PINN with $\ReLU^3$ activation;
(f) PINN with Sine activation. }
    \label{fig:example3b_overview}
\end{figure}
 
Comparing approximation using $\ReLU^3$ (Fig.~\ref{fig:example3b_overview}(b)) with PINN using $\ReLU^3$ (Fig.~\ref{fig:example3b_overview}(e)), we see that 1) the NN representation has a pretty strong low frequency bias, and 2) the frequency bias is overturned in PINN due to the overall balance between the low frequency bias of the NN representation and the prevailed high frequency bias of the differential operator. We also observe that the two opposite frequency biases almost offset each other in PINN, which leads to a better overall frequency balance in training and a faster error decay than direct function approximation in $L^2$ norm as shown in Fig.~\ref{fig:example3b_overview}(d). 
%Fig.~\ref{fig:example3b_overview}(e) also indicates the differential operator doesn't just partially cancel the network's intrinsic low-frequency bias, instead, it overcorrects and produces a spectral bias toward high frequencies. This behavior is indicated by the rapid convergence of Mode 20.0 compared with Mode 1.0.
Note that the relative $L^2$ error for PINN exhibits an oscillatory behavior because it is different from the loss function for PINN, which is equivalent to $H^2$ error.
Actually, the overall PINN loss decreases smoothly on average, indicating that this oscillation is not caused by an excessively large learning rate.

With Sine activation function and initial scaling, the multi-layer network is much more capable in representation, which is evidenced by its minimal frequency bias as shown in Fig.~\ref{fig:example3b_overview}(c) and a good approximation result as shown in Fig.~\ref{fig:example3b_overview}(d). However, this means dominance from the differential operator in PINN model, which leads to a pronounced high frequency bias in the training process, as shown in Fig.~\ref{fig:example3b_overview}(f) and results in a slower decay in $L^2$ error compared to the direct approximation as shown in Fig.~\ref{fig:example3b_overview}(d).

Table \ref{tab:freqbias_pde} summarizes these observations.  The numerical experiments indicate that the eventual training performance correlates with the overall frequency bias as a balance of network representation and differential operator. %When the architecture exhibits a strong low-frequency bias in approximation, such as when $\ReLU^3$ is the activation function, the differential operator's high-frequency bias overcorrects, leading to faster high-frequency learning and overall better $L^2$ accuracy. When the architecture has minimal bias (such as with the Sine activation function), the differential operator creates severe high-frequency bias and worse $L^2$ accuracy. 

 Notably, we frequently observe that the PINN loss continues to decrease while the $L^2$ error stagnates or even grows, signaling a misalignment between the optimization objective and the target error metric. We attribute these phenomena to the frequency weighting induced by the differential operator in the PINN loss. PINNs minimize a residual norm (e.g., $\Vert \Delta u_\theta-f\Vert_{L^2(\Omega)}=\Vert \Delta(u_\theta-u)\Vert_{L^2(\Omega)}\approx \Vert u_\theta-u\Vert_{H^2(\Omega)}$), that is not directly aligned with the solution error $\Vert u_\theta-u\Vert_{L^2(\Omega)}$. In the frequency domain, a differential operator 
amplifies higher frequency modes, thereby emphasizing high-frequency residuals. This operator-induced bias can counteract, or even reverse the spectral tendency of the network representation, explaining the divergent behaviors noted above.

\begin{table}
\centering
\caption{Frequency bias in different settings (Example 3).}
\label{tab:freqbias_pde}
\resizebox{0.9\textwidth}{!}{%
\begin{tabular}{c|c|c|c|c}
\hline
 \textbf{Task} & \textbf{Activation} & \textbf{Width/Depth} & \textbf{Frequency Bias} & \textbf{$L^2$ Error after 2K epochs} \\
\hline
\hline
\multicolumn{5}{c}{\textit{Intrinsic Bias from Approximation}} \\
\hline
 Approx & $\ReLU^3$ & 200 / 4 & Low freq bias & 6.00E-01 \\
Approx & Sine (scaled) & 200 / 4 & \textbf{Minimal} & \textbf{6.91E-02} \\
\hline
\multicolumn{5}{c}{\textit{Freq bias in PDE solving (PINN)}} \\
\hline
PINN & $\ReLU^3$ & 200 / 4 & Mild high freq bias & 1.85E-01 \\
PINN & Sine (scaled) & 200 / 4 & High freq bias & 1.83E+00 \\
\hline
\end{tabular}%
}
\end{table}

\section{Operator-aware Preconditioner for PINN}\label{Sec: Preconditioner}
 Based on the insight and observation we gained in previous studies, we propose a preconditioner for PINN that would counteract the frequency bias induced by the differential operator and recover the direct function approximation metric 
 with the help of a fixed integral kernel in the loss function.

Recall that a PINN model is trained by minimizing the loss function 
\begin{equation}\nonumber
\mathcal{L}_{\text{PINN}}(\theta) = \mathcal{L}_{\text{PINN},r}(\theta)+\lambda_b \mathcal{L}_b(\theta).
\end{equation}
We propose to minimize the following loss function instead:
\begin{equation}\label{form:preconditionPINN}
\mathcal{L}_{G}(\theta) = \mathcal{L}_{G,r}(\theta)+\lambda_b \mathcal{L}_b(\theta),
\end{equation}
where 
 \begin{align}
\mathcal{L}_{G,r}(\theta)
&= \frac{1}{2}\int_\Omega \left|\int_\Omega G(x,y)\left[Lu_\theta(y) - f(y)\right]dy \right|^2  dx
\label{eq: loss gr}
\\
&= \frac{1}{2} \int_\Omega\!\left|\int_\Omega G(x,y)L\left[u_\theta(y) - u(y)\right]dy\right|^2  dx
\nonumber
\end{align}
and $G(x,y)$ denotes the Green's function for a reference elliptic operator $L_0$ which is spectrally equivalent to $L$. In practice, we choose $G(x,y)$ to be the Green's function of the Laplacian operator with either homogeneous Dirichlet boundary condition for simple domains or in free-space, for which analytical expressions are available. For the more complex operator $L$ in equation \eqref{eq_PDE_problem}, this choice still provides effective preconditioning while maintaining computational tractability, as will be demonstrated through numerical experiments.

%\subsection{Gradient Flow Dynamics and Learning Operator}
We denote by $\G$ the integral operator  $\G f(x):= \int_{\Omega} G(x,y)f(y)dy$. The corresponding continuous gradient flow becomes
\begin{align*}
    \frac{d\theta}{dt} =-\int_\Omega [\G L\frac{\partial u_\theta}{\partial \theta}][\G L  ( u_\theta - u )]dx -\lambda_b\int_{\partial\Omega}\frac{\partial u_\theta}{\partial \theta}  ( u_\theta - u)ds(x),
    \\
    \frac{d\mathcal{L}_{G}}{dt}= -\left\Vert \int_\Omega [\G L\frac{\partial u_\theta}{\partial \theta}] [\G L  ( u_\theta - u )]dx +\lambda_b\int_{\partial\Omega}\frac{\partial u_\theta}{\partial \theta}  ( u_\theta - u)ds(x)  \right\Vert^2.
\end{align*}

We can regard $\G L$ as the differential operator with preconditioning. In other words, if $\G=L_0^{-1}$ and, $L_0$ and $L$ are spectrally equivalent, $\mathcal{L}_G(\theta)\approx \|u_\theta-u\|_{L^2(\Omega)}$.

\subsection{Implementation of the preconditioned PINN}
In practical implementations, several modifications can be introduced to improve the efficiency and stability of the preconditioned PINN approach.

{\bf Truncated integral domain.}
Instead of performing global integration, we restrict the summation of the Green’s function preconditioner to a local neighborhood around each point.  Specifically, for a point $x\in\Omega$, the integration, that is, the summation in discrete form, is carried out over a ball whose radius equals the distance from $x$
to the boundary $\partial\Omega$. 
The term $\mathcal{L}_{G,r}$ in the preconditioned loss function of discrete form is written as
\begin{equation}\label{eq:loss_discrete}
\mathcal{L}_{G,r}(\theta)
= \frac{1}{2N_r}\sum_{i=1}^{N_r}
\left[
  \frac{1}{N_r}\sum_{j=1}^{M_i}
  G(x_i, y_{ij})
  \big(Lu_\theta(y_{ij}) - f(y_{ij})\big)
\right]^2,
\end{equation}
where $\{x_i\}_{i=1}^{N_r}$ are the interior collocation points and 
$\{y_{ij}\}_{j=1}^{M_i} \subset B(x_i, r_i)$ are sample points within the ball 
$B(x_i, r_i)$ centered at $x_i$ with radius $r_i = \operatorname{dist}(x_i,\partial\Omega)$.

The reasons and benefits of the truncation are as follows.
First, without truncation, allowing $y_{ij}$ to go over the entire domain causes the integration region for points $x$ near the boundary to be partially cut off by the domain boundary. As $x$ approaches the boundary, this truncation becomes increasingly asymmetric, leading to growing deviation from the true Green’s function and consequently to poor preconditioning. Second, this localization reduces computational cost and, most importantly, can still capture the singularity of the Green's function, which determines high-frequency asymptotics of the differential operator and hence maintains a good preconditioning effect. Moreover, the localization is adaptive to the distance to the boundary, which can capture the most influential boundary condition. 
The effect of this truncation is analyzed in detail in Subsection~\ref{subsec:truncation}.

The formulation~\eqref{eq:loss_discrete} is a discrete approximation of the continuous integral operator. In particular, the distribution of $x_i$ and $y_{ij}$ determines the discretization error. In our tests, we choose our collocation points on a grid for both $x_i$ and $ y_{ij}$. So, the spectral property of the kernel can be maintained accurately for frequency modes that can be resolved well by the underlying grid. We remark that, since a NN representation allows evaluation and differentiation at any point, $x_i$ and $y_{ij}$ can be sampled freely. One can potentially design more efficient and accurate adaptive sampling around each $x_i$ accordingly. This meshless representation is more helpful for complicated geometry and in high dimensions. 

{\bf Boundary condition.} It is well-known that one challenge in designing the loss function for the standard PINN model~\eqref{eq_loss_PINN} is the proper balance between the PDE residue term $\mathcal{L}_r$ in the domain interior and the boundary error term $\mathcal{L}_b$. Actually, this balance is solution-dependent. Take the Dirichlet boundary condition as an example, $\mathcal{L}_r$ is the approximation error involving the second derivatives of the solution, while $\mathcal{L}_b$ is the approximation error of the solution on the boundary. In particular, if the solution has multiple scales, it is difficult to choose a constant $\lambda_b$ to achieve a proper balance. 
% The same issue persists for Neumann boundary condition or mixed (Robin type) boundary condition when the boundary term $\mathcal{L}_b$ involves the normal derivative. Again there is a mismatch between the order of derivatives in $\mathcal{L}_r$ and $\mathcal{L}_b$. 
% Below, we show that this mismatch can be solved by the preconditioning strategy (and a careful implementation of the boundary condition in the case of Neumann boundary condition).
For Dirichlet boundary condition, this is straightforward for the preconditioned PINN model~\eqref{form:preconditionPINN}, since $\mathcal{L}_{G,r}(\theta)\approx \|u_\theta-u\|_{L^2(\Omega)}$, if $G=L_0^{-1}$ and $L_0$ is spectrally equivalent to $L$. So as long as $\lambda_b$ balances the sampling density at the boundary and in the interior properly, the two terms are well-matched and $\mathcal{L}_G(\theta)\approx \|u_\theta-u\|_{L^2({\overline{\Omega}})}$, where ${\overline{\Omega}}=\Omega\cup\partial\Omega$. In this work, we set the regularization parameter in the 
preconditioned loss function to $1$ in all numerical experiments. 
We also observe that the performance is not sensitive to the choice 
of this parameter when preconditioning is applied.

{\bf Computation Cost.} Using an integral operator introduces an extra summation operation. This summation can be accelerated using the fast multi-pole method (FMM) \cite{greengard1987fast, ying2004kernel} or a hierarchical hemi-separable (HHS) \cite{chandrasekaran2006fast, chandrasekaran2006fast} sparsification of the kernel in general. However, in practice, this added summation cost is insignificant compared to the computation cost of taking derivatives of a NN representation and the backpropagation operation in each training step. Hence, the computation cost for the preconditioned PINN and the standard PINN is comparable as shown in our numerical tests.  

We present extensive numerical experiments, from 1D examples to 2D multi-scale problems, to demonstrate significant acceleration of convergence and improvement of accuracy when preconditioning is applied. Moreover, the preconditioning framework works on top of any network structure and improves their performance.  For 1D problems, we use FCNNs as the network representation. For more challenging 2D problems, we employ a multi-component and multi-layer neural network (MMNN) architecture introduced in~\cite{zhang2025fourier,zhang2025structured} to represent the solution $u_\theta$. 
While an easy modification to FCNNs through the introduction of balanced multi-component structures in the network - the number of components in each layer, called the rank, is much smaller than the network width, MMNNs achieve a significant reduction of training parameters, a much more efficient training process, and a much improved accuracy compared to FCNNs, especially for functions with complex features.
Such architectures improve scalability in terms of network width and depth for high-dimensional problems, making the preconditioned PINN more practical for large-scale applications. We will showcase in Example 6(a) and 6(b) that for the same PDE, using the MMNN architecture can reach the same level of accuracy with one order of magnitude fewer learnable parameters than FCNN. For all network structures, we use Sine activation function and initial scaling to improve the representation capability and reduce the frequency bias of NN representations. Moreover, a network using Sine as the activation function can be differentiated as many times as needed. The combination of these NN setups and the preconditioning strategy allows us to solve PDEs almost as direct function approximation in terms of efficiency and accuracy.
\subsection{Effective norm of the preconditioned loss term}\label{subsec:truncation}
In this subsection, we analyze the continuous form of the preconditioned loss $\cL_{G, r}$ in~\eqref{eq:loss_discrete}, and show that with the truncated integral domain, this term is weaker than $\| L ( u_{\theta} - u) \|^2_{H^{-1+\eps}(\Omega)}$ for any sufficiently small $\eps > 0$. It essentially means the loss function retrieves at least $(1-\eps)$ derivative back from $L ( u_{\theta} - u)$. Moreover, if we can have a uniform estimate of the decay of the Fourier transform of $L ( u_{\theta} - u)$, it is possible to retrieve $3/2$ derivatives for the loss function.

\begin{lemma}\label{Lem: A estimate}
Let $\Phi_d(r) = \log(r)$ for $d=2$ and $\Phi_d(r) = \sqrt{\frac{2}{\pi}}\frac{1}{r}$ for $d=3$. Define $\cA_d(\rho, \ell)$, $\ell\in[0, 1]$, as follows:
\begin{equation}\nonumber
\begin{aligned}
    \cA_d(\rho, \ell) &:= \rho^{-\frac{d-2}{2}}\int_0^{\ell} \Phi_d(r) J_{\frac{d-2}{2}}(\rho r) r^{\frac{d-2}{2}} r dr \\
    &= \begin{cases}
        \frac{1}{\rho^2}\left(1 - J_0(\rho\ell) - \log(\ell) \rho \ell J_1(\rho\ell)\right),\quad &d = 2,\\
        \frac{1}{\rho^{2}} (1 - \cos(\rho\ell)),\quad &d=3,
    \end{cases}
\end{aligned}
\end{equation}
where $J_{\nu}$ is the Bessel function of the first kind of order $\nu$. Then there is an absolute constant $C_{\cA} > 0$ that 
$$|\cA_d(\rho, \ell)| \le \begin{cases}
C_{\cA} \min (1, \rho^{-3/2}),\quad& d=2\\
C_{\cA} \min (1, \rho^{-2}),\quad& d=3.
\end{cases}$$ 
for all $\rho \ge 0$.
\end{lemma}
\begin{proof}
    When $d=2$, for $x\in [0, 1]$, $J_{0}(x) = 1 + \cO(x^2)$ and $J_1(x) = \cO(x)$. For $x > 1$, we use the asymptotic approximation for $k=0, 1$ in~\cite{krasikov2014approximations} that 
    \begin{equation*}
        J_k(x) = \sqrt{\frac{2}{\pi x}}\cos(x - \frac{k \pi}{2} - \frac{\pi}{4}) + \cO\left(\frac{|k^2-\frac{1}{4}|}{x^{3/2}}\right),
    \end{equation*}
    which concludes the desired estimate by letting $x = \rho\ell$. Notice that $\ell^2 |\log \ell| \le \ell|\log\ell|$ is uniformly bounded on $\ell\in[0,1]$. 
    
    When $d=3$, we have $1-\cos(\rho\ell) = \cO((\rho\ell)^2)$ when $\rho\ell \in [0, 1]$ and if $\rho\ell > 1$, we use $|1-\cos(\rho\ell)| = \cO(1)$.
\end{proof}

\begin{theorem}
Let $\Omega = B(0,1)$ be the unit disc in $\bbR^d$, $d=2, 3$, 
and $\ell(x) = \operatorname{dist}(x, \partial\Omega) = 1- |x|$. Denote $\chi$ be the characteristic function in the unit disc and $$G_d(x) = \begin{cases}
    -\frac{\log |x|}{2\pi},\quad &d=2, \\
     -\frac{1}{4\pi|x|},\quad &d= 3, 
\end{cases}$$
be the Green's function of the Laplacian in free space. Define 
\begin{equation*}
    \cT f(x) := \int_{\bbR^2}\chi\left(\frac{x - y}{\ell(x)}\right) G(x - y) f(y) dy.
\end{equation*}
\begin{enumerate}
    \item Let $f\in H^s(\Omega)$, then for any $\eps > 0$, there exists a constant $C_{\eps}$ that 
\begin{equation*}
    \|\cT (I - \Delta)^{s} f\|_{L^2(\Omega)} \le C_{\eps} \|f\|_{H^{s-(1-\eps)}(\Omega)},
\end{equation*}
\item Let $\widehat{f}(\xi)$ be the Fourier transform of $f$. Suppose 
$$\widehat{f}(\xi) = 
\begin{cases}
  \displaystyle \sum_{k=-\infty}^{\infty} c_k(|\xi|) e^{i k \arg(\xi)},\quad &d=2, \\
 \displaystyle  \sum_{k=0}^{\infty}\sum_{l=-k}^k c_{k,l}(|\xi|) Y_{kl}\left(\frac{\xi}{|\xi|}\right),\quad &d=3,
\end{cases}
$$ satisfies that
\begin{equation*}
\begin{aligned}
 |c_{\alpha}(|\xi|)| &\le D \gamma(\alpha) (1 + |\xi|^2)^{-\frac{s}{2} -\delta - \frac{d-2}{2}}, \\
 |c_{\alpha}(|\xi|)| &\le \gamma(\alpha) \left(\sum_{\beta} |c_{\beta}(|\xi|)|^2 \right)^{1/2},
\end{aligned}   
\end{equation*}
where $\alpha, \beta$ represent the multi-subscripts in $d=2$ or $d=3$, $Y_{kl}$ is the spherical harmonics,  $D$ is a constant, $\delta > 0$, and $\sum_{\beta} |\gamma(\beta)|^2 < \infty$. Then there exists $p\in (2,\infty)$ and a constant $C' > 0$ that
\begin{equation*}
     \|\cT (I - \Delta)^{s} f\|_{L^2(\Omega)} \le C' \|f\|_{B,s,p}^{(p-2)/p} \|f\|_{H^{s-3/2}(\Omega)}^{2/p},
\end{equation*}
where the norm
\begin{equation*}
    \|f\|_{B,s,p} := \left( \int_0^{\infty} \left|(1+|\xi|^2)^{s-(2-\frac{d}{2})} \sum_{\alpha}|c_{\alpha}(\xi)|^2\right|^{\frac{p-2}{2(p-1)}}  d|\xi| \right)^{\frac{p-1}{p-2}}.
\end{equation*}
\end{enumerate}

\end{theorem}
\begin{proof}
Let $\phi\in L^2(\Omega)$, then 
\begin{equation*}
\begin{aligned}
   & \int_{\Omega} \left| \int_{\bbR^2}\chi\left(\frac{x - y}{\ell(x)}\right) G_d(x - y)  \phi(y) dy \right|^2 dx \\
    &=\int_{\Omega} \left| \int_{\bbR^2}\chi\left(\frac{x - y}{\ell(x)}\right) G_d(x - y) \int_{\bbR^2} \widehat{\phi}(\xi) e^{i \pi \xi\cdot y} d y \right|^2 dx \\ 
    &=\int_{\Omega} \left| \int_{\bbR^2} \widehat{\phi}(\xi) \left( \int_{\bbR^2} \chi\left(\frac{x - y}{\ell(x)}\right) G_d(x - y)    e^{i \xi\cdot y} dy \right) d\xi \right|^2 dx \\
    &= \int_{\Omega} \left| \int_{\bbR^2} \widehat{\phi}(\xi) e^{i\xi \cdot x} \cA_d(|\xi|, \ell(x)) d\xi \right|^2 dx\\
    &= \int_{\bbR^2} \int_{\bbR^2} \cH_d(\xi, \zeta) \widehat{\phi}(\xi) \overline{\widehat{\phi}(\zeta)} d\xi d\zeta.
\end{aligned}
\end{equation*}
The kernel function
\begin{equation}\label{EQ: H Kernel}
    \begin{aligned}
\cH_d(\xi, \zeta) &= \dint_{\Omega} \cA_d(|\xi|, \ell(x)) \cA_d(|\zeta|, \ell(x)) e^{i(\xi - \zeta)\cdot x} dx  \\
& = \begin{cases}
2\pi  \dint_0^1   \cA_d(|\xi|, 1-r) \cA_d(|\zeta|, 1-r) J_0(|\xi - \zeta| r) rd r, \quad &d =2 ,\\
(2\pi)^{3/2} \dint_0^1   \cA_d(|\xi|, 1-r) \cA_d(|\zeta|, 1-r) \frac{J_{1/2}(|\zeta - \xi| r)}{\sqrt{|\zeta - \xi| r}}  r^2d r, \quad &d=3,
\end{cases}
    \end{aligned}
\end{equation}
where $J_{\nu}$ denotes the Bessel function of the first kind at order $\nu$. Next, we reformulate~\eqref{EQ: H Kernel} by Graf's and  Gegenbauer's addition formula 
\begin{equation*}
    \begin{aligned}
J_0(|\zeta - \xi|r) &= \sum_{k=-\infty}^{\infty} J_k(|\xi|r) J_k(|\zeta|r) C_{k}^{(0)}(\cos\theta),\\
\frac{J_{1/2}(|\zeta - \xi| r)}{\sqrt{|\zeta - \xi| r}} &= \sqrt{\frac{\pi}{2}}\sum_{k=0}^{\infty} (2k + 1) J_{k+1/2}(|\xi|r) J_{k+{1/2}}(|\zeta| r) C^{(1/2)}_{k}(\cos\theta),
    \end{aligned}
\end{equation*}
 where $\theta$ is the angle between $\xi$ and $\zeta$, $C^{(0)}_{k}$ and $C^{(1/2)}_{k}$ are Gegenbauer polynomials.
Then we obtain a decomposition for $d=2,3$:
\begin{equation}\nonumber
    \begin{aligned}
        \cH_d(\xi,\zeta) &= 
        \begin{cases}
       \displaystyle 2\pi \sum_{k=-\infty}^{\infty} \cH_{d, k}(|\xi|, |\zeta|) C_{k}^{(0)}(\cos\theta), \quad &d=2,\\
         \displaystyle 2\pi^{2} \sum_{k=0}^{\infty} (2k+1) \cH_{d, k}(|\xi|, |\zeta|) C_{k}^{(\frac{1}{2})}(\cos\theta), \quad &d = 3,
        \end{cases}
        \\
         \cH_{d, k}(|\xi|, |\zeta|) &= \begin{cases}
             \dint_0^1 \cA_d(|\xi|, 1-r) \cA_d(|\zeta|, 1-r)  J_k(|\xi|r) J_k(|\zeta|r) rd r,\quad &d=2,\\
     \dint_0^1 \cA_d(|\xi|, 1-r) \cA_d(|\zeta|, 1-r)  J_{k+1/2}(|\xi|r) J_{k+1/2}(|\zeta|r) r^2d r, \quad &d=3.
         \end{cases}
    \end{aligned}
\end{equation}
By Lemma~\ref{Lem: A estimate}, $\cH_{d, k}(|\xi|, |\zeta|)$ is uniformly bounded for all $\xi, \zeta$ and $k$. 
For $|\nu|\le \frac{1}{2}$, it is known that $|J_{\nu}(\rho r)|\le \sqrt{\frac{2}{\pi \rho r}}$~\cite[Theorem 7.31.2]{szeg1939orthogonal}. Hence, there exists an absolute constant $C$ that $|\cH_{d, 0}(|\xi|, |\zeta|)| \le \frac{C}{|\xi|^{(d+2)/2} |\zeta|^{(d+2)/2}}$ for $d=2, 3$.
For $|\nu|\ge 1$, we use the estimate from~\cite{krasikov2014approximations}, $|J_{\nu}(\rho r)| \le \sqrt{\frac{2}{\pi}} |\rho^2 r^2 - \omega_{\nu}^2 |^{-1/4}$, $\omega_{\nu}^2 = \nu^2 - \frac{1}{4}$. Then, take $\omega = \omega_{k}$ or $\omega = \omega_{k+1/2}$, we have the estimate
\begin{equation*}
\begin{aligned}
   |  \cH_{d, k}(|\xi|, |\zeta|) | &\le \frac{ C_{\cA}^2 }{|\xi|^{(d+2)/2}|\zeta|^{(d+2)/2}} \int_0^1 \frac{ r  }{ | r^2 - (\frac{\omega}{|\xi|})^2 |^{1/4}  | r^2 - ( \frac{\omega}{|\zeta|})^2 |^{1/4}} dr\\
   &\le \frac{ C_{\cA}^2 }{|\xi|^{(d+2)/2}|\zeta|^{(d+2)/2}} \int_0^1 \frac{ \sqrt{r}  }{ | r - \frac{\omega}{|\xi|} |^{1/4}  | r -  \frac{\omega}{|\zeta|} |^{1/4}} \frac{\sqrt{r} }{ | r + \frac{\omega}{|\xi|} |^{1/4}  | r +\frac{\omega}{|\zeta|} |^{1/4}} dr \\
   &\le \frac{ C_{\cA}^2 }{|\xi|^{(d+2)/2}|\zeta|^{(d+2)/2}} \int_0^1 \frac{ \sqrt{r}  }{ | r - \frac{\omega}{|\xi|} |^{1/4}  | r -  \frac{\omega}{|\zeta|} |^{1/4}}  dr \\
   &\le \frac{ C }{|\xi|^{(d+2)/2}|\zeta|^{(d+2)/2}} . 
\end{aligned}
\end{equation*}
Therefore, combined with the uniform boundedness of $\cH_{d,k}$, we have the bound
\begin{equation*}
    |  \cH_{d, k}(|\xi|, |\zeta|) | \le \frac{C'}{(1+|\xi|^2)^{(d+2)/4}(1 + |\zeta|^2)^{(d+2)/4}}
\end{equation*}
for some absolute constant $C' > 0$. For any $f\in H^s(\Omega)$, we take $\phi = (I - \Delta)^{s/2} f \in L^2(\Omega)$, then 
\begin{equation*}
\begin{aligned}
\widehat{\phi}(\xi) = 
\begin{cases}
\displaystyle \sum_{k=-\infty}^{\infty} (1+|\xi|^2)^{s/2} c_{k}(|\xi|) e^{ik\arg(\xi)},\quad &d = 2,\\
\displaystyle    \sum_{k=0}^{\infty}\sum_{l=-k}^k (1+|\xi|^2)^{s/2} c_{k, l}(|\xi|) Y_{kl}\left(\frac{\xi}{|\xi|}\right),\quad &d=3,
\end{cases}
\end{aligned}    
\end{equation*}
where $Y_{kl}$ denotes the spherical harmonics on $\bbS^{d-1}$. Then, denote $\Theta(|\xi|, |\zeta|) = (1+|\xi|^2) (1+|\zeta|^2)$, use Cauchy-Schwartz inequality,  
\begin{equation*}
\begin{aligned}
    &\int_{\bbR^d} \int_{\bbR^d} \cH_d(\xi, \zeta) \widehat{\phi}(\xi) \overline{\widehat{\phi}(\zeta)} d\xi d\zeta \\
    &\asymp \begin{cases}
    \displaystyle \sum_{k=-\infty}^{\infty}  \iint_{\bbR^2_{+}} \cH_{d, k}(|\xi|, |\zeta|) \Theta(|\xi|, |\zeta|)^{s/2} c_{k}(|\xi|)\overline{c_{k}(|\zeta|)} |\zeta|^{d-1}|\xi|^{d-1} d|\zeta| d|\xi|, \quad & d=2\\
    \displaystyle \sum_{k=0}^{\infty}\sum_{l=-k}^k \iint_{\bbR^2_{+}} \cH_{d, k}(|\xi|, |\zeta|) \Theta(|\xi|, |\zeta|)^{s/2} c_{k,l}(|\xi|)\overline{c_{k,l}(|\zeta|)} |\zeta|^{d-1}|\xi|^{d-1} d|\zeta| d|\xi|,\quad & d=3
    \end{cases}\\
    &\le C' \sum_{\alpha} \left(   \int_0^{\infty}  (1+|\xi|^2)^{\frac{s}{2}-\frac{d+2}{4}} |c_{\alpha}(|\xi|)| |\xi|^{d-1} d|\xi| \right)^2\\
    &\le  C' \sum_{\alpha} \left( \int_0^{\infty} \frac{1}{(1+|\xi|^2)^{\frac{d}{2}+\eps}} |\xi|^{d-1} d|\xi|\right) \left(\int_0^{\infty}  (1+|\xi|^2)^{s -1 +\eps} |c_{\alpha}(|\xi|)|^2 |\xi|^{d-1}d|\xi|\right) \\& \asymp \|f\|_{H^{s-(1-\eps)}}^2,
\end{aligned}
\end{equation*}
where we use $\alpha$ to represent all possible subscripts in the summation for both $d=2$ and $d=3$.
For the second part, we adopt the H\"older inequality instead of the Cauchy-Schwartz inequality in the last step. Using the controlled decay rates of $|c_k(|\xi|)|$ in the theorem, we obtain
\begin{equation}\nonumber
\begin{aligned}
& \sum_{\alpha} \left(   \int_0^{\infty}  (1+|\xi|^2)^{\frac{s}{2} - \frac{d+2}{4}} |c_{\alpha}(|\xi|)| |\xi|^{d-1} d|\xi| \right)^2 \\
    &\le\sum_{\alpha} \left( \int_0^{\infty} \left|(1+|\xi|^2)^\frac{s-1}{2}|c_{\alpha}(\xi)|\right|^{\frac{p-2}{p-1}}  d|\xi|\right)^{2(p-1)/p} \left(\int_0^{\infty}  (1+|\xi|^2)^{s-3/2} |c_{\alpha}(\xi)|^2 |\xi|^{d-1}d|\xi|\right)^{2/p}\\
    &\le C'' \left( \int_0^{\infty} \left|(1+|\xi|^2)^{s-(2-\frac{d}{2})}\sum_{\alpha}|c_{\alpha}(|\xi|)|^2 \right|^{\frac{p-2}{2(p-1)}}  d|\xi|\right)^{2(p-1)/p} \|f\|_{H^{s-3/2}(\Omega)}^{4/p}\\
    &= C''  \|f\|_{B,s, p}^{2(p-2)/p} \|f\|_{H^{s-3/2}(\Omega)}^{4/p},
\end{aligned}
\end{equation}
where $C'' =  \sum_{k=-\infty}^{\infty} \gamma_k^2 <\infty$, and when $p$ is sufficiently large, $\|f\|_{B,s, p}$ is finite.
\end{proof}

\section{Numerical Experiments}
\label{Sec:Experiments}
In this section, we investigate the performance of the preconditioned PINN under a range of problem settings to assess the effectiveness of the proposed preconditioning strategy. Examples 4–5 consider one-dimensional Poisson equations. Examples 6(a)–(b) examine a one-dimensional elliptic equation with variable coefficients. Examples 7 and 8 study two-dimensional Poisson equations with homogeneous and nonhomogeneous Dirichlet boundary conditions, respectively, while Example 9 addresses a two-dimensional elliptic equation with variable coefficients.  Example 10 examines  the three-dimensional Poisson equations.  For all examples, we use full batch training and the activation function is chosen to be Sine, with the weights of the first layer scaled proportional to $(n_1)^{\frac{1}{d}}$, where 
$n_1$ is the first layer network width and $d$ is the input dimension.

When implementing the preconditioned loss formulation \eqref{form:preconditionPINN}, for all one-dimensional examples in the domain $\Omega = (-1,1)$, we compute $\mathcal{L}_{G,r}$ of the form \eqref{eq: loss gr} by choosing $G$ as the Green's function for the Laplacian with homogeneous Dirichlet boundary conditions, which has the analytical expression:
$$G(x,y) = \begin{cases}
\frac{1}{2}(1-y)(1+x), & \text{if } x \leq y \\
\frac{1}{2}(1-x)(1+y), & \text{if } x > y.
\end{cases}$$
 For the two-dimensional ($d=2$) and three-dimensional ($d=3$) problems in domains of different shapes, we implement the truncated integral form in \eqref{eq:loss_discrete} and select the function $G$ in the preconditioned formulation to be the Green’s function of the Poisson equation in free space, given by
$$G(x,y) =\begin{cases} -\frac{1}{2\pi} \log(|x-y|),\quad &d=2,\\
  -\frac{1}{4\pi|x-y|},\quad &d= 3.
\end{cases}$$
% $$G(x) = \begin{cases}
%     -\frac{\log |x|}{2\pi},\quad &d=2, \\
%      -\frac{1}{4\pi|x|},\quad &d= 3, 
% \end{cases}$$

% \textcolor{blue}{Specifically, the scaling factor is set to $n_1/4$ for Examples 4–6 (where $n_1$ denotes the width of the first hidden layer), while a reduced factor of $n_1/10$ is utilized for Example 7-9.}

In examples below, we implement the preconditioned PINN to solve PDEs, compare with PINN without preconditioning and function approximation, plot the $L^2$ errors, and present the dynamics of selected frequency modes in errors.

\textbf{Example 4:} We solve the Poisson equation with homogeneous Dirichlet boundary conditions, where the force term $f$ is chosen such that  $u_1(x) = \sin(2\pi x) + \sin (5\pi x) + \sin(9\pi x)$ is the exact solution. An FCNN with width 50 and depth 3, corresponding to 5,251 learnable parameters, is used to represent the solution.  The model is trained using 2,000 collocation points.
\begin{figure}[!htb]
\centering
\begin{tabular}{cc}
(a)&(b)\\
\includegraphics[width=0.45\textwidth]{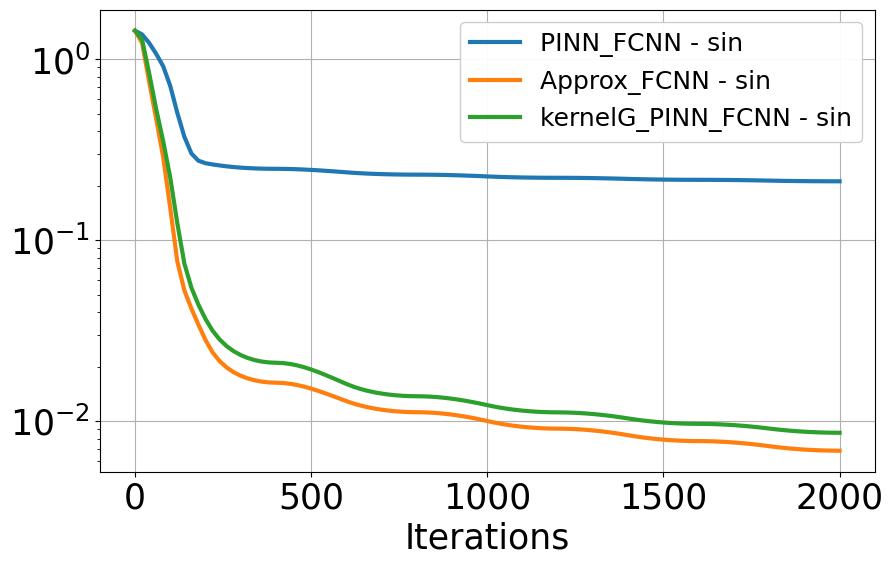}&
\includegraphics[width=0.45\textwidth]{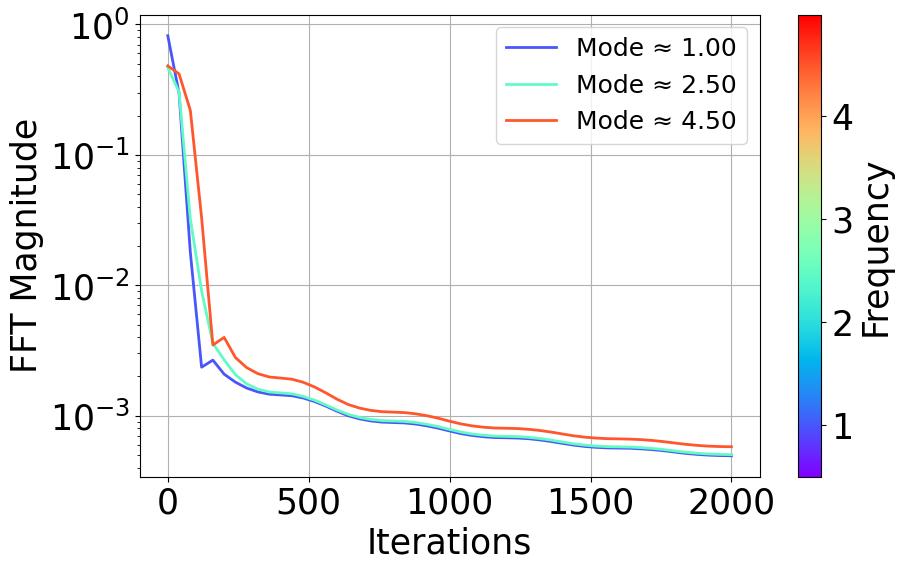}\\[3pt]
(c)&(d)\\
\includegraphics[width=0.45\textwidth]{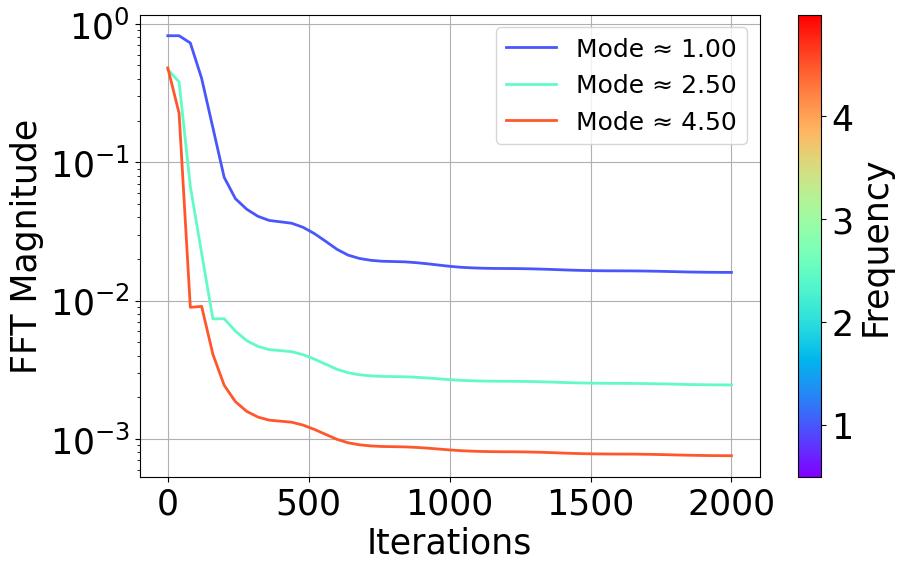}&
\includegraphics[width=0.45\textwidth]{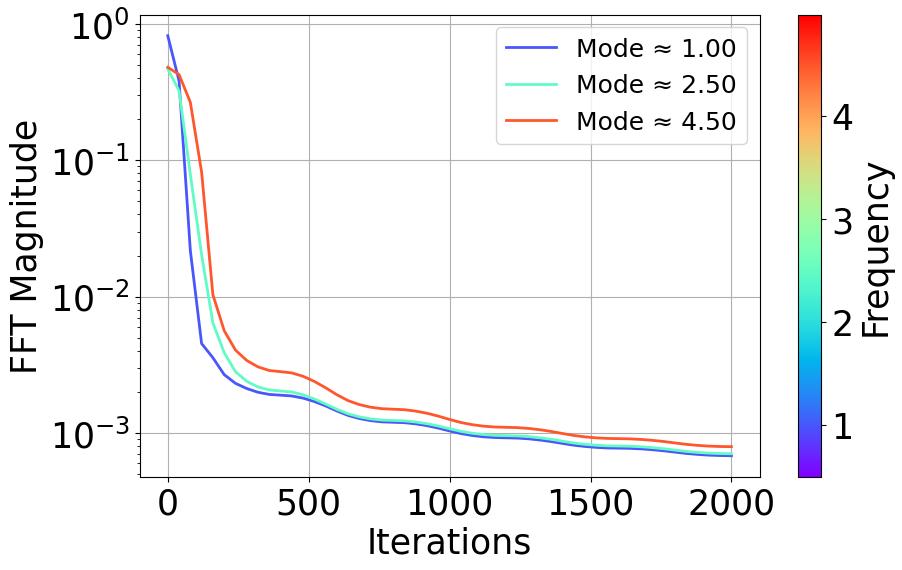}
\end{tabular}
\caption{Example 4: Comparison of $L^2$ error and errors in selected frequency modes in training for different models.  
(a) Training loss for all models;  
(b) Function Approximation;  
(c) PINN;  
(d) Preconditioned PINN.}
\label{fig:example4_comparison}
\end{figure}

% \begin{figure}[htbp]
%     \centering

%     % --- Row 1 ---
%     \begin{subfigure}[b]{0.45\textwidth}
%         \includegraphics[width=\linewidth]{figures/example4_training_L2_loss_all_models.jpg}
%         \caption{Training loss (all models)}
%     \end{subfigure}
%     \hfill
%     \begin{subfigure}[b]{0.45\textwidth}
%         \includegraphics[width=\linewidth]{figures/example4_selected_freq_Approx_sin.jpg}
%         \caption{Approx model (freq)}
%     \end{subfigure}

%     \vspace{0.5cm}

%     % --- Row 2 ---
%     \begin{subfigure}[b]{0.45\textwidth}
%         \includegraphics[width=\linewidth]{figures/example4_selected_freq_PINN_sin.jpg}
%         \caption{PINN model (freq)}
%     \end{subfigure}
%     \hfill
%     \begin{subfigure}[b]{0.45\textwidth}
%         \includegraphics[width=\linewidth]{figures/example4_selected_freq_kernelG_PINN_sin.jpg}
%         \caption{preconditioned PINN (freq)}
%     \end{subfigure}

%     \caption{Example 4: Comparison of training and frequency responses for different models.}
%     \label{fig:example4_comparison}
% \end{figure}
Fig.~\ref{fig:example4_comparison} presents the comparison between standard PINN, preconditioned PINN (labeled as \textit{kernelG\_PINN\_FCNN} in the legend), and direct function approximation. Indeed, we observe that the frequency modes dynamic for the preconditioned PINN is very similar to that for function approximation. The decaying rate of the $L^2$ error is also much better than the PINN formulation without preconditioning and comparable to the direct approximation.
The training times are 16 seconds for direct function approximation, and 24 and 26 seconds for the standard PINN and the preconditioned PINN, respectively. Despite the slight additional cost, the preconditioning strategy effectively mitigates the high-frequency bias induced by the differential operator in the standard PINN formulation, restoring balance across frequency modes and improving both convergence and training accuracy.

%  This demonstrates that the preconditioning strategy successfully mitigates the problematic high frequency bias induced by the differential operator in the standard PINN formulation, restoring balance across frequency modes and improving both convergence and accuracy for the training.
% The training times are 16 seconds for direct function approximation and are comparable for the standard PINN and the preconditioned PINN (24 and 26 seconds, respectively).

\textbf{Example 5:} We solve the Poisson equation with homogeneous Dirichlet boundary conditions, and the exact solution is the localized function $u_2$ in Example 3, which contains fine features and high frequency modes. We use an FCNN with width $200$ and depth $3$, maintaining the same training configuration as Example 4.
\begin{figure}[!htb]
\centering
\begin{tabular}{cc}
(a)&(b)\\
\includegraphics[width=0.45\textwidth]{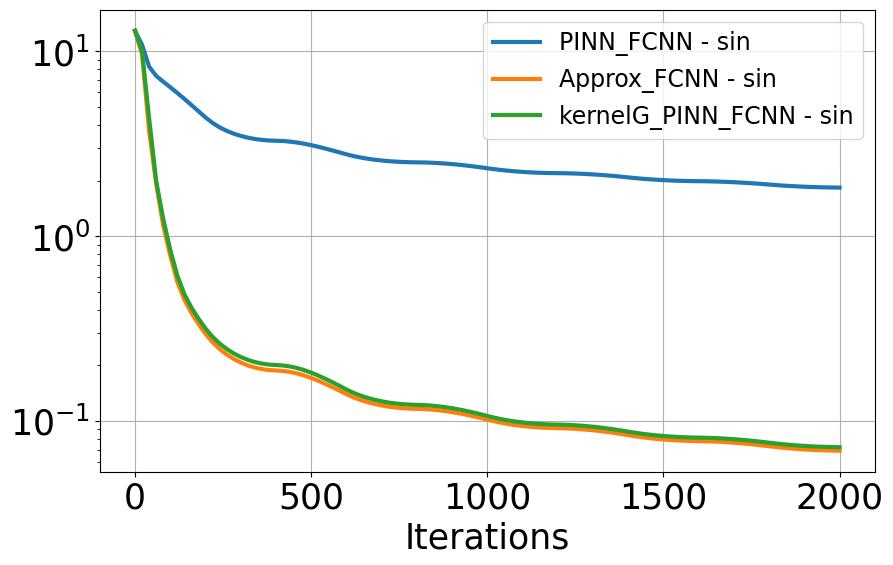}&
\includegraphics[width=0.45\textwidth]{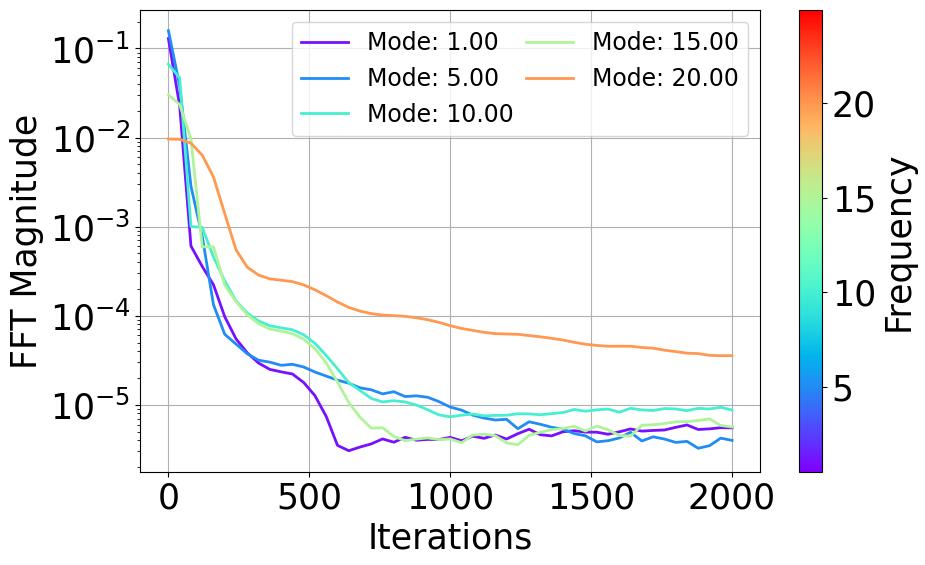}\\[3pt]
(c)&(d)\\
\includegraphics[width=0.45\textwidth]{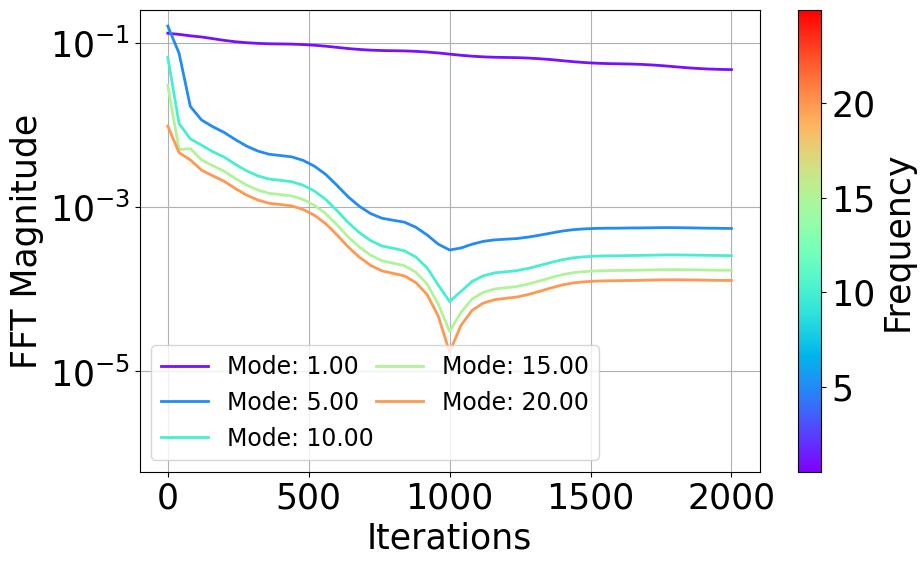}&
\includegraphics[width=0.45\textwidth]{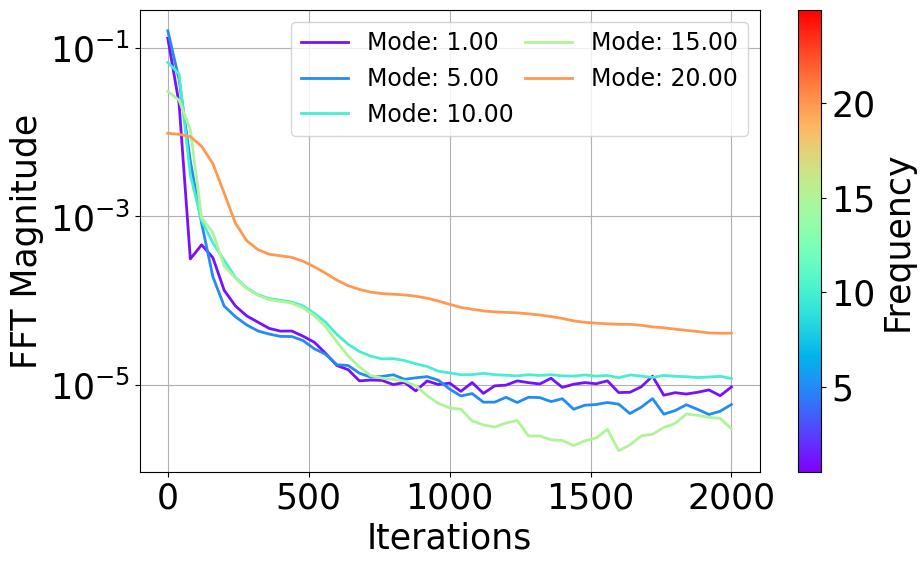}
\end{tabular}
\caption{Example 5: Comparison of $L^2$ error and errors in selected frequency modes in training for different models.  
(a) Training loss for all models;  
(b) Function approximation;  
(c) PINN;  
(d) Preconditioned PINN.}
\label{fig:example5_comparison}
\end{figure}

In this example, the errors in selected frequency modes  $1.0$, $5.0$, $10.0$, $15.0$, and $20.0$, are shown in Fig.~\ref{fig:example5_comparison}, representing a broader range than in Example 4. Since the localized function $u_2$
 contains quite high-frequency components, using a relatively shallow neural network for function approximation still shows mild low frequency bias and requires more iterations to capture the frequency mode 20, as shown in Fig.~\ref{fig:example5_comparison}(b). However, Fig.~\ref{fig:example5_comparison}(c) shows that, for the standard PINN, the presence of the differential operator changes the picture completely due to reweighting the frequency modes in the loss function and placing more emphasis on high-frequency components. As a result, it leads to strong high frequency bias and slow convergence of smooth components in training, which is unacceptable in real practice (similar to using finite element methods for elliptic PDEs without a preconditioner).  In contrast, the preconditioned PINN exhibits frequency mode dynamics that are nearly identical to those of the direct function approximation and consequently achieves much improved efficiency and accuracy in training. The training times are 32 seconds for direct function approximation and are comparable for the standard PINN and the preconditioned PINN (38 and 41 seconds, respectively). The tests again demonstrate the effectiveness of our proposed preconditioning strategy.

%This example demonstrates that preconditioning is particularly effective for solutions with localized high-frequency features, where standard PINNs fail catastrophically due to the compounded effects of spectral bias and differential operator structure.

\textbf{Example 6(a):} We consider an elliptic PDE with a rapidly oscillating variable coefficient: \begin{equation}\label{eq_PDE_multi}
\begin{cases}
    \nabla \cdot (k(x)\nabla u)=-1&\text{in}~\Omega\\
    u = 0&\text{on}~\partial\Omega
\end{cases}\;,
\end{equation}
where 
\begin{equation}
k(x) = \frac{2}{2 + A \sin\left( \frac{2\pi x}{\varepsilon} \right)} + 2 + A \cos\left( \frac{2\pi x}{\varepsilon} \right).
\end{equation}
In this test, we take $A = 1.8$ and $\varepsilon = 0.1$ which creates a rapidly oscillating diffusion coefficient with period $\varepsilon$, leading to a solution with both slow global variations and fast local oscillations, which is a classical homogenization problem. We use a FCNN of width 100 and depth 3, corresponding to  20,501 learnable parameters, to represent the solution.

This example is particularly challenging for PINNs due to the multiscale nature of the solution. We observe that the standard PINN suffers from strong high frequency bias and struggles to capture the low-frequency modes, resulting in a large relative $L^2$ error even after 20,000 epochs, with the error decaying very slowly. Despite the fact that the reference differential operator $\Delta$ differs from the exact differential operator in the equation, the frequency bias is well alleviated by the preconditioner, showing much less high frequency bias as reflected in the mode dynamics \ref{fig:example6a_summary}(d). This reduction in frequency bias is accompanied by a notable improvement in overall accuracy. Specifically, the relative $L^2$ error for the standard PINN with an FCNN is $9.61 \times 10^{-1}$, while the preconditioned PINN reduces the error to $1.80 \times 10^{-2}$. For comparison, the direct function approximation model using an FCNN achieves a relative $L^2$ error of $1.88 \times 10^{-3}$. 
% Training for 20,000 epochs requires 227 seconds for a direct approximation, 298 seconds for the standard PINN, and 327 seconds for the preconditioned PINN.

\begin{figure}[!htbp]
\centering
\begin{tabular}{cc}
(a)&(b)\\
\includegraphics[width=0.45\textwidth]{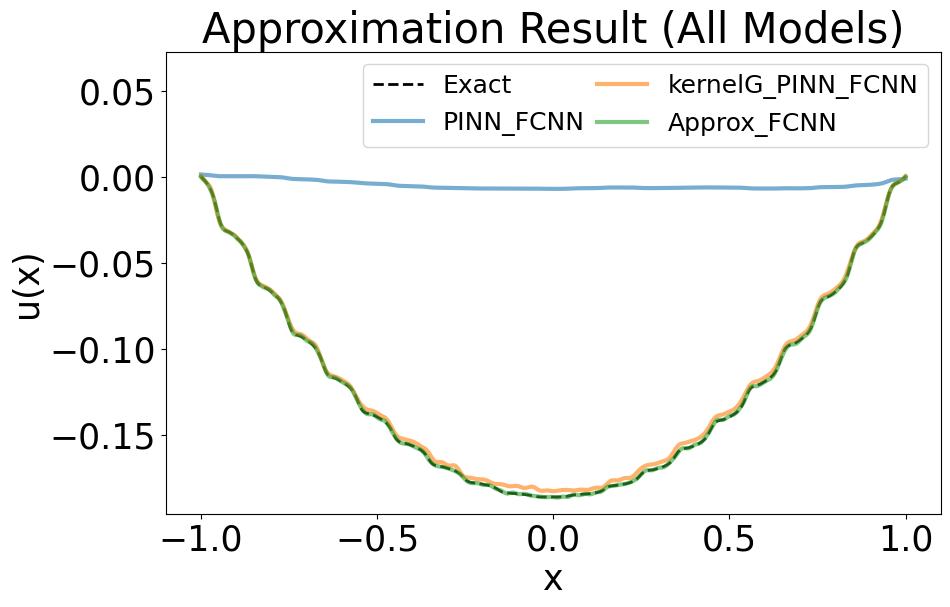}&
\includegraphics[width=0.45\textwidth]{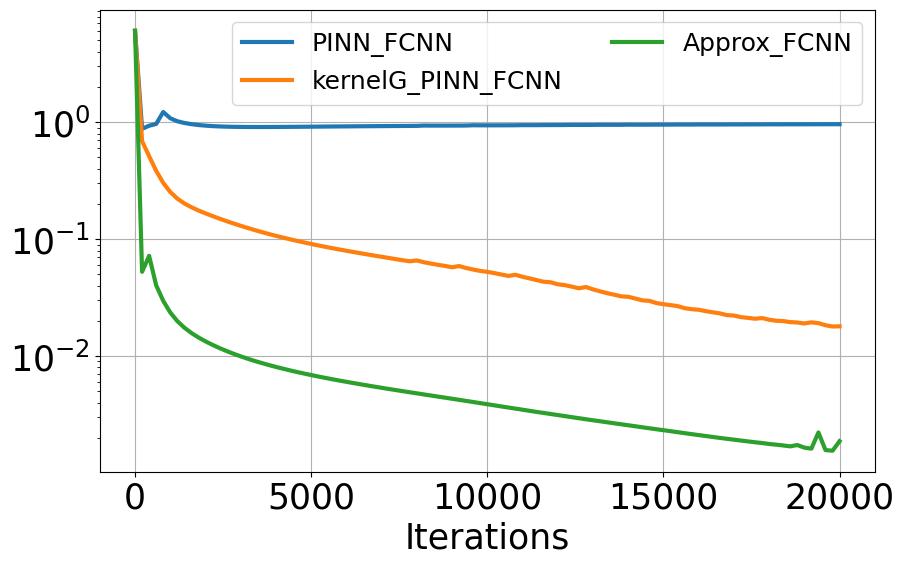}\\[3pt]
\end{tabular}
\begin{tabular}{ccc}
(c)&(d)&(e)\\
\includegraphics[width=0.33\textwidth]{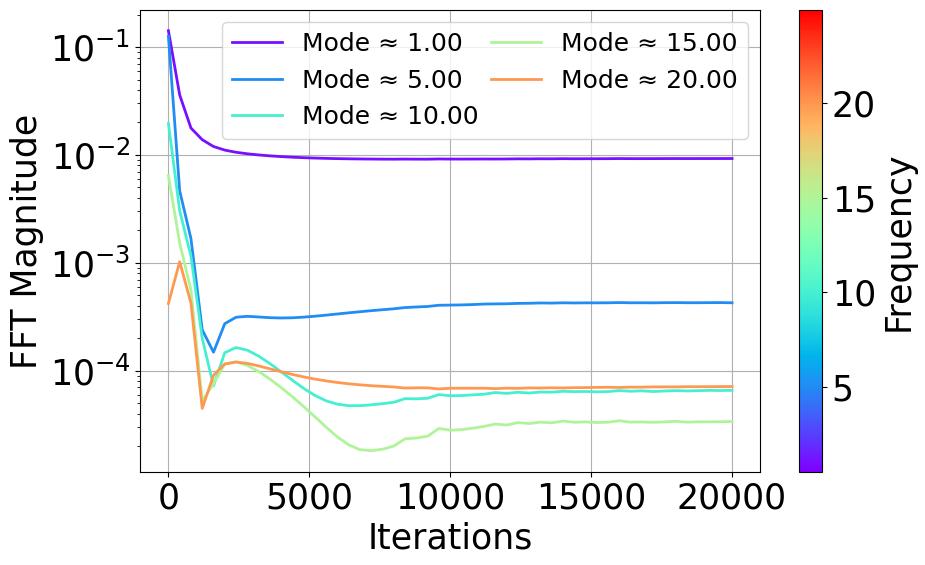}&
\includegraphics[width=0.33\textwidth]{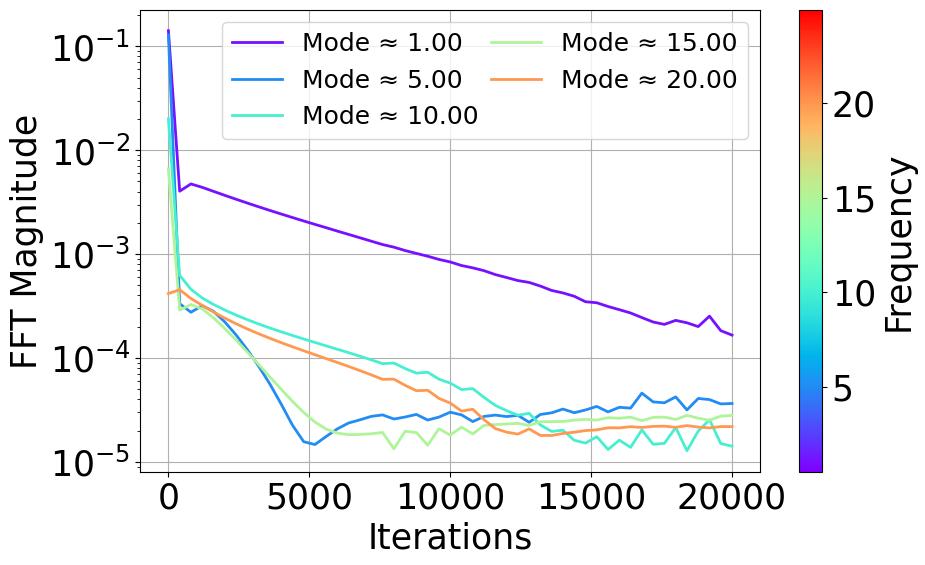}&
\includegraphics[width=0.33\textwidth]{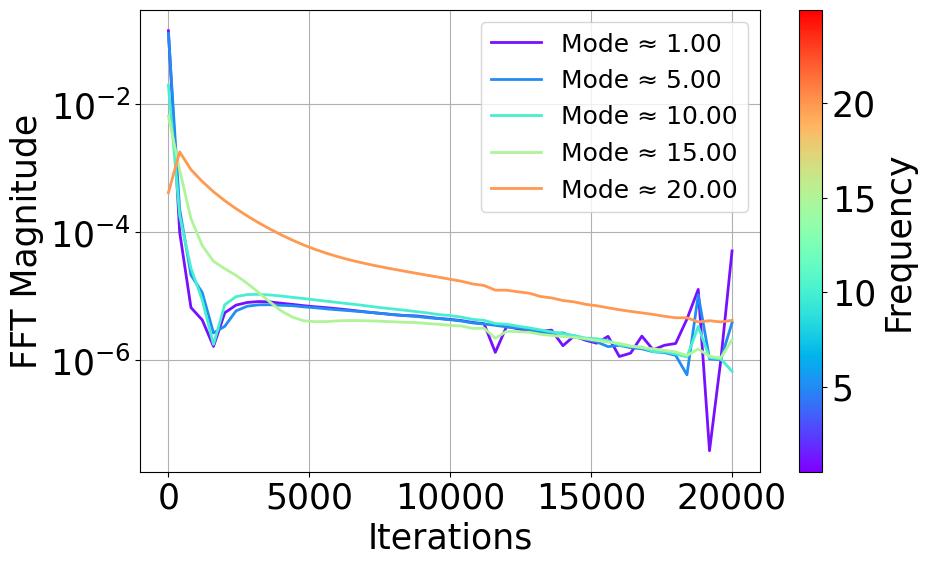}
\end{tabular}
 \caption{Example 6(a): (a) Exact solution and numerical results using FCNN; (b) $L^2$ relative error history; Errors in selected modes for PINN (c), preconditioned PINN (d), and approximation (e). }
    \label{fig:example6a_summary}
    \end{figure}

\textbf{Example 6(b):} In this example, we evaluate the use of a Fourier MMNN ~\cite{zhang2025fourier} (FMMNN) with width 50, rank 20, and depth 3 for function approximation or to solve the same variable coefficient equation from Example 6(a).

Similar phenomena as in Example 6(a) are observed. The MMNN architecture achieves the same level of accuracy with significantly fewer learnable parameters, with 2,091 parameters compared to 20,501 parameters for the FCNN, which is approximately a tenfold reduction. Despite this substantial decrease in model size, the same qualitative trends persist. In particular, the standard PINN with an MMNN exhibits a large relative $L^2$ error of $9.78 \times 10^{-1}$, while the preconditioned PINN reduces the error to $3.07 \times 10^{-2}$. For comparison, the MMNN-based approximation model achieves a relative $L^2$ error of $4.16 \times 10^{-3}$. 
This demonstrates that architectural innovations and preconditioning are complementary strategies for improving PINN performance, regardless of whether an FCNN or MMNN is employed. 
% Training for 20,000 epochs requires 243 seconds for a direct approximation, 348 seconds for the standard PINN, and 358 seconds for the preconditioned PINN.

\begin{figure}[!htbp]
\centering
\begin{tabular}{cc}
(a)&(b)\\
\includegraphics[width=0.45\textwidth]{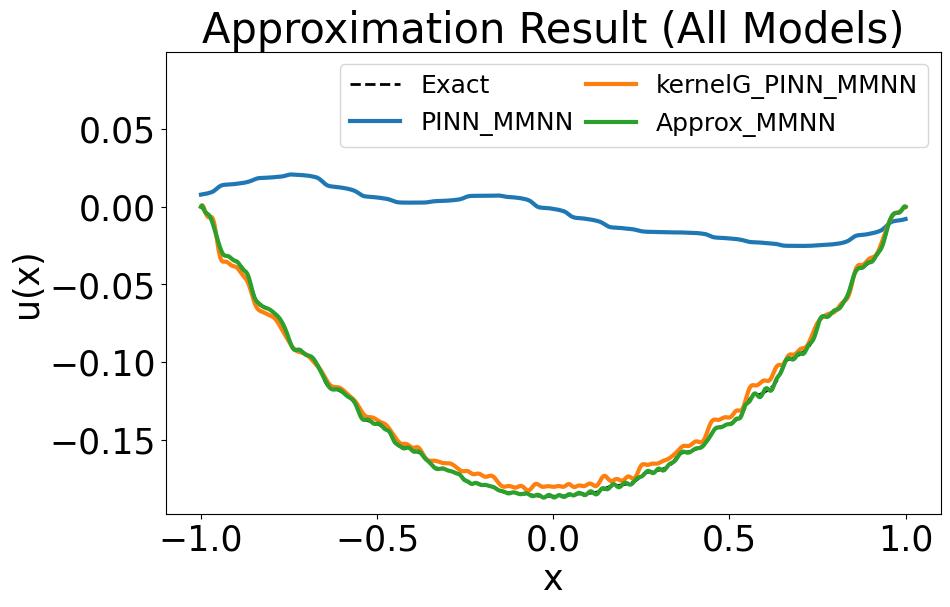}&
\includegraphics[width=0.45\textwidth]{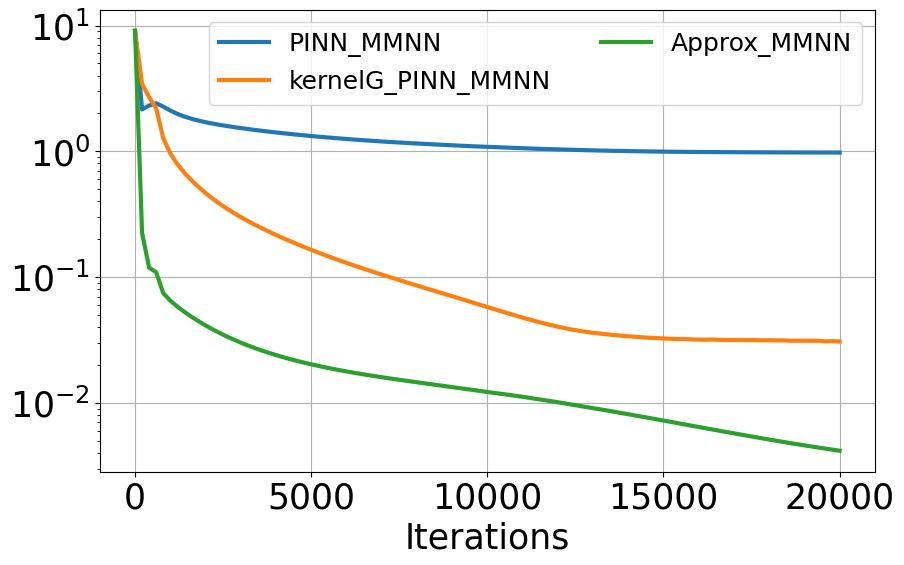}\\[3pt]
\end{tabular}
\begin{tabular}{ccc}
(c)&(d)&(e)\\
\includegraphics[width=0.33\textwidth]{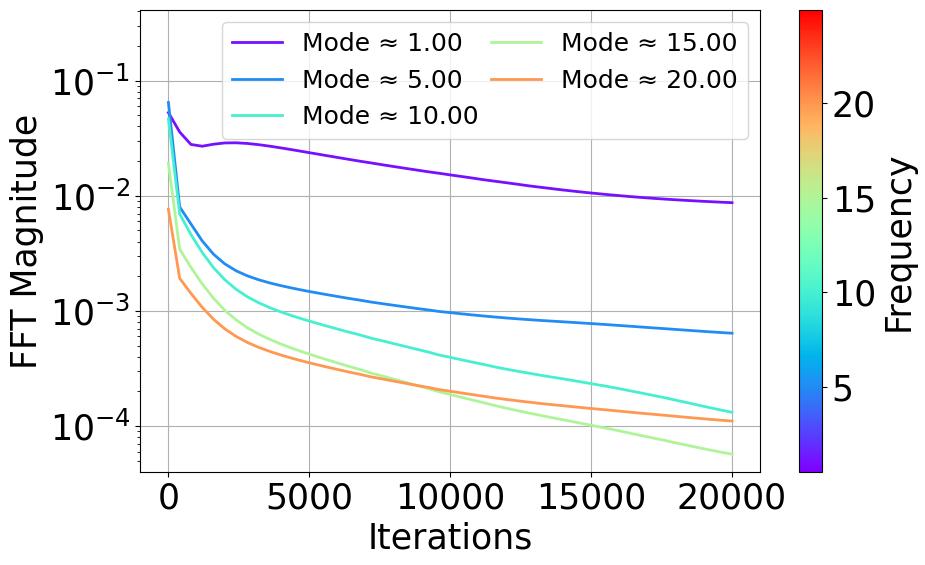}&
\includegraphics[width=0.33\textwidth]{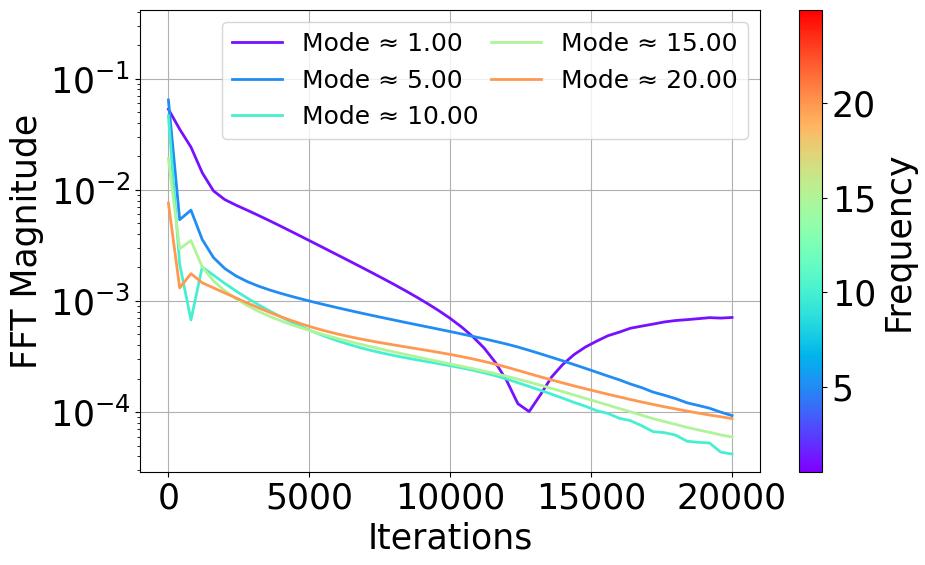}&
\includegraphics[width=0.33\textwidth]{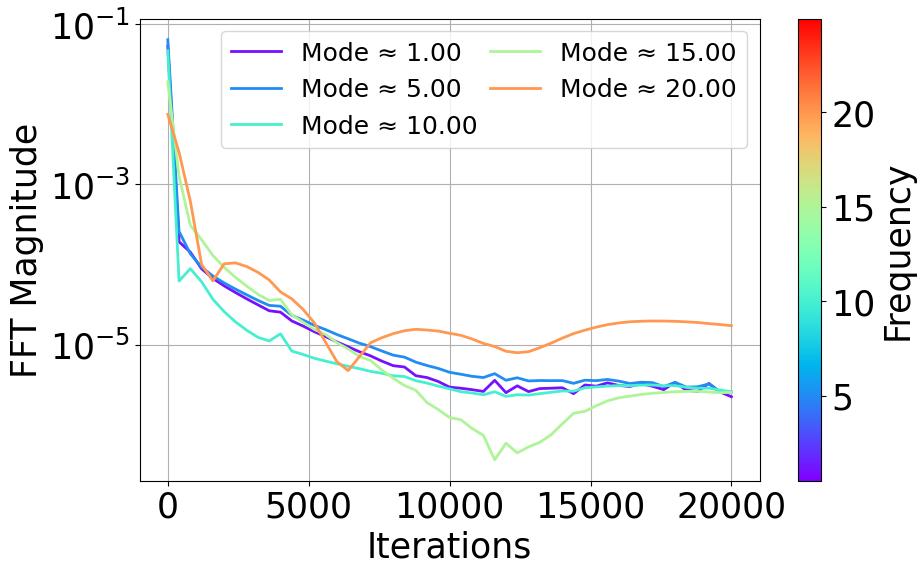}
\end{tabular}
 \caption{Example 6(b): (a) Exact solution and numerical results using MMNN; (b) $L^2$ relative error history; Errors in selected modes for PINN (c), preconditioned PINN (d), and approximation (e). }
\label{fig:example11_summary}
    \end{figure}

As demonstrated in Example~6, the MMNN architecture is able to achieve comparable accuracy with significantly fewer learnable parameters.  We adopt the MMNN architecture for all subsequent two-dimensional numerical experiments.

In the two-dimensional setting, the frequency bias is analyzed using a two-dimensional discrete Fourier transform of the error. Each Fourier coefficient corresponds to a pair of frequency modes $(k_x,k_y)$ in the $x$- and $y$-directions respectively. To visualize the frequency-dependent learning dynamics, we track a selected set of Fourier modes $(k_x,k_y) = (1,0), (2,0), (3,0), (4,0)$, which correspond to oscillations along the $x$-direction. Monitoring the errors in these modes over training iterations enables direct comparison with the one-dimensional frequency analysis. At the same time, the spatial error will also be plotted to offer a global view of the approximation quality and reveal error patterns in the learned solution. Together, the selected mode dynamics and spatial error visualization provide complementary insight into frequency bias in the two-dimensional setting.

\textbf{Example 7:} We solve the Poisson equation $\Delta u = -f$ in $\Omega = [-1,1]\times[-1,1]$ with homogeneous Dirichlet boundary conditions, and the source function $f$ is chosen such that the exact solution is
% \[
% \begin{aligned}
% u_3(x_1, x_2) =\; &1.0 \cdot \phi(x_1, x_2; 0.0, 0.0, 0.8) \\
% &+ 0.7 \cdot \phi(x_1, x_2; 0.6, 0.6, 0.2) \\
% &+ 0.5 \cdot \phi(x_1, x_2; -0.5, 0.3, 0.1) \\
% &+ 0.3 \cdot \phi(x_1, x_2; 0.2, -0.7, 0.1),
% \end{aligned}
% \]
% where $\phi$ is a smooth bump function with center $(c_1,c_2)$ and scale $s$:
% \[
% \phi(x_1, x_2; c_1, c_2, s) =
% \begin{cases}
% \exp\left( -\dfrac{1}{1 - \left( \dfrac{(x_1 - c_1)^2 + (x_2 - c_2)^2}{s^2} \right)} \right), & \text{if } \dfrac{(x_1 - c_1)^2 + (x_2 - c_2)^2}{s^2} < 1 \\
% 0, & \text{otherwise.}
% \end{cases}
% \]
\[
\begin{aligned}
u_3(x) =\;
&1.0 \cdot \phi(x; c^{(1)}, 0.8) \\
&+ 0.7 \cdot \phi(x; c^{(2)}, 0.2) \\
&+ 0.5 \cdot \phi(x; c^{(3)}, 0.1) \\
&+ 0.3 \cdot \phi(x; c^{(4)}, 0.1),
\end{aligned}
\]
where
\[
c^{(1)} = (0.0, 0.0), \quad
c^{(2)} = (0.6, 0.6), \quad
c^{(3)} = (-0.5, 0.3), \quad
c^{(4)} = (0.2, -0.7),
\]
and $\phi$ is the smooth compactly supported bump function
\[
\phi(x; c, s) =
\begin{cases}
\displaystyle
\exp\!\left(
-\dfrac{1}{1 - \dfrac{|x - c|^2}{s^2}}
\right),
& \text{if } |x - c| < s, \\[1.2em]
0, & \text{otherwise}.
\end{cases}
\]

This creates a solution with four localized bumps of different amplitudes and scales, as shown in Fig.~\ref{fig:example22_losshistory}(a). The largest bump at the origin has scale $0.8$, while the smallest bumps have scale $0.1$, creating a multi-scale challenge. We use an MMNN architecture of width $200$, rank 20 and depth 3, corresponding to 8,241 learnable parameters, and train with 10,000 collocation points.

\begin{figure}[!htb]
\centering
\begin{tabular}{cc}
(a)&(b)\\
\includegraphics[width=0.33\textwidth]{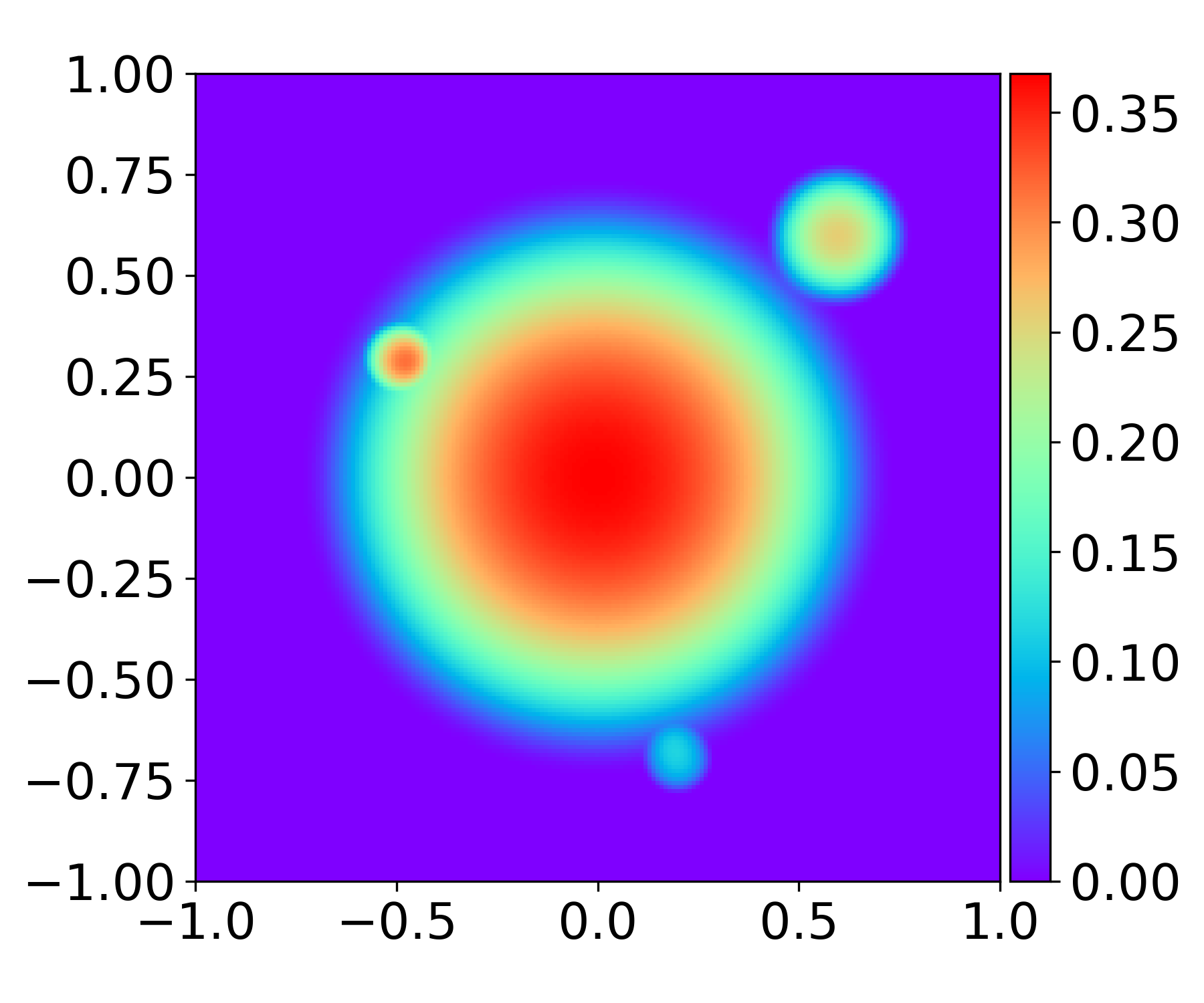}&
\includegraphics[width=0.45\textwidth]{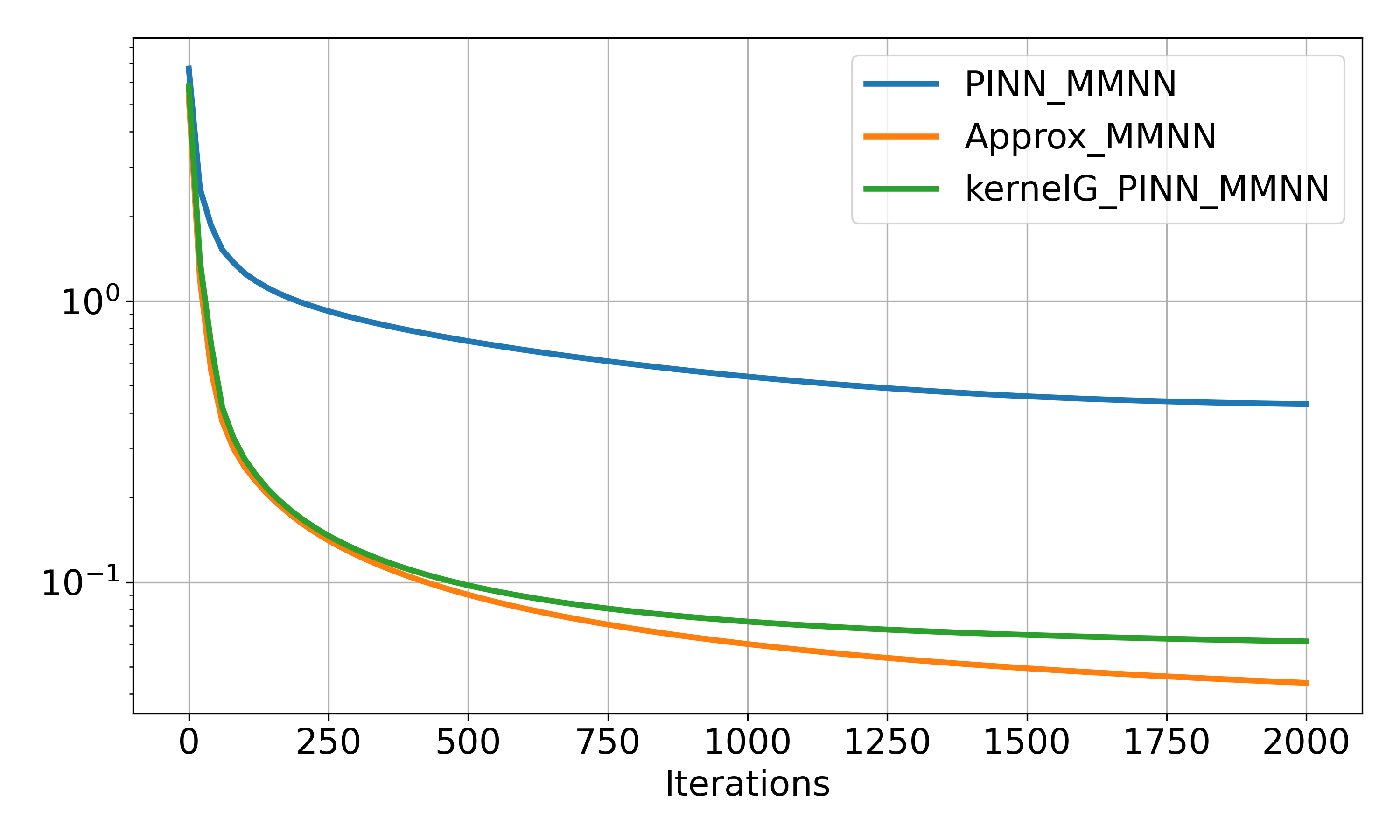}
\end{tabular}
\caption{Example 7: (a) Exact solution; (b) Relative $L^2$ error during training for different models.}\label{fig:example22_losshistory}
    \end{figure}

\begin{figure}[!htb]
\centering
\begin{tabular}{ccc}
(a)&(b)&(c)\\
\includegraphics[width=0.33\linewidth]{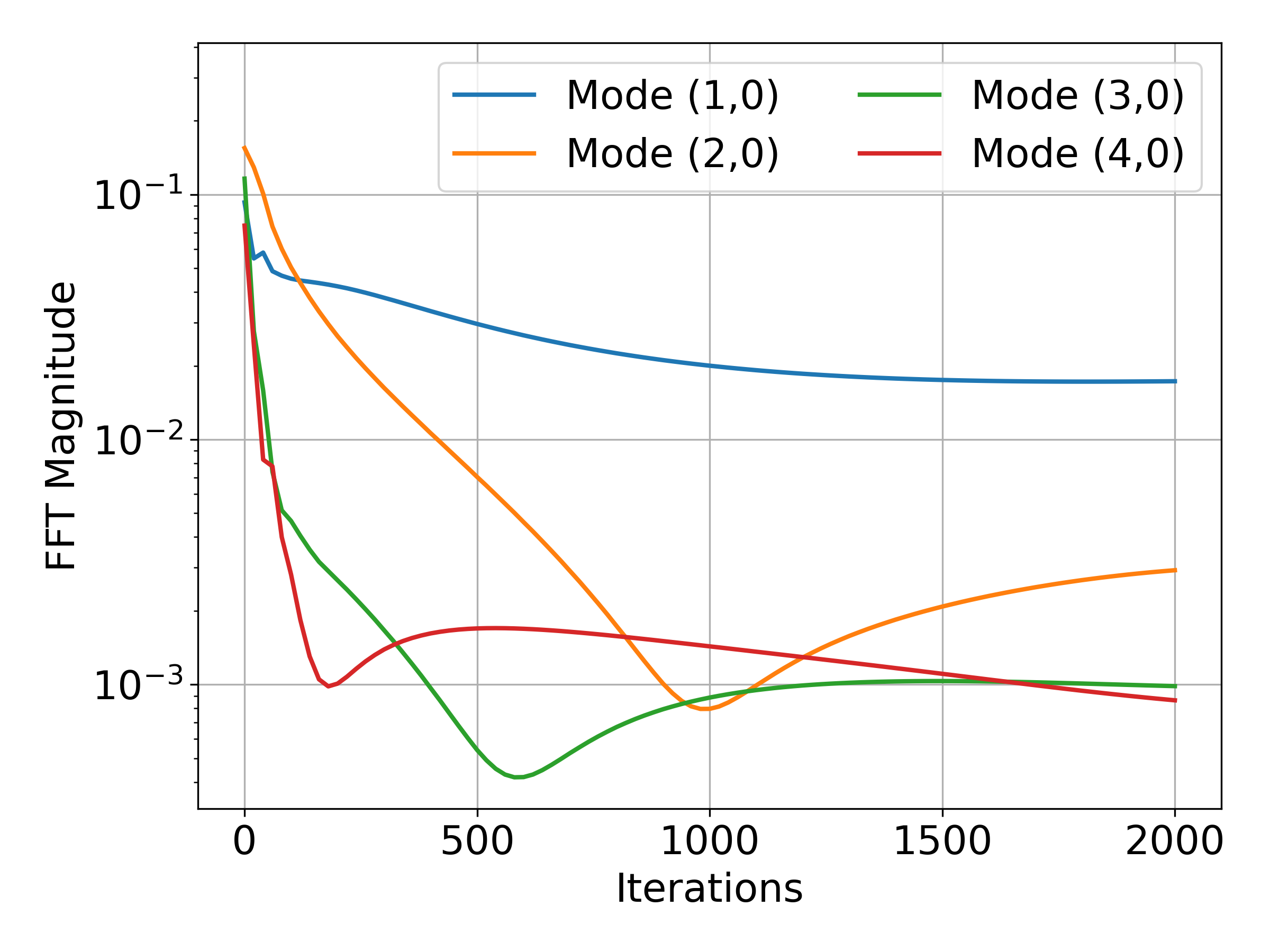}&
\includegraphics[width=0.33\linewidth]{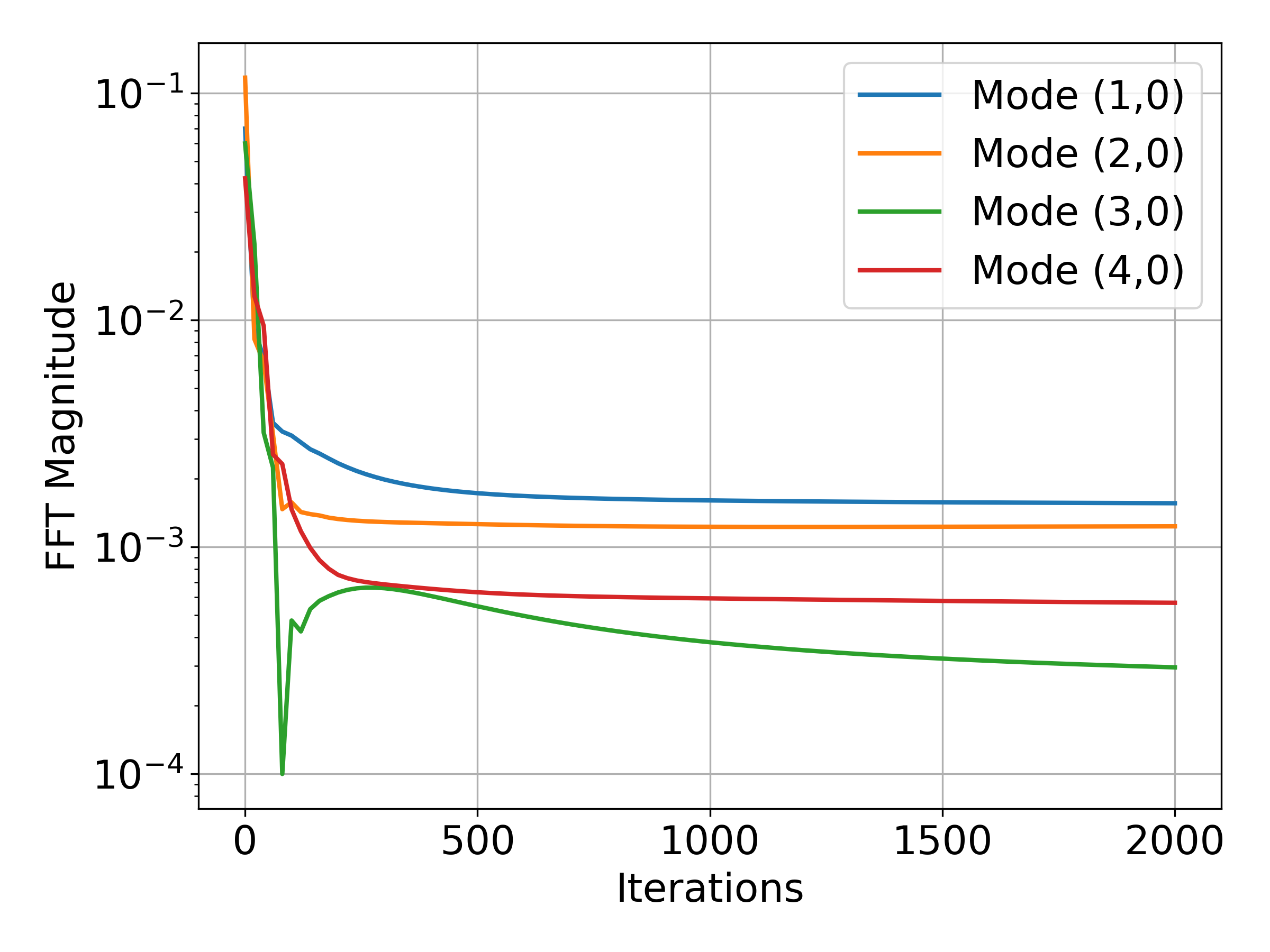}&
\includegraphics[width=0.33\linewidth]{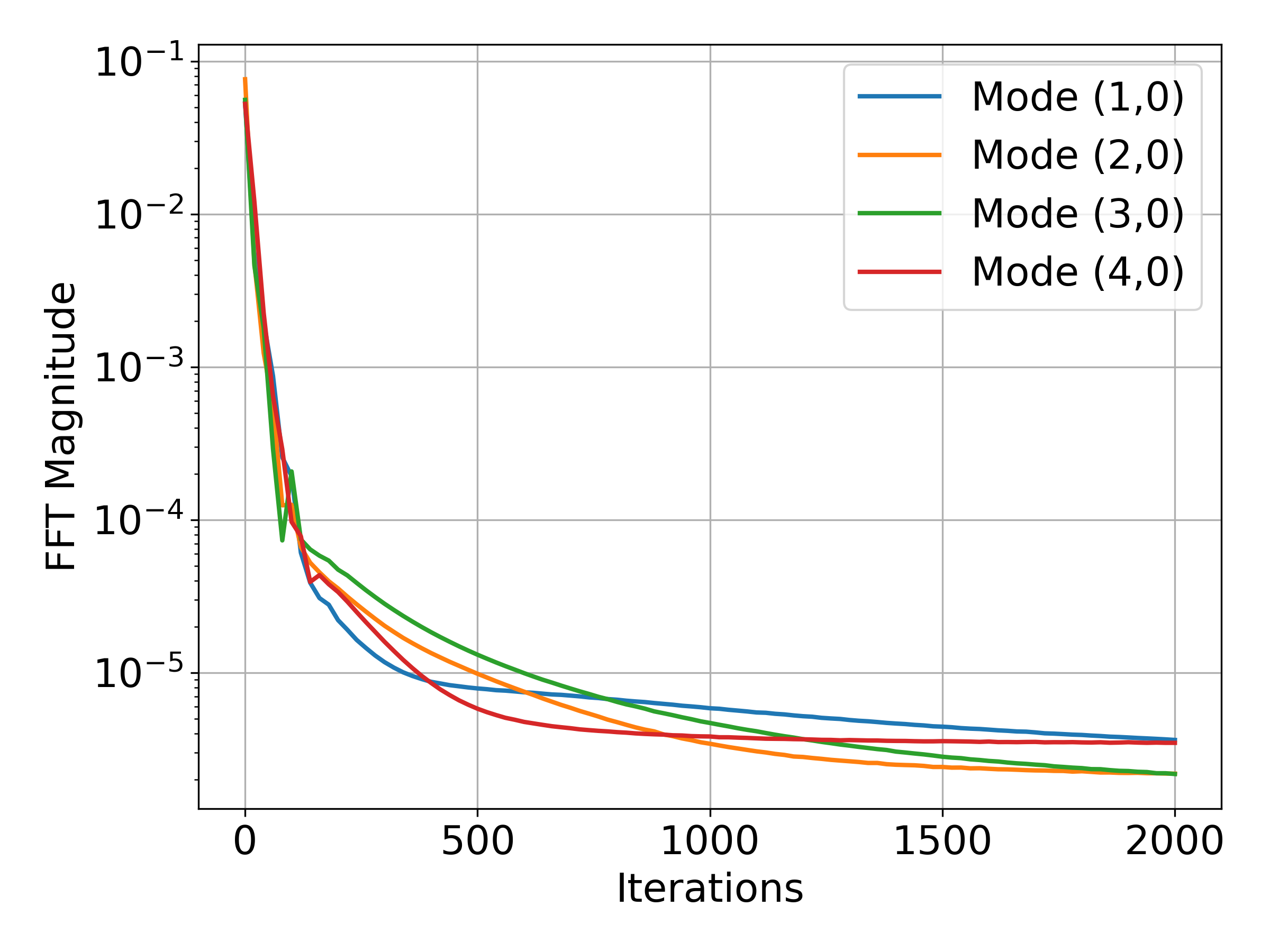}\\
(d)&(e)&(f)\\
\includegraphics[width=0.33\linewidth]{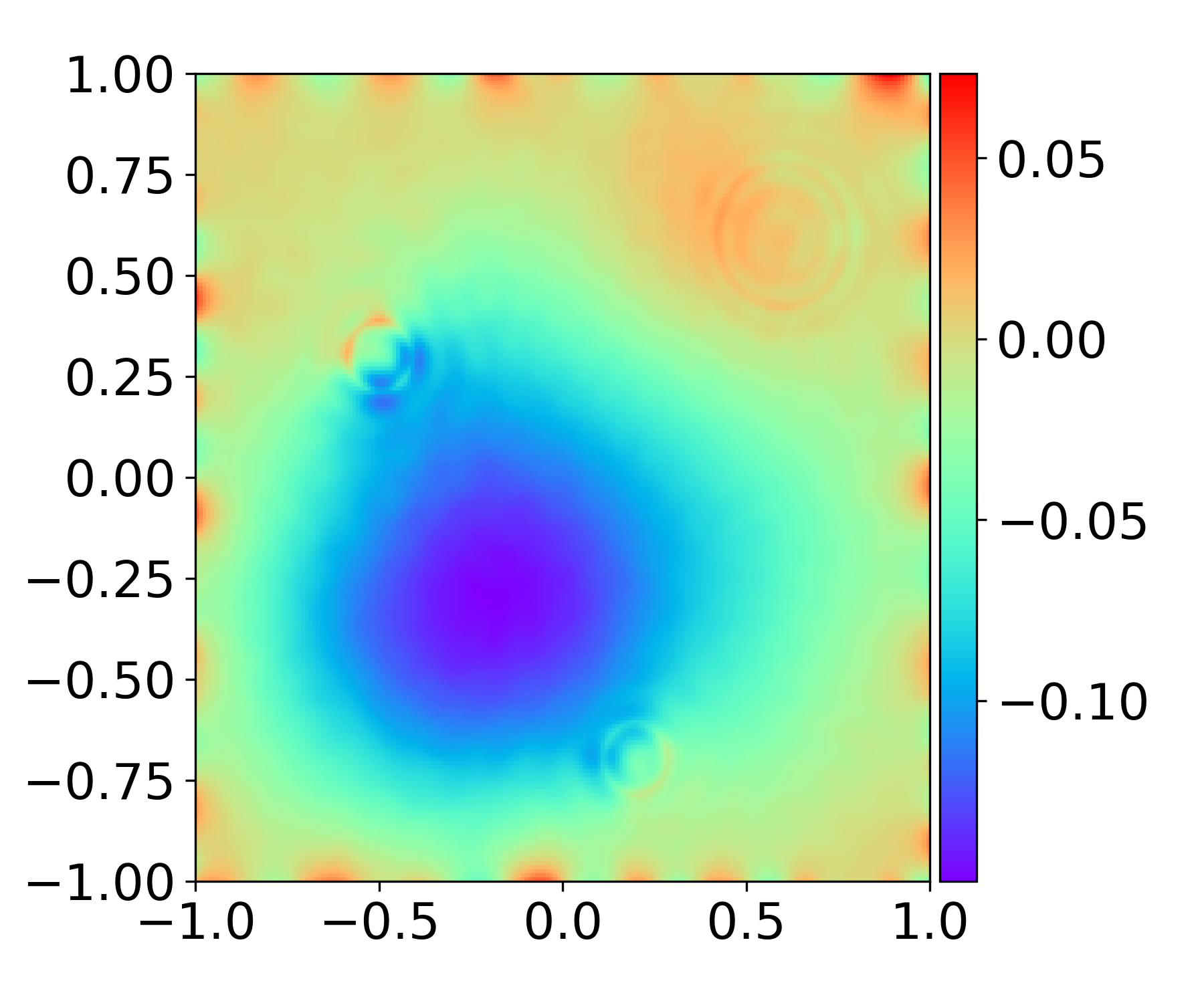}&
\includegraphics[width=0.33\linewidth]{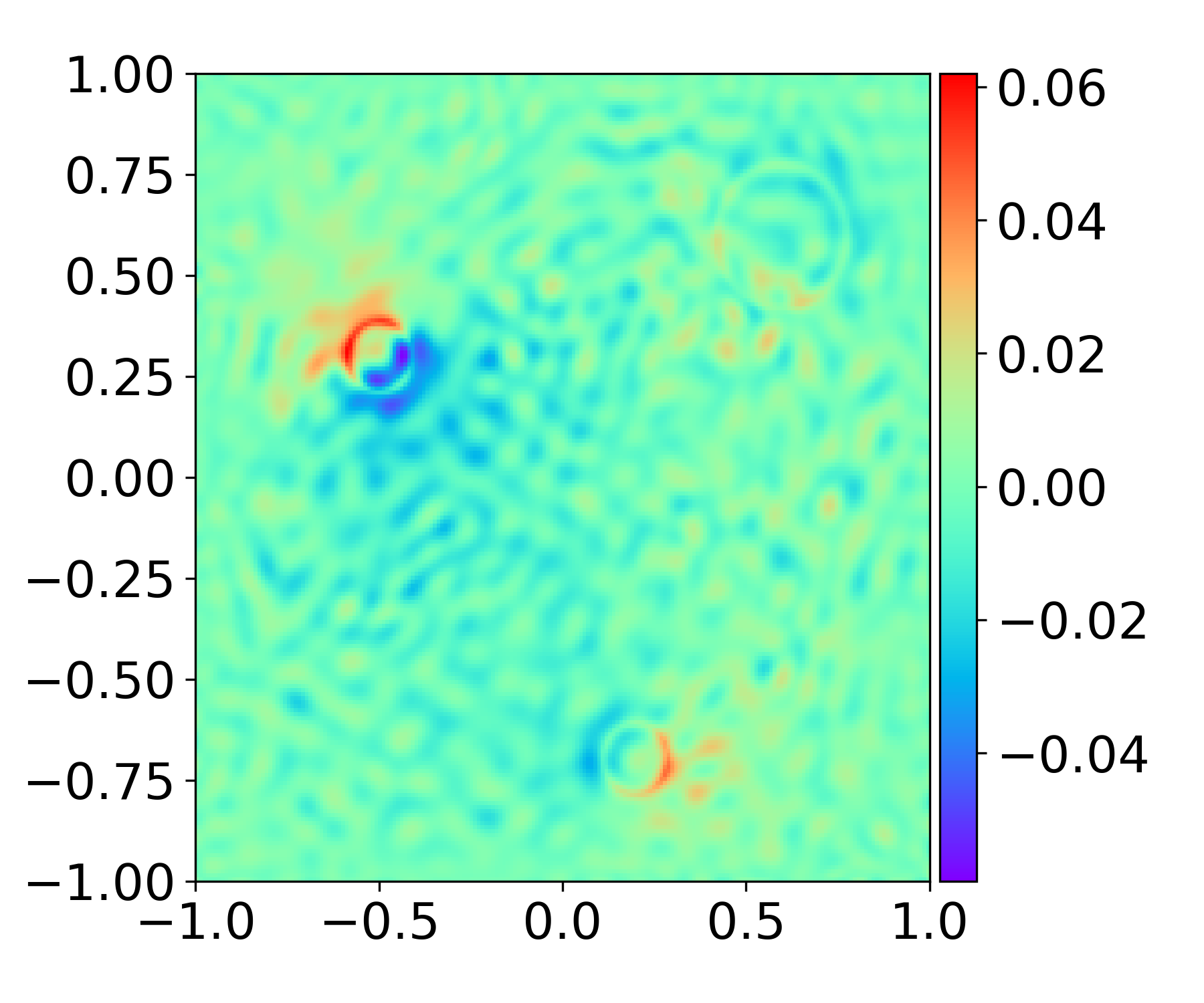}&
\includegraphics[width=0.33\linewidth]{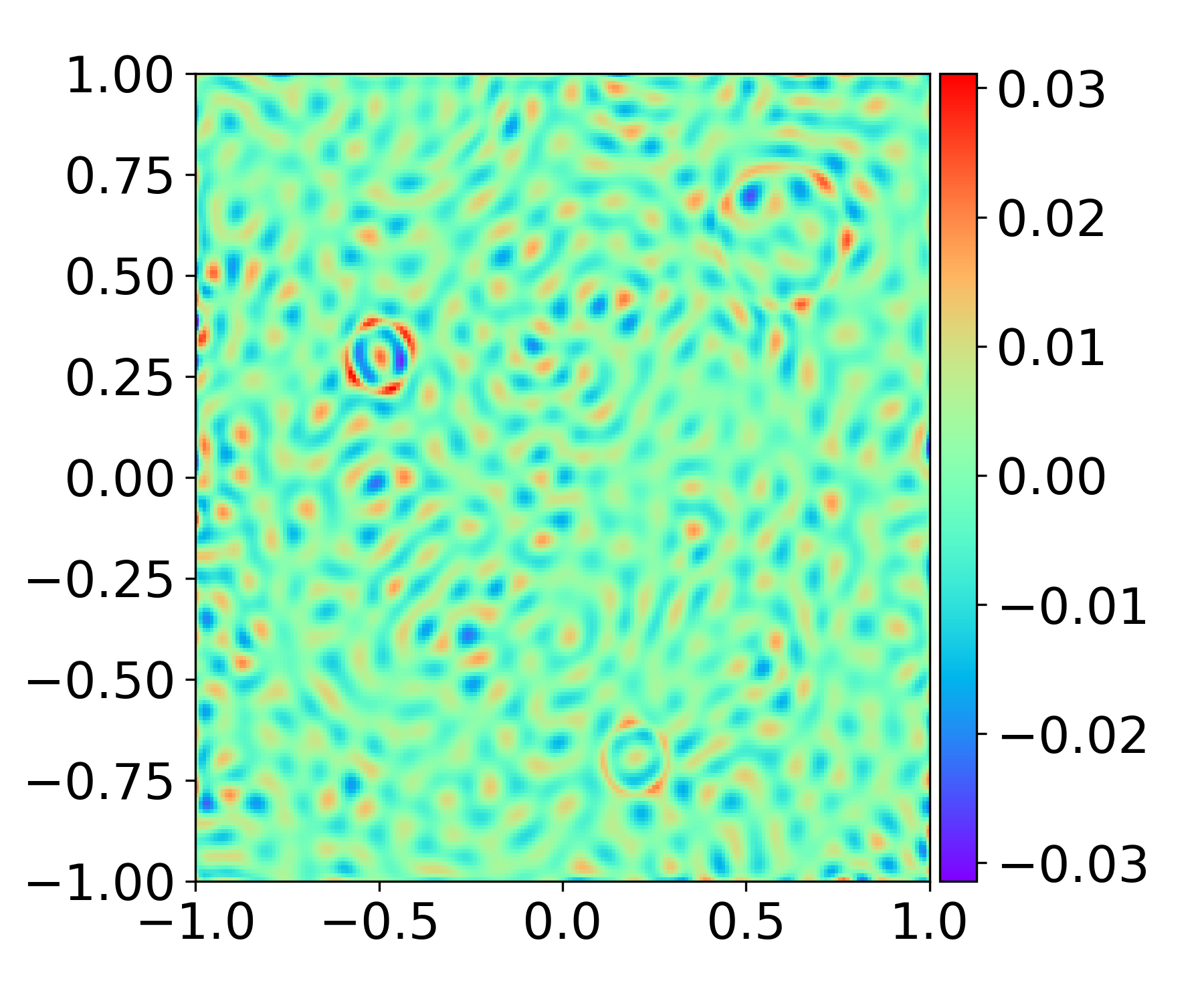}

\end{tabular}
  \caption{Example 7: Errors in selected modes for PINN (a), preconditioned PINN (b) and approximation (c); Spatial error distributions for PINN (d), preconditioned PINN (e) and  approximation (f).}
    \label{fig:example22_fourier_and_difference}
\end{figure}

Fig.~\ref{fig:example22_losshistory}(b) shows dramatic different decay in $L^2$ errors during training across the three tasks. When approximating the function directly, the MMNN achieves the fastest convergence, reaching a relative $L^2$ error of $4.38\times10^{-2}$ in 2,000 epochs. In contrast, the standard PINN exhibits much slower convergence and plateaus at a larger error, with a relative $L^2$ error of $4.30\times10^{-1}$. The corresponding training times are 217 seconds for direct approximation, 259 seconds for the preconditioned PINN, and 256 seconds for the standard PINN. The preconditioned PINN has a significantly improved performance over the standard PINN, reducing the relative $L^2$ error to $6.15\times10^{-2}$, comparable to direct function approximation. These results and their underlying mechanism are revealed by the frequency bias shown in Figs.~\ref{fig:example22_fourier_and_difference}(a)–(c). Figs.~\ref{fig:example22_fourier_and_difference}(c) shows that FMMNN has little frequency bias and excellent convergence across all modes, which correlates well with its overall performance in $L^2$ error. Due to the weak frequency bias of FMMNN as a representation, the involvement of differential operators in a standard PINN model will induce strong high frequency bias, as shown in Figs.~\ref{fig:example22_fourier_and_difference}(a), where the error in the lowest frequency mode decays very slowly. In contrast, the preconditioned PINN (Fig.~\ref{fig:example22_fourier_and_difference}(b)) exhibits much more uniform convergence across all modes, similar to the function approximation task, again verifying the effectiveness of our preconditioning strategy.
%For the standard PINN (Fig.~\ref{fig:example22_fourier_and_difference}(a)), the lowest-frequency mode $(1,0)$ remains trapped at an error level of approximately $10^{-1}$, while higher-frequency modes converge to around $10^{-3}$, indicating an extreme high frequency bias. In contrast, the preconditioned PINN (Fig.~\ref{fig:example22_fourier_and_difference}(b)) exhibits much more uniform convergence across all modes, with errors settling near the $10^{-3}$ level. The approximation model (Fig.~\ref{fig:example22_fourier_and_difference}(c)) achieves the best spectral performance, with all modes converging to errors between $10^{-5}$ and $10^{-6}$. 
These phenomena are further corroborated by the corresponding spatial error distributions shown in Figs.~\ref{fig:example22_fourier_and_difference}(d)–(f): the error for standard PINN mainly concentrates on the largest bump, whereas the preconditioned PINN and direct function approximation yield more randomized error distributions, i.e, in high frequency modes. 

This example demonstrates that, when using well-designed NNs with good representation capability for high frequency modes and fine features (meaning weak frequency bias) to solve elliptic PDEs, proper preconditioning is crucial. Otherwise, it is difficult or costly to capture even a smooth solution using gradient-based training.     

\textbf{Example 8:} In this example, we solve the Poisson equation $\Delta u = -f$ in the unit disk $\Omega = B(0,1)\in\mathbb{R}^2$ with inhomogeneous Dirichlet boundary conditions. The exact solution is
\[
u_4(x) =
0.3 \cdot \sin\left(4\pi x_1 + 2\pi x_1 x_2\right) +
0.05 \cdot \sin\left(12\pi x_2 + 4\pi x_1 x_2\right),
\]
where $x = (x_1,x_2)$. We use an MMNN architecture of width 200, rank 40, and depth 3, corresponding to 16,281 learnable parameters, and train with 7,854 sample points.

\begin{figure}[!htb]
\centering
\begin{tabular}{cc}
(a)&(b)\\
\includegraphics[width=0.33\textwidth]{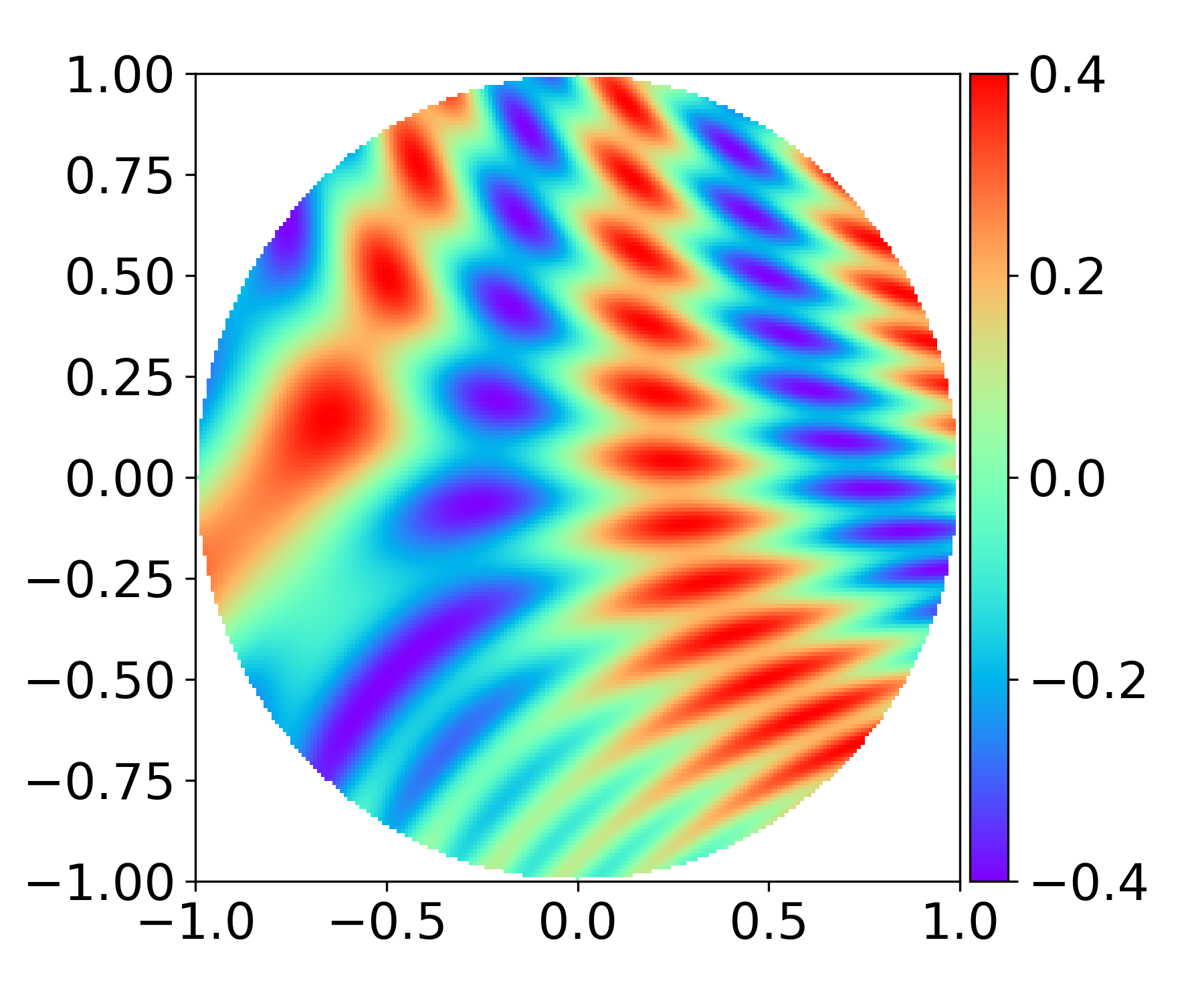}&
\includegraphics[width=0.45\textwidth]{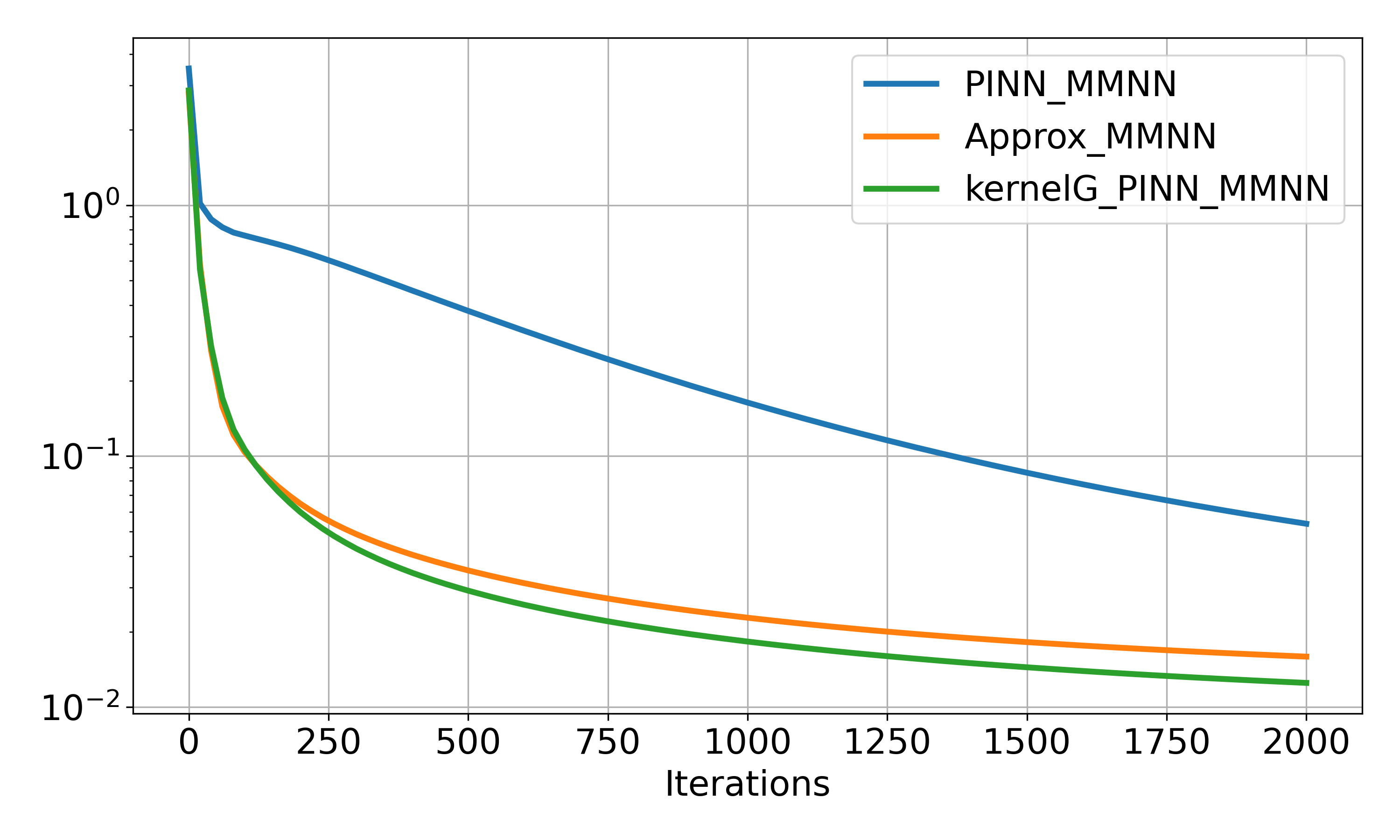}
\end{tabular}
 \caption{Example 8: (a) Exact solution; (b) Relative $L^2$ error during training for different models.}\label{Example7ab}
    \end{figure}

\begin{figure}[!htb]
\centering
\begin{tabular}{ccc}
(a)&(b)&(c)\\
\includegraphics[width=0.33\linewidth]{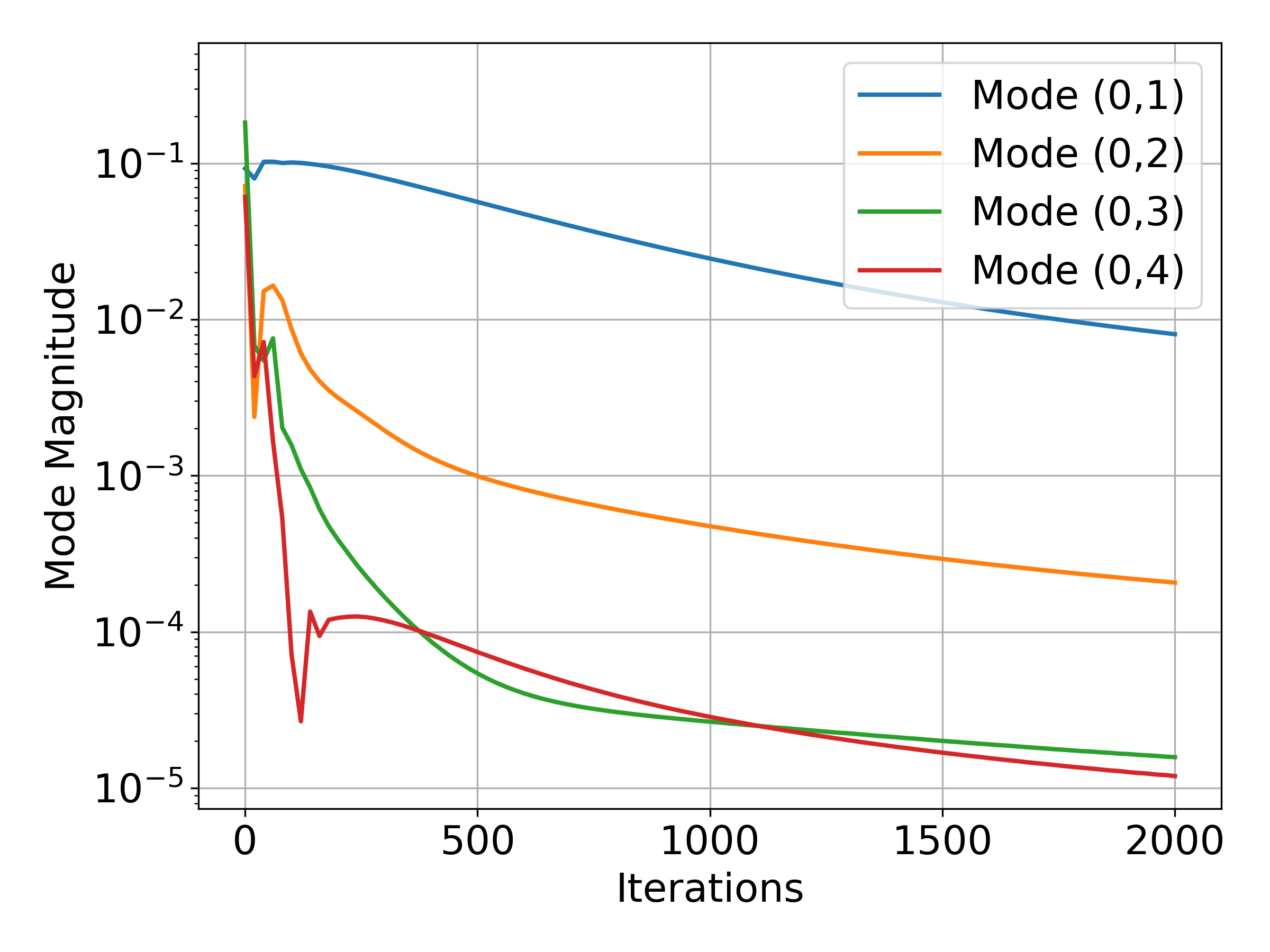}&
\includegraphics[width=0.33\linewidth]{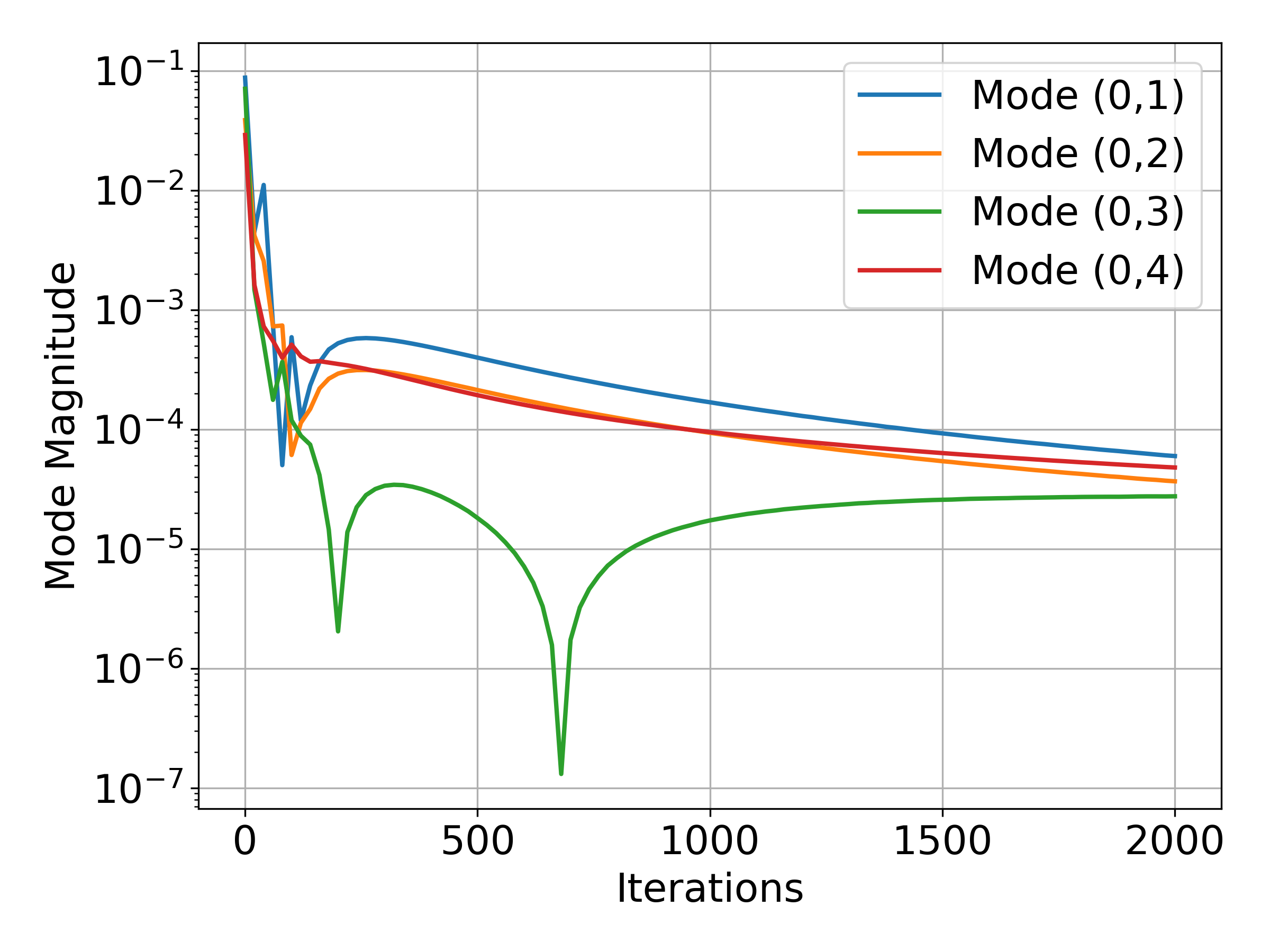}&
\includegraphics[width=0.33\linewidth]{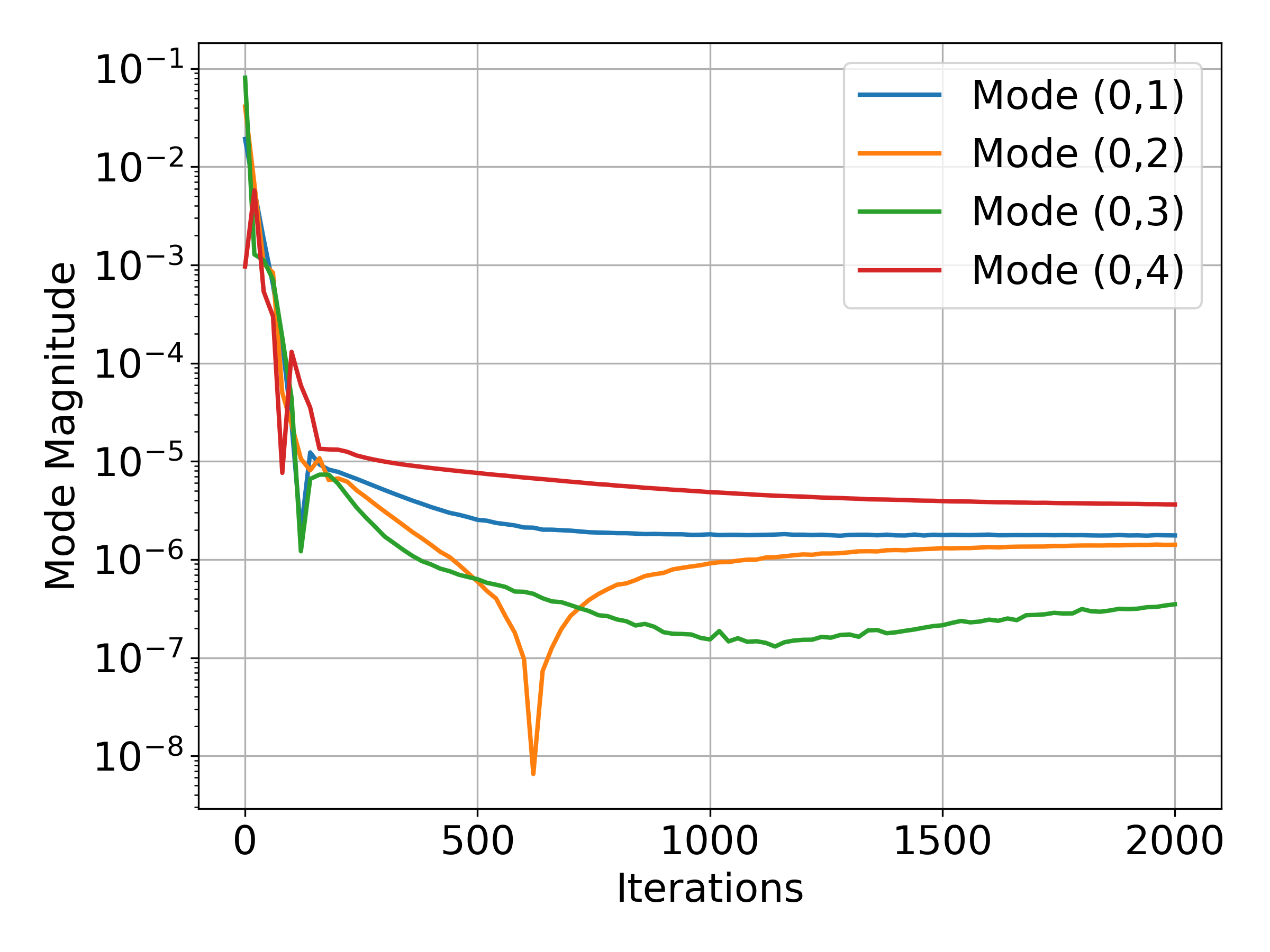}\\
(d)&(e)&(f)\\
\includegraphics[width=0.33\linewidth]{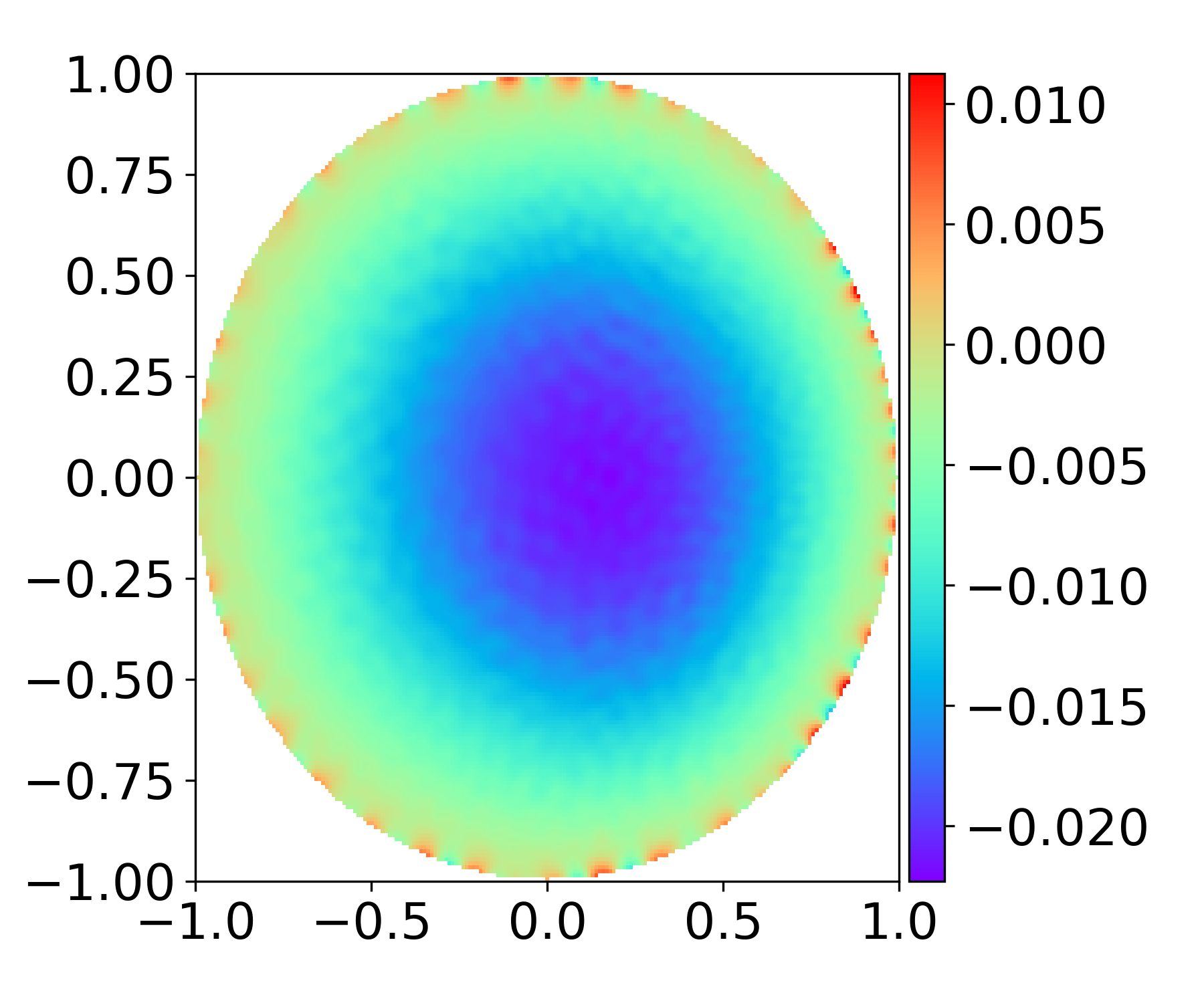}&
\includegraphics[width=0.33\linewidth]{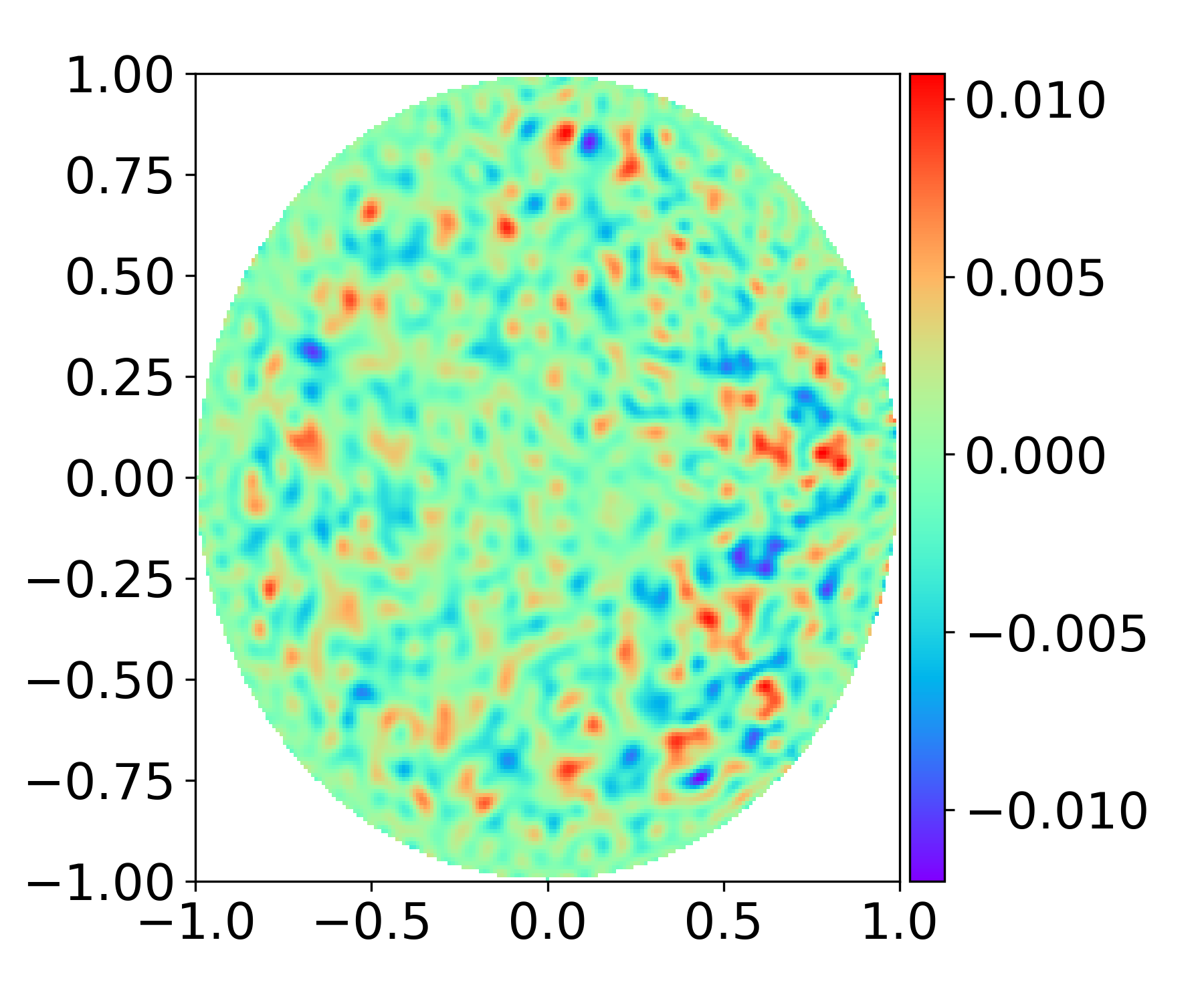}&
\includegraphics[width=0.33\linewidth]{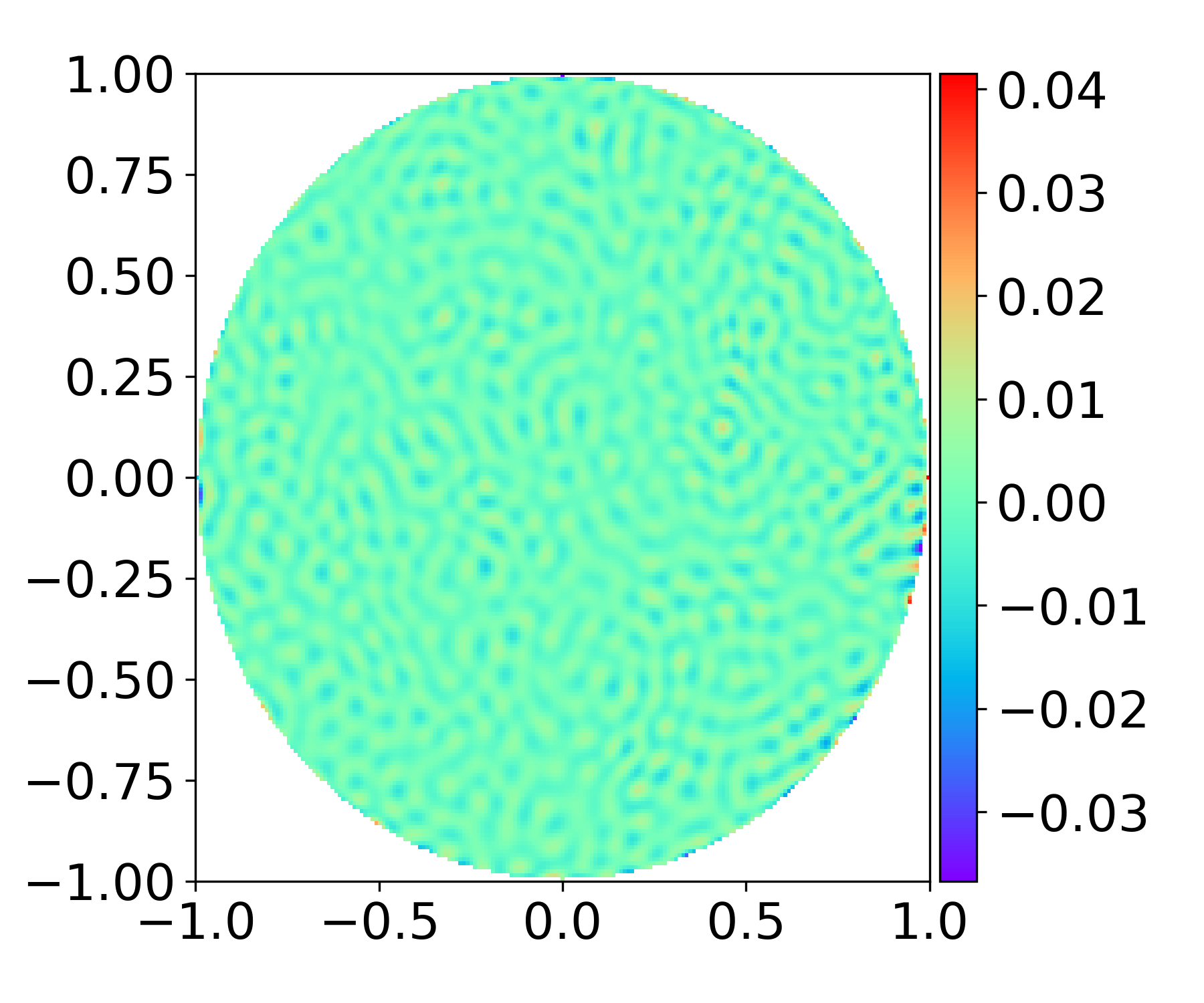}

\end{tabular}
  \caption{Example 8: Errors in selected modes for PINN (a), preconditioned PINN (b) and approximation (c); Spatial error distributions for PINN (d), preconditioned PINN (e) and  approximation (f).}
    \label{fig:example21_fourier_and_difference}
\end{figure}

The results are presented in Fig.~\ref{Example7ab} and Fig.~\ref{fig:example21_fourier_and_difference}, which show behaviors and patterns similar to those in Example 7. The use of NN representation allows one to deal with domain geometry easily. The corresponding training times are 231 seconds for direct approximation, 253 seconds for the preconditioned PINN, and 250 seconds for the standard PINN.

\textbf{Example 9:} We consider a two-dimensional multiscale variable-coefficient problem as follows.
\begin{equation*}\label{eq_PDE_multi}
\begin{cases}
    \nabla \cdot (k(x)\nabla u)=-1&\text{in}~\Omega=(-1,1)^2\\
    u = 0&\text{on}~\partial\Omega
\end{cases}\;,
\end{equation*}
where $x = (x_1,x_2)$ and the diffusion coefficient is defined by
\begin{equation*}
k(x) \;=\;
\frac{2 + A \sin\!\left(\dfrac{2\pi x_1}{\varepsilon}\right)}
     {2 + A \cos\!\left(\dfrac{2\pi x_2}{\varepsilon}\right)}
\;+\;
\frac{2 + A \sin\!\left(\dfrac{2\pi x_2}{\varepsilon}\right)}
     {2 + A \sin\!\left(\dfrac{2\pi x_1}{\varepsilon}\right)},
\end{equation*}
% where the diffusion coefficient is
% \begin{equation}
% k(x,y) \;=\; 
% \frac{2 + A \, \sin\!\left(\tfrac{2\pi x}{\varepsilon}\right)}
%      {2 + A \, \cos\!\left(\tfrac{2\pi y}{\varepsilon}\right)}
% \;+\;
% \frac{2 + A \sin\!\left(\tfrac{2\pi y}{\varepsilon}\right)}
%      {2 + A \, \sin\!\left(\tfrac{2\pi x}{\varepsilon}\right)},
% \end{equation}
with $A = 1.8$ and $\varepsilon =0.3$. This choice creates a diffusion coefficient with rapid oscillations in both directions. The PDE solution (Fig.\ref{fig:example23} (a)) exhibits a smooth macroscopic profile modulated by microscopic heterogeneity. We use an FMMNN architecture of width 200, rank 40, and depth 3, corresponding to 16,281 learnable parameters, and train with 10,000
sample points.

\begin{figure}[!htb]
\centering
\begin{tabular}{cc}
(a)&(b)\\
\includegraphics[width=0.33\textwidth]{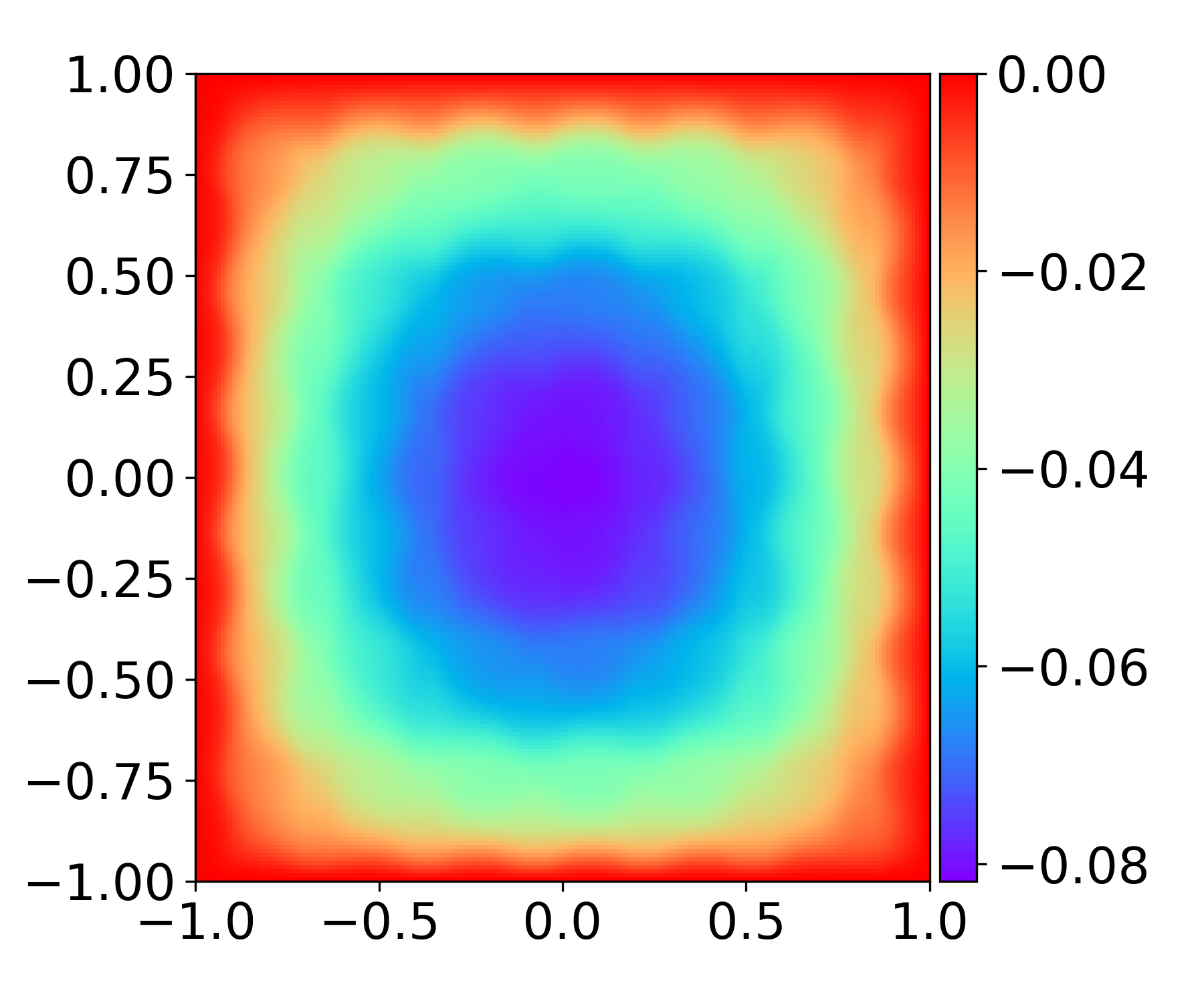}&
\includegraphics[width=0.45\textwidth]{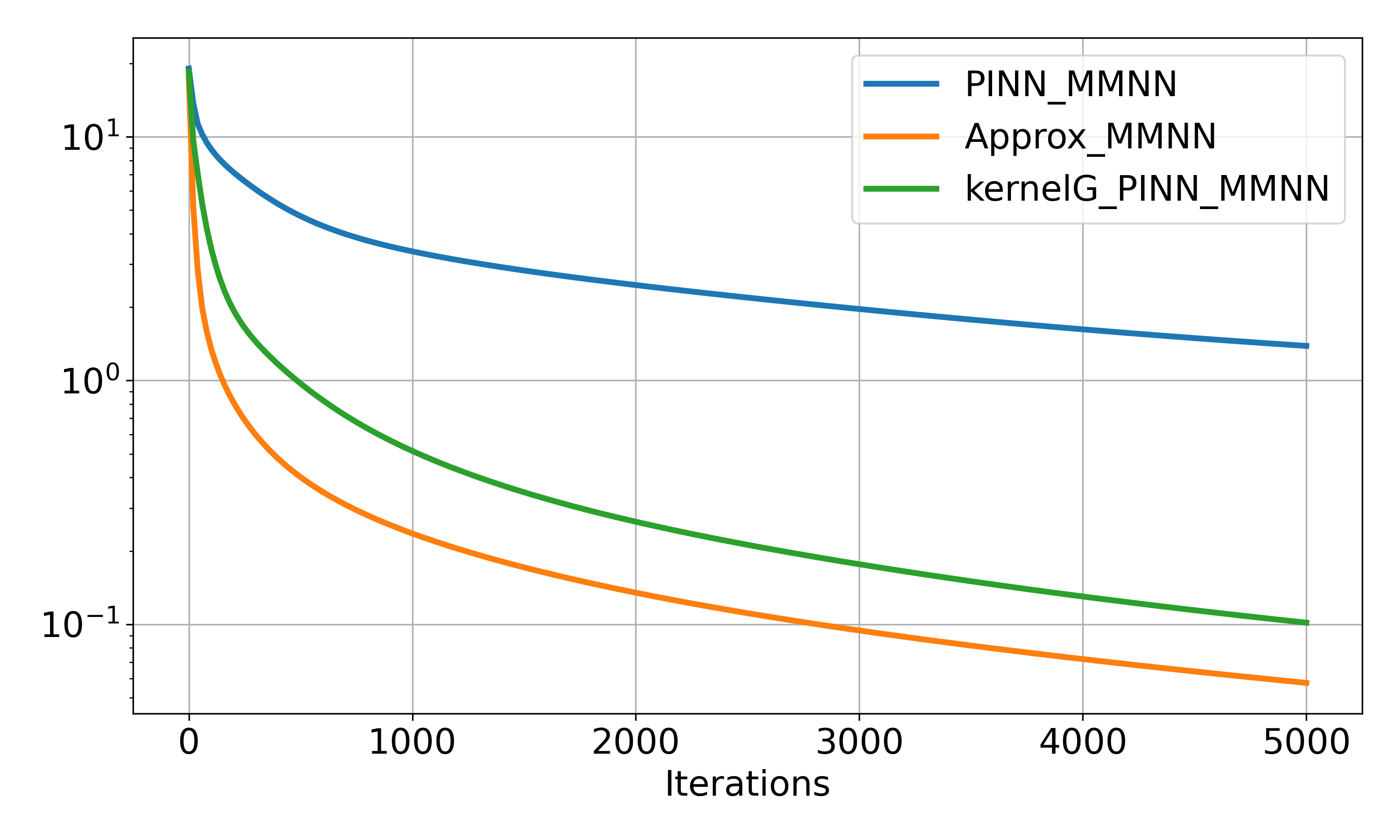}
\end{tabular}
 \caption{Example 9: (a) Exact solution; (b)
 Relative $L^2$ error during training for different models}
 \label{fig:example23}
    \end{figure}

\begin{figure}[!htb]
\centering
\begin{tabular}{ccc}
(a)&(b)&(c)\\
\includegraphics[width=0.33\linewidth]{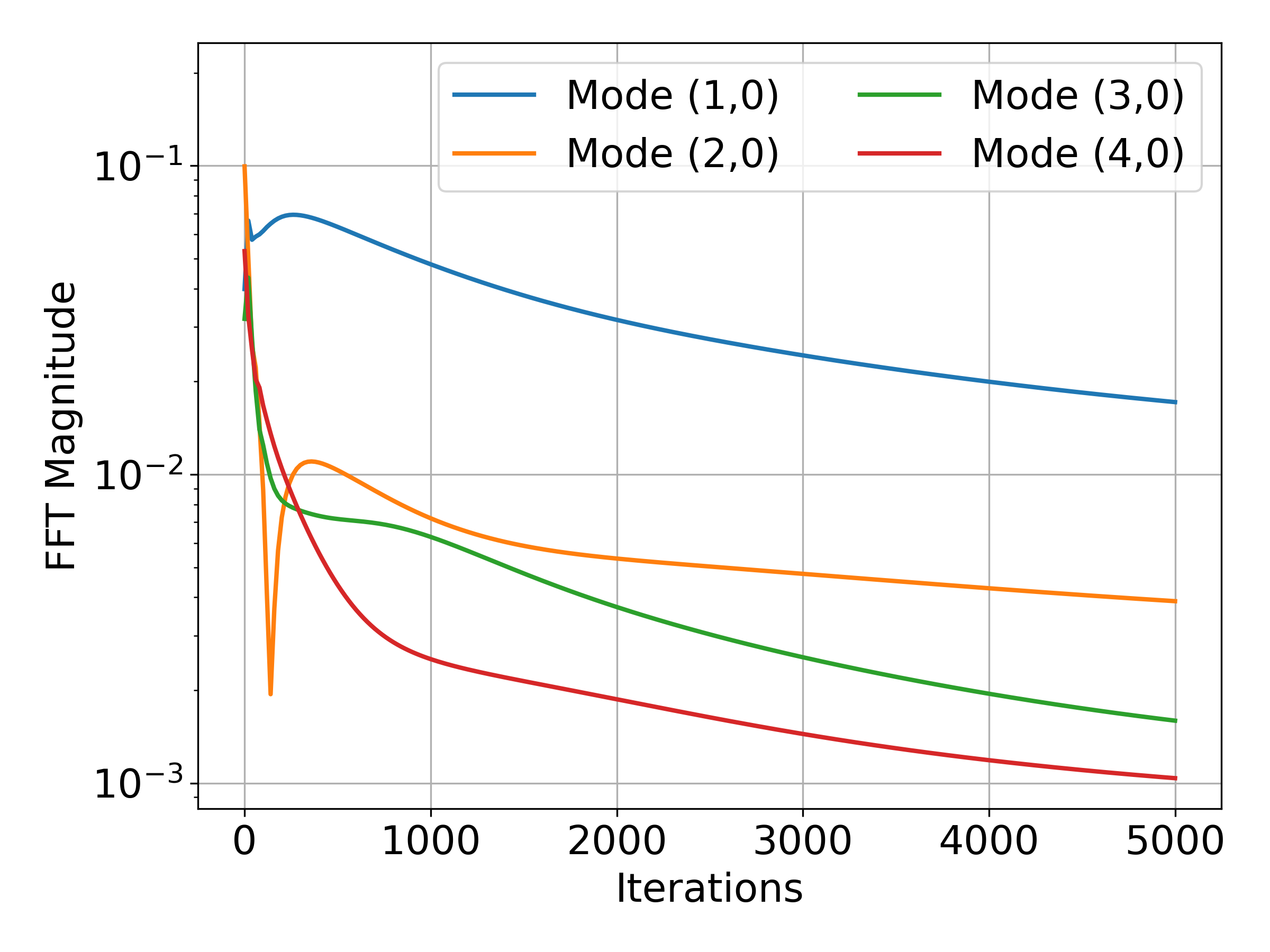}&
\includegraphics[width=0.33\linewidth]{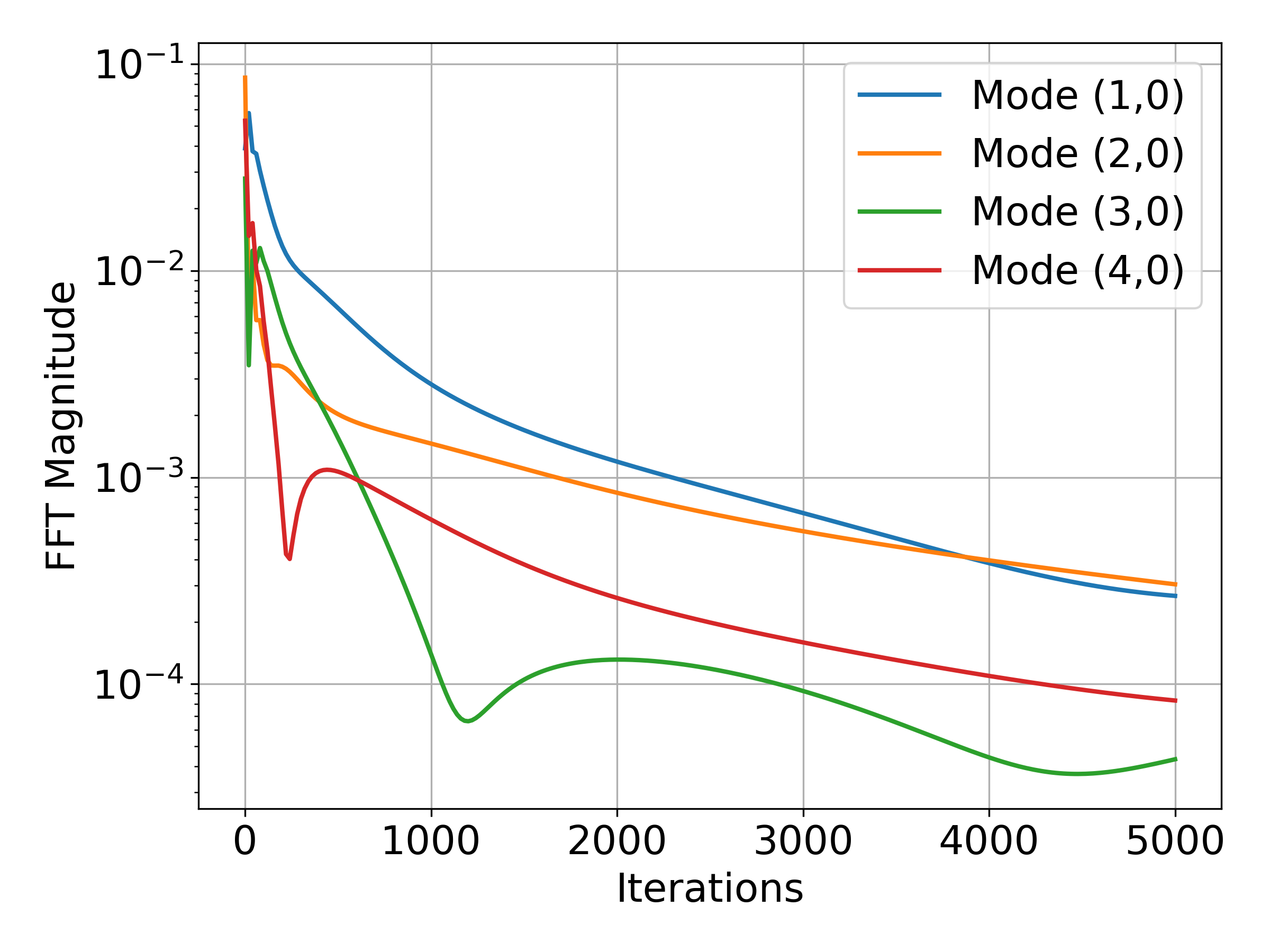}&
\includegraphics[width=0.33\linewidth]{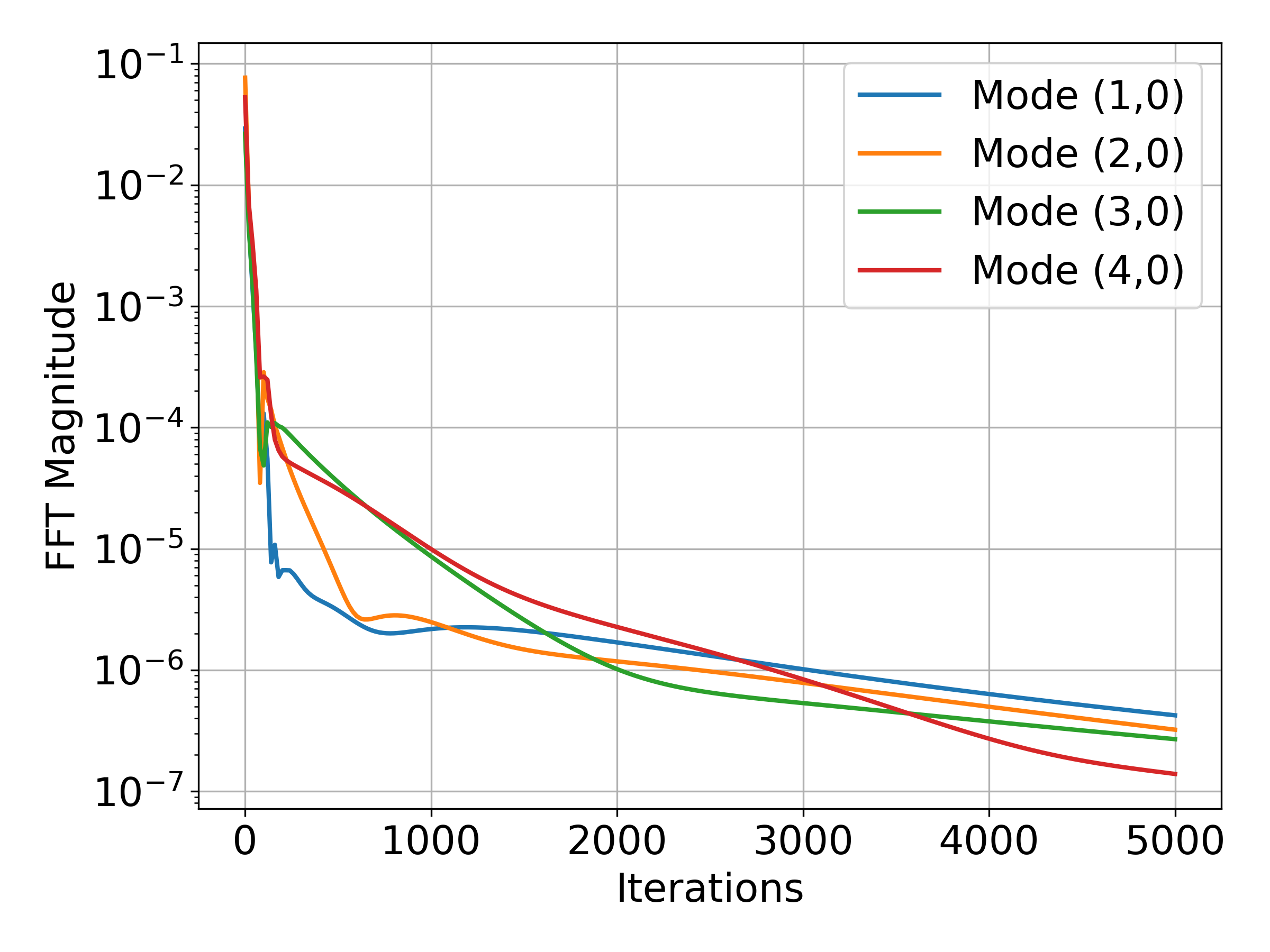}\\
(d)&(e)&(f)\\
\includegraphics[width=0.33\linewidth]{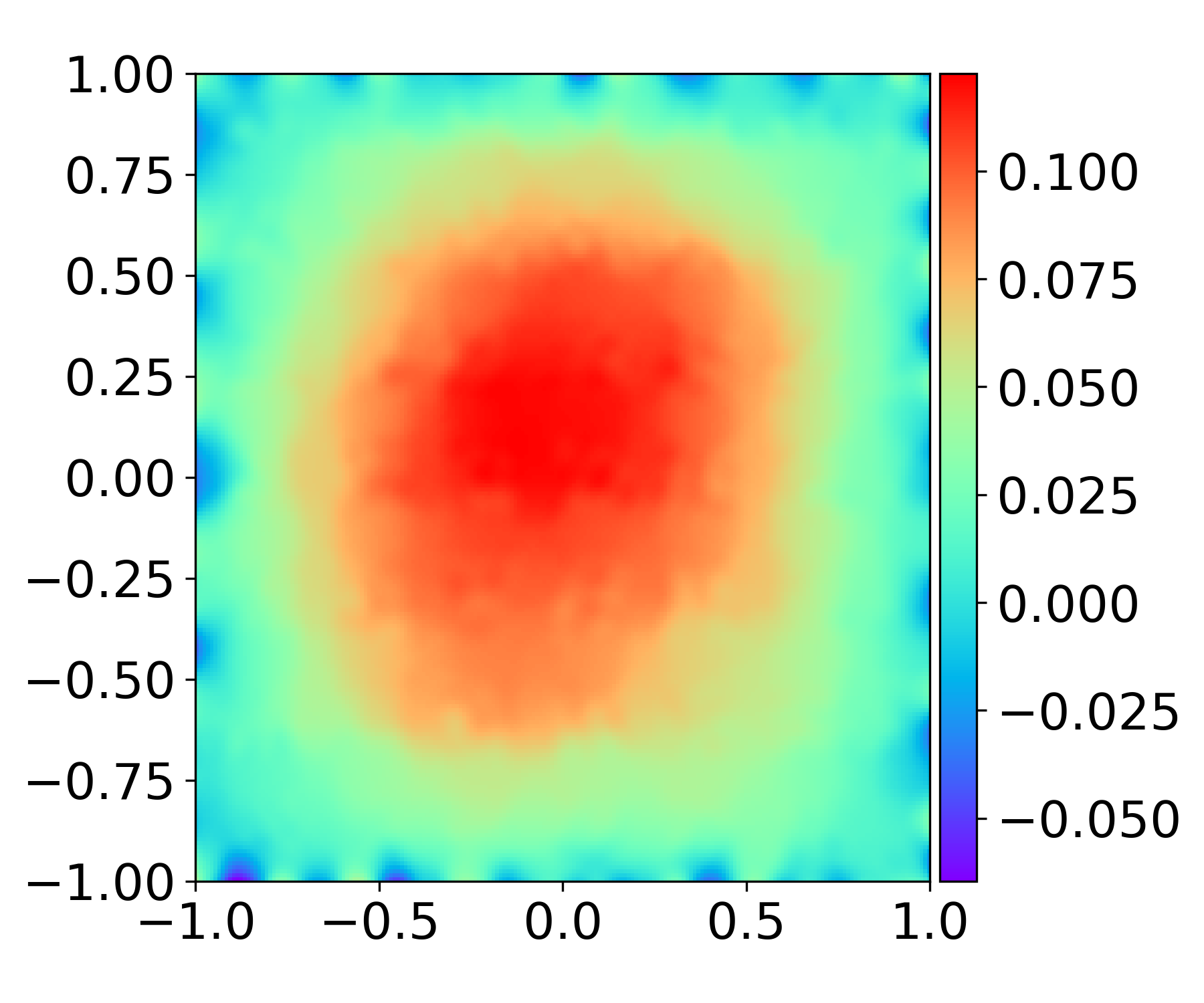}&
\includegraphics[width=0.33\linewidth]{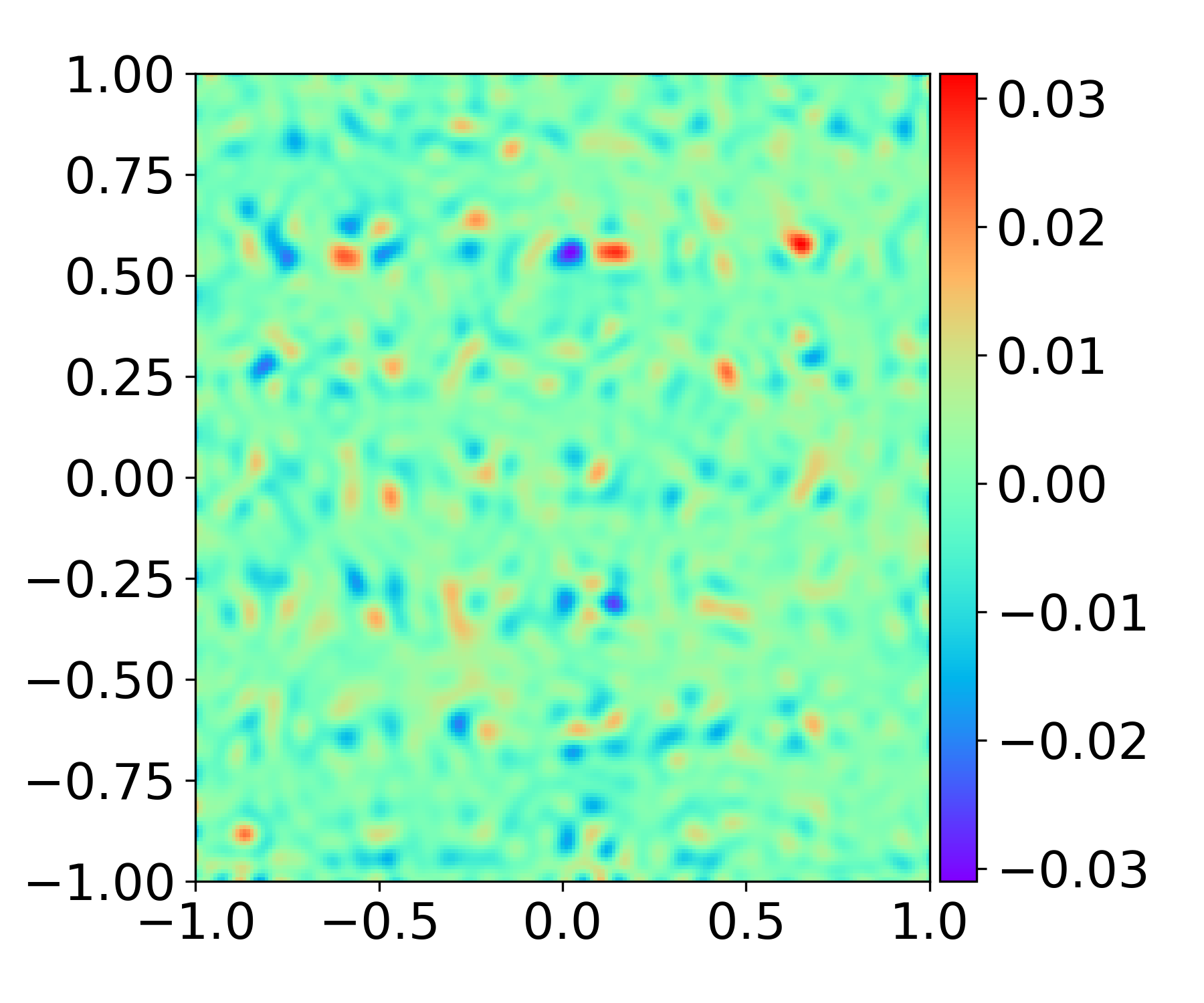}&
\includegraphics[width=0.33\linewidth]{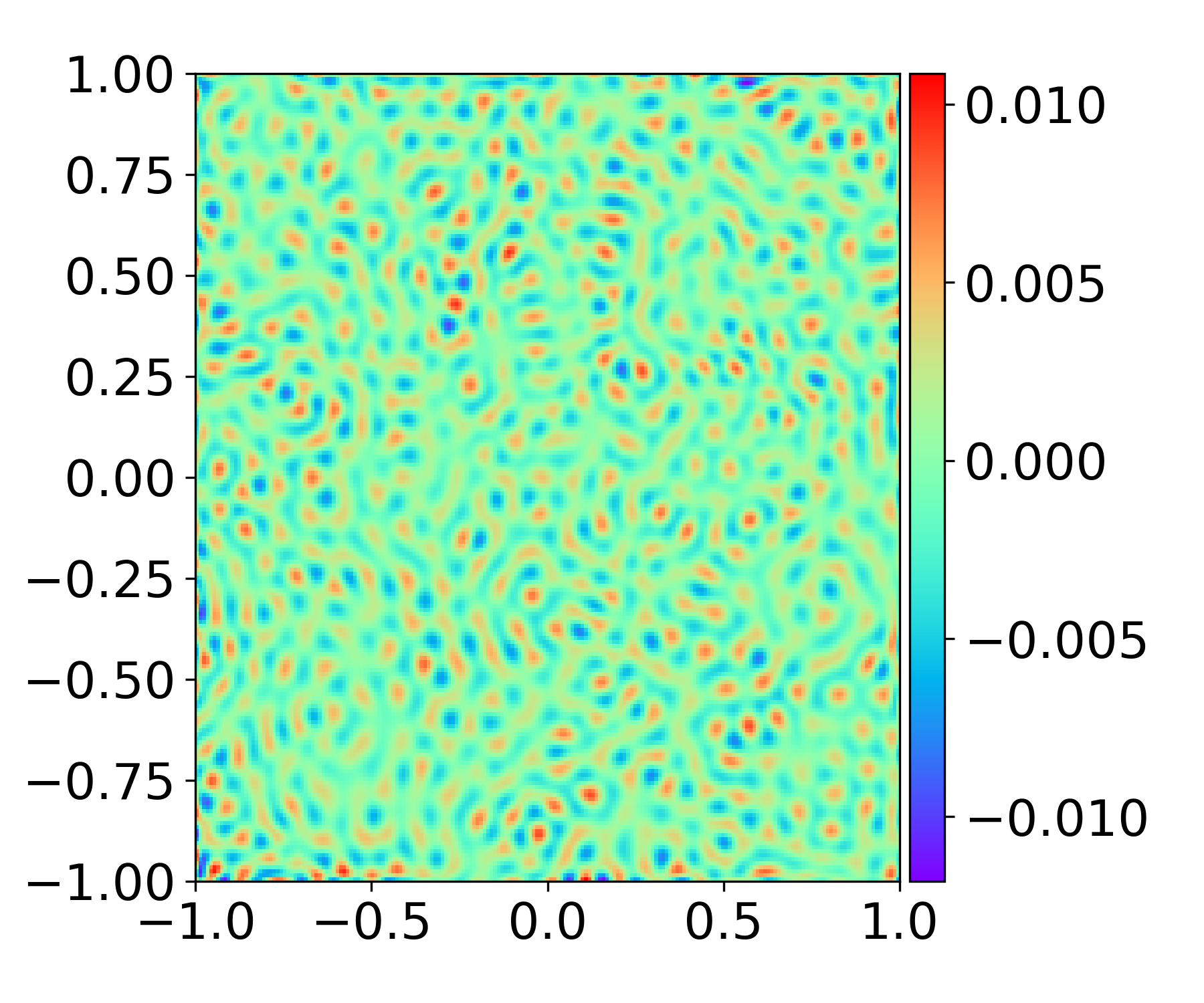}

\end{tabular}
  \caption{Example 9: Errors in selected modes for PINN (a), preconditioned PINN (b) and approximation (c); Spatial error distribution for PINN (d), preconditioned PINN (e) and  approximation (f)}
 \label{fig:example23_fourier_and_difference}
\end{figure}

This is a typical example of an equation with a multiscale feature in two dimensions. Without preconditioning, PINN fails to capture the solution, maintaining a relative $L^2$ error of $1.39$ , after 5,000 epochs. 
 However, the preconditioned PINN achieves much better accuracy with relative $L^2$ error of $1.02\times 10^{-1}$, although less accurate than the approximation with relative  $L^2$ error of $5.75\times 10^{-2}$. While the free-space Green's function for the Laplace equation used is not the exact one for this variable-coefficient operator, the frequency bias caused by the differential operator is pretty well offset by the approximate preconditioning, which is shown both from the error mode dynamics and the error plot in Fig.~\ref{fig:example23_fourier_and_difference}. This example demonstrates that preconditioning remains effective even for challenging multiscale variable-coefficient problems where the preconditioner is only approximate.

\textbf{Example 10:}  We solve the Poisson equation $\Delta u = -f$ in a unit ball $\Omega = B(0,1)\subset \mathbb{R}^3$ with homogeneous Dirichlet boundary conditions, and the source function $f$ is chosen such that the exact solution is
\[
\begin{aligned}
u(x) =\;
&1.0 \cdot \phi(x; c^{(1)}, 0.8) \\
&+ 0.7 \cdot \phi(x; c^{(2)}, 0.4) \\
&+ 0.5 \cdot \phi(x; c^{(3)}, 0.3),
\end{aligned}
\]
where
\[
c^{(1)} = (0.0, 0.0, 0.0), \quad
c^{(2)} = (0.5, 0.5, 0.1), \quad
c^{(3)} = (-0.4, 0.3, -0.1),
\]
and $\phi$ is the smooth compactly supported bump function
\[
\phi(x; c, s) =
\begin{cases}
\displaystyle
\exp\!\left(
-\dfrac{1}{1 - \dfrac{|x - c|^2}{s^2}}
\right),
& \text{if } |x- c| < s, \\[1.2em]
0, & \text{otherwise}.
\end{cases}
\]
 We use an MMNN architecture of width 200, rank 40, and depth 3, corresponding to 16,281 learnable parameters, and train with 65,450  sample points.

 In the three-dimensional setting, the frequency bias is analyzed using a spectral decomposition based on spherical Bessel functions and spherical harmonics. 
We consider basis functions of the form
\[
j_\ell(z_{\ell n} r)\, Y_{\ell m}(\theta,\varphi),
\]
which form an orthogonal system in the unit ball. 
Here $j_\ell$ denotes the spherical Bessel function of order $\ell$, 
$Y_{\ell m}$ is the spherical harmonic of degree $\ell$ and order $m$, 
$r=|x|$, and $z_{\ell n}$ is the $n$-th positive zero of $j_\ell$.
Each mode is indexed by a triplet $(\ell,m,n)$, where $\ell$ controls the angular frequency,
$m$ specifies the angular orientation, and $n$ determines the radial frequency.
To investigate frequency-dependent learning behavior, we track a selected set of modes $(\ell,m,n)$  to quantify how different spatial frequencies are learned. 
In addition, we visualize the spatial error on the cross-section $z=0$. The results are presented in Fig.~\ref{Example10} and Fig.~\ref{fig:example10_fourier_and_difference}, which show behaviors and patterns similar to those in the two-dimensional case. The mode dynamics and spatial error visualization verify that the frequency bias caused by the differential operator is substantially mitigated by the approximate preconditioning in the three-dimensional setting.
\begin{figure}[!htb]
\centering
\begin{tabular}{cc}
(a)&(b)\\
\includegraphics[width=0.33\textwidth]{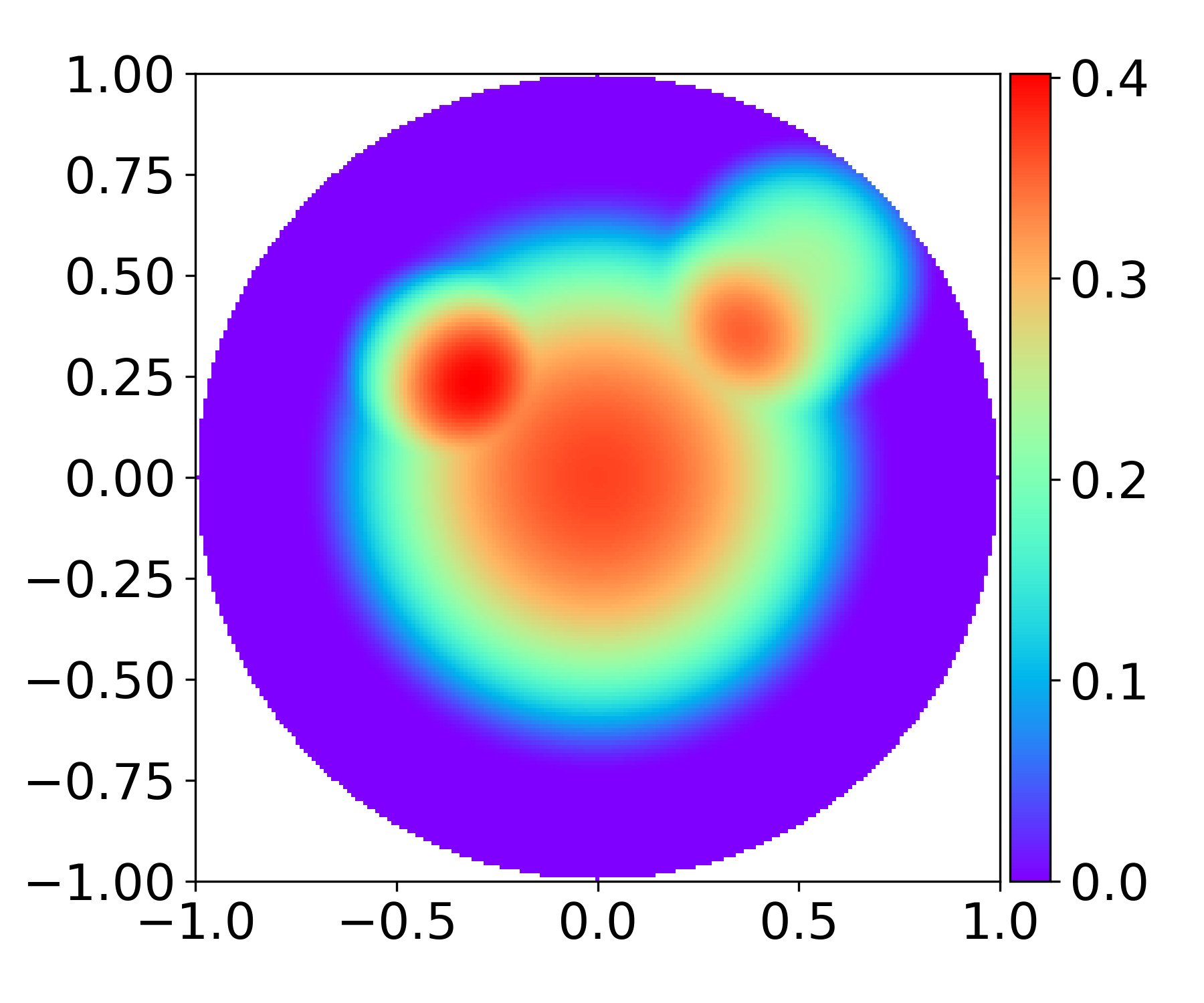}&
\includegraphics[width=0.45\textwidth]{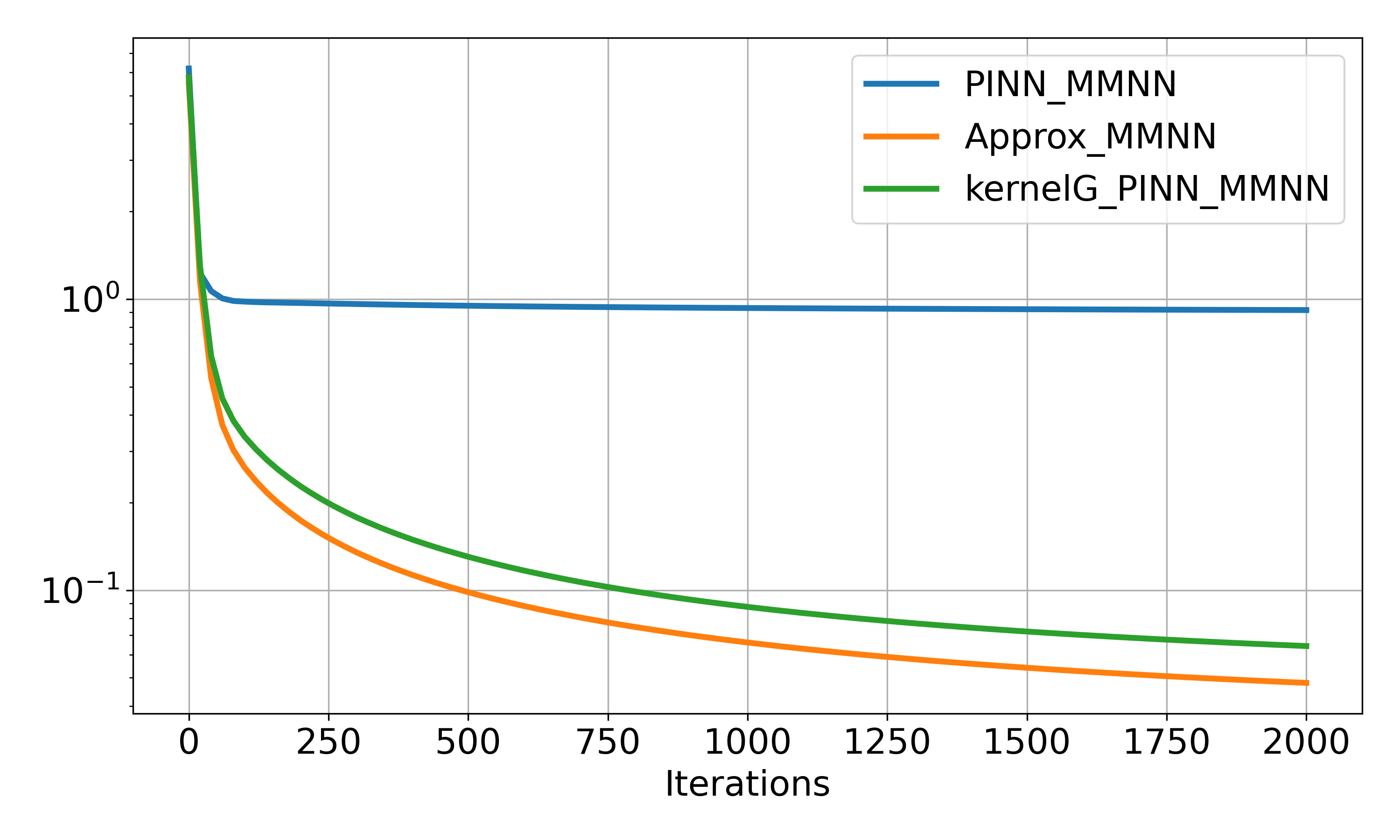}
\end{tabular}
 \caption{Example 10: 
(a) Cross-section of the exact solution at $z =0$;
 (b) Relative $L^2$ error during training for different models.}\label{Example10}
    \end{figure}

\begin{figure}[!htb]
\centering
\begin{tabular}{ccc}
(a)&(b)&(c)\\
\includegraphics[width=0.33\linewidth]{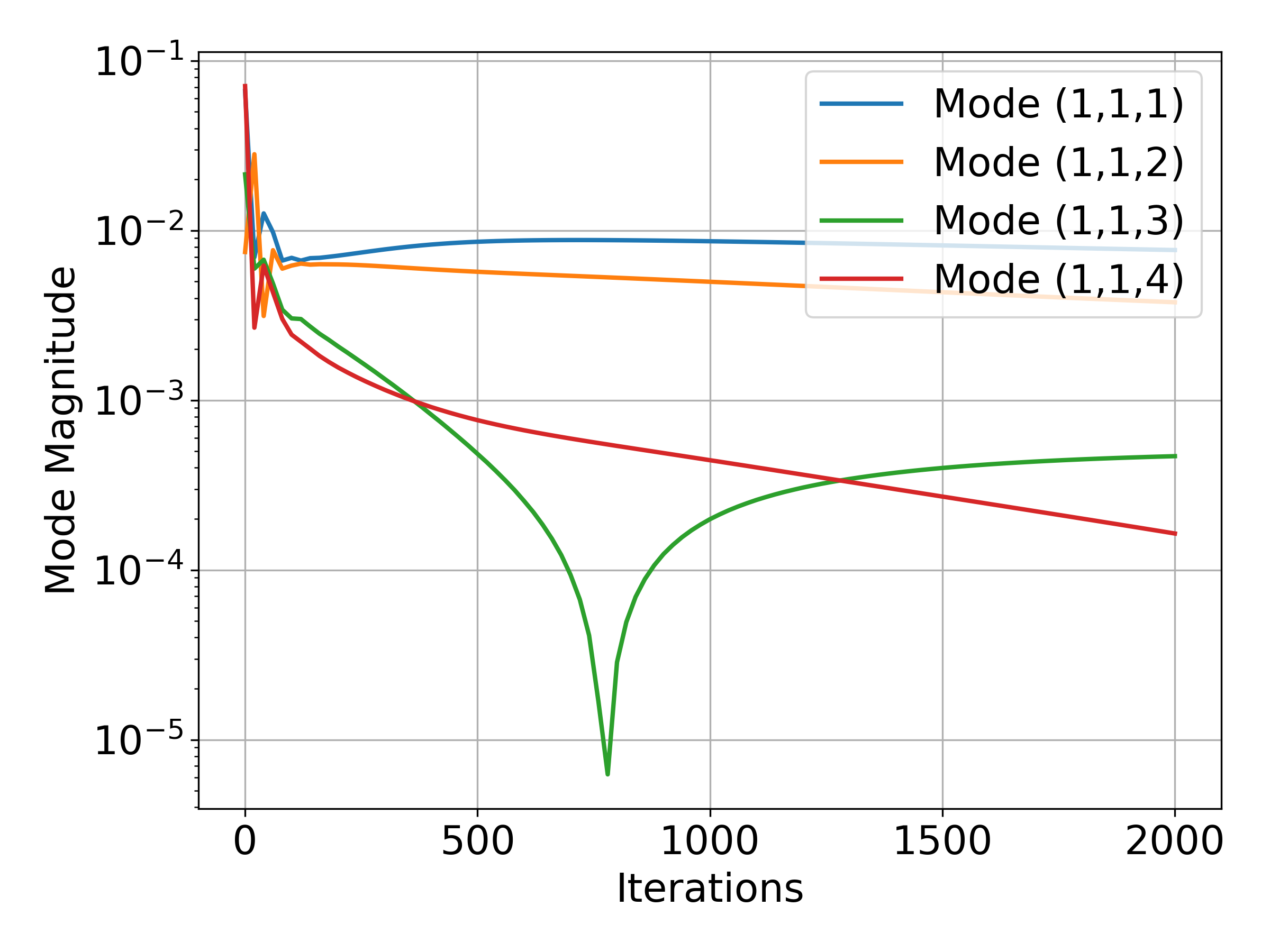}&
\includegraphics[width=0.33\linewidth]{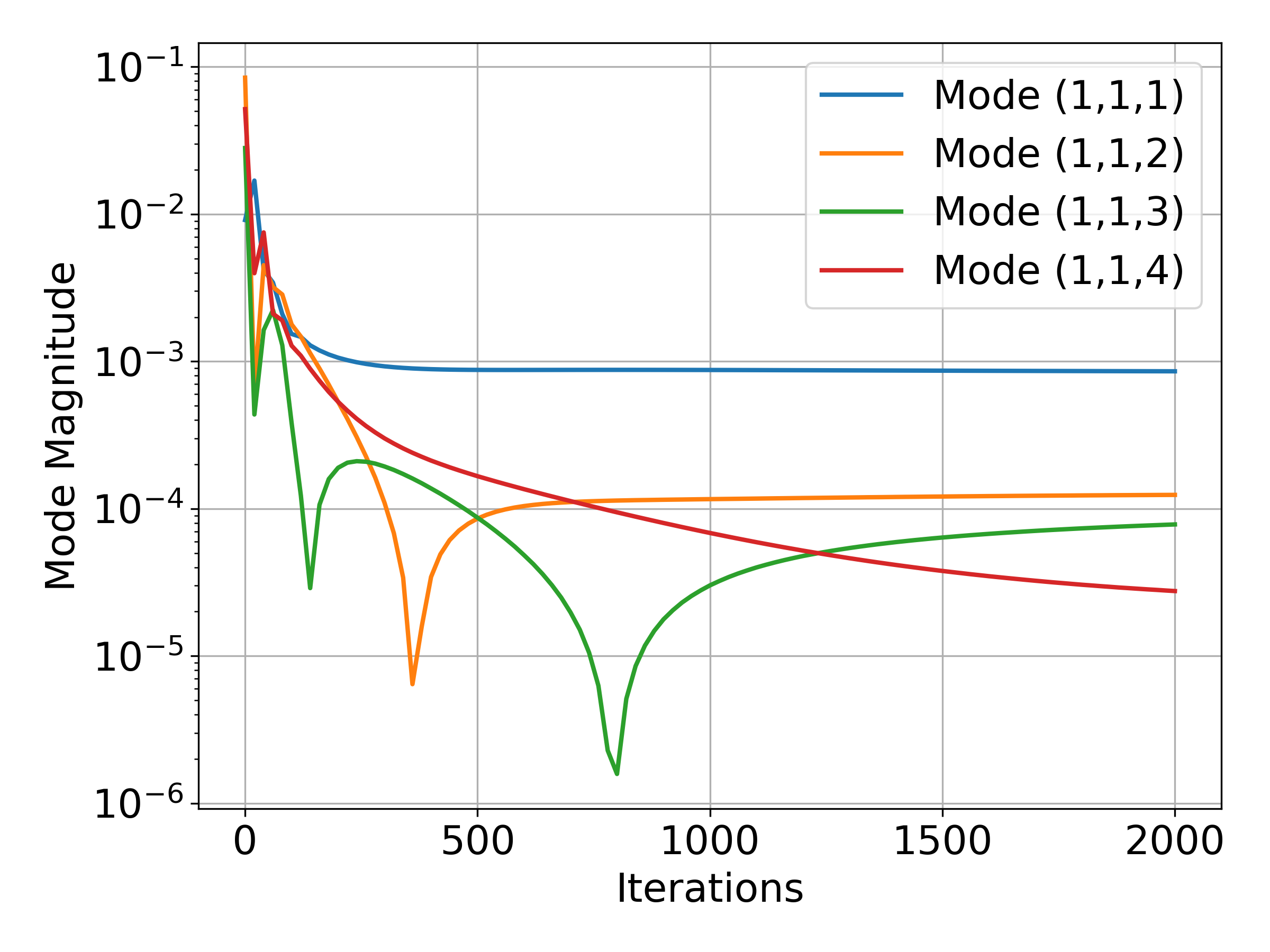}&
\includegraphics[width=0.33\linewidth]{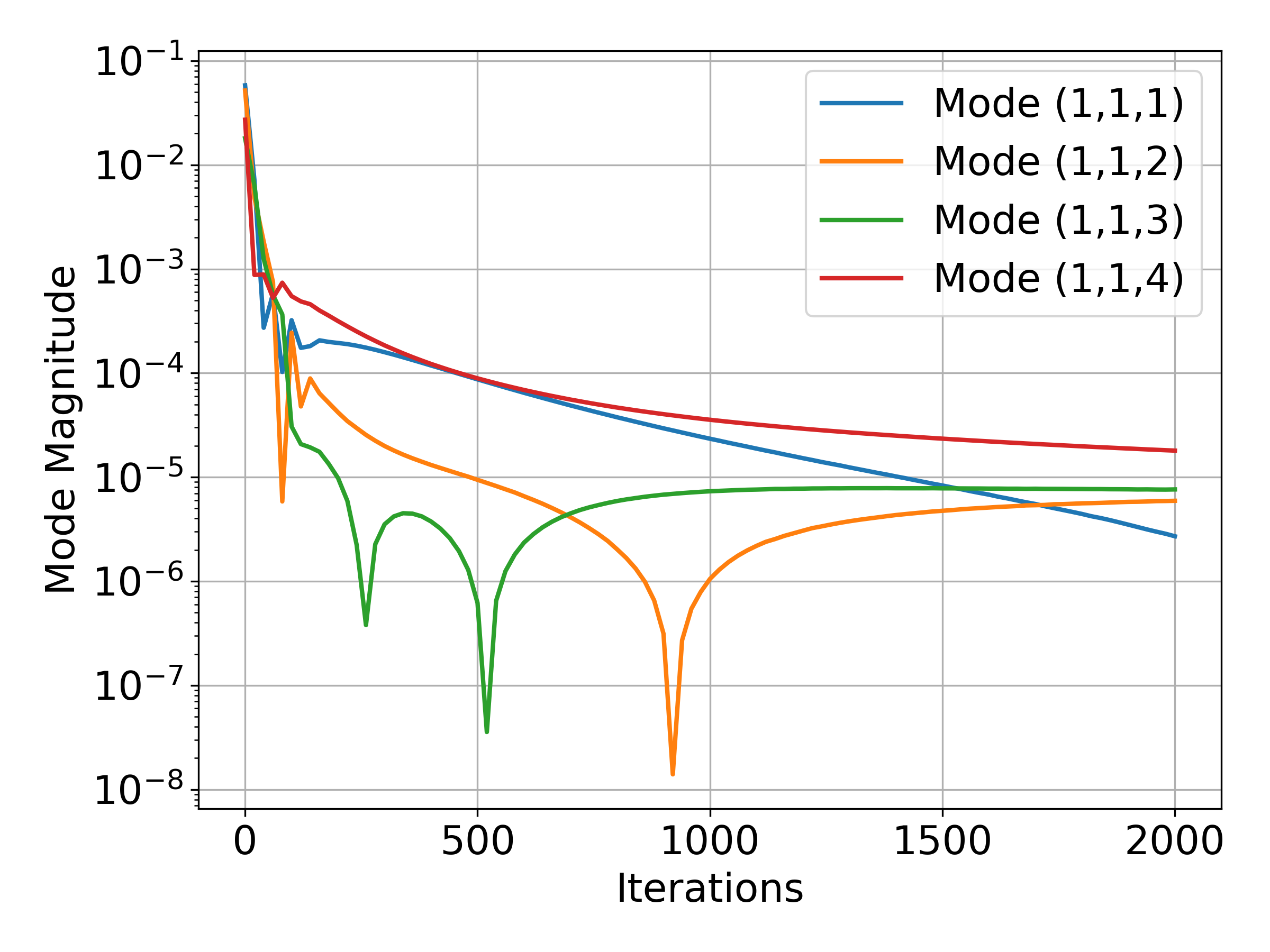}\\
(d)&(e)&(f)\\
\includegraphics[width=0.33\linewidth]{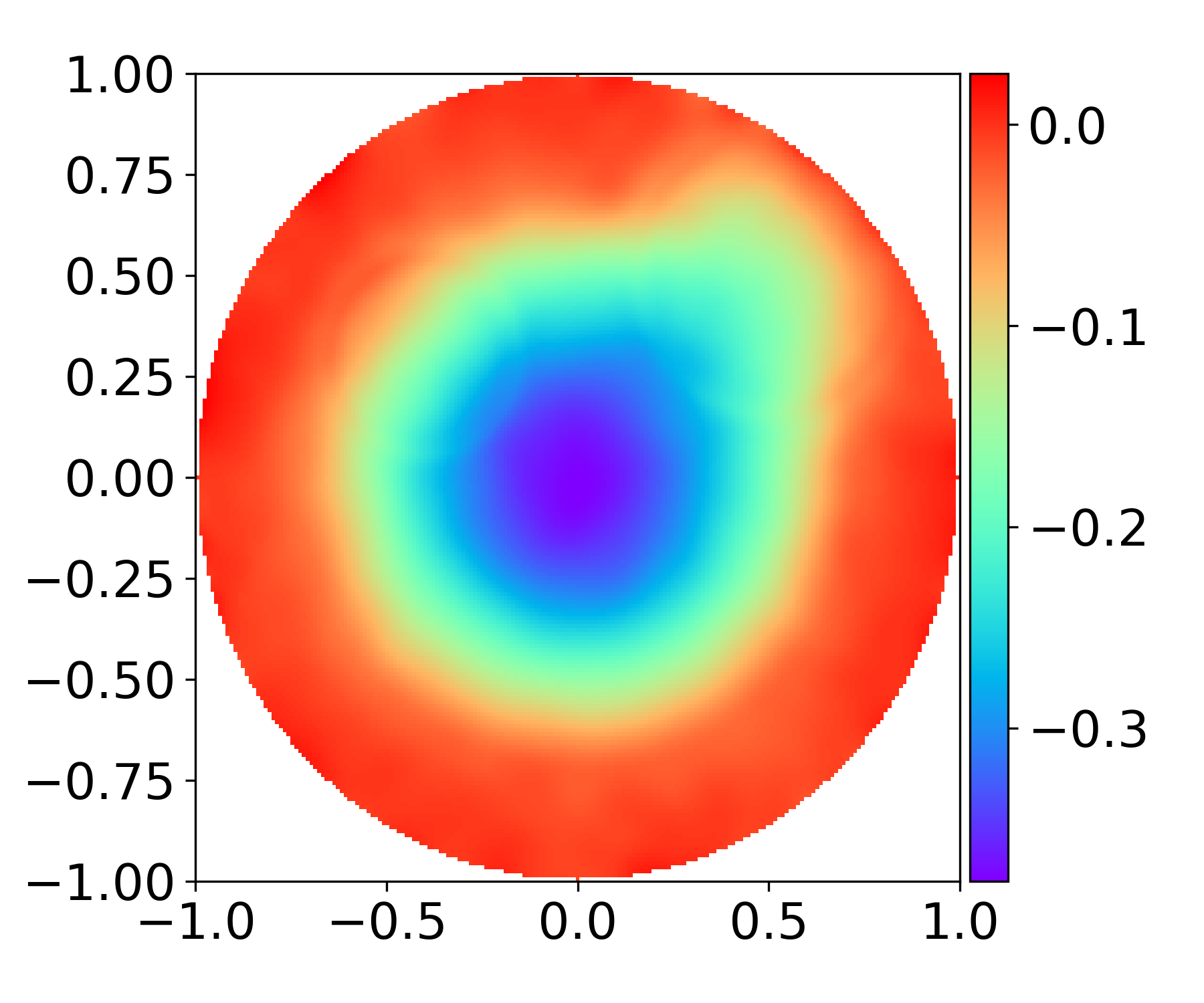}&
\includegraphics[width=0.33\linewidth]{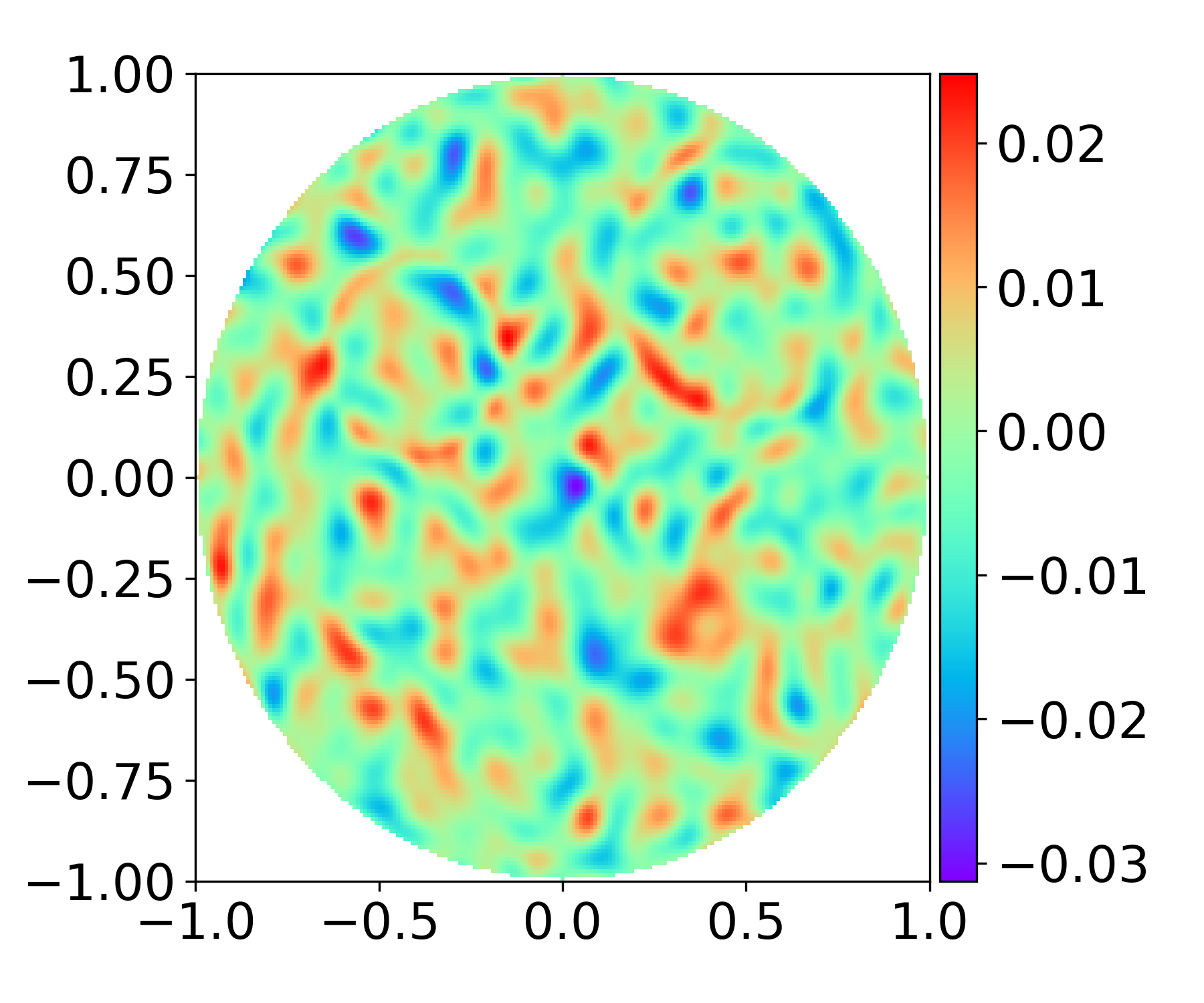}&
\includegraphics[width=0.33\linewidth]{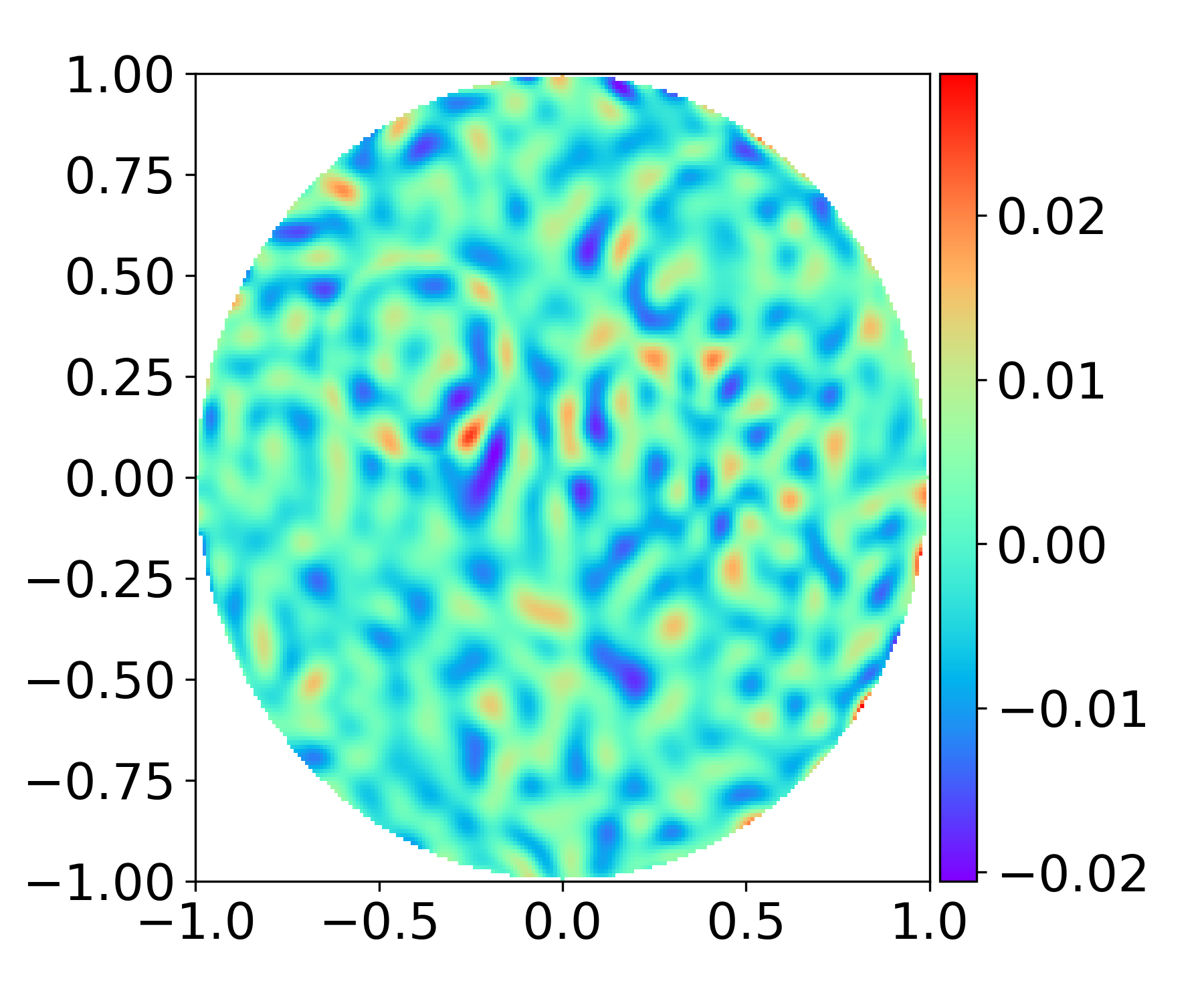}

\end{tabular}
  \caption{Example 10: Errors in selected modes for PINN (a), preconditioned PINN (b) and approximation (c); Spatial error distributions for PINN (d), preconditioned PINN (e) and  approximation (f).}
    \label{fig:example10_fourier_and_difference}
\end{figure}

\section{Conclusion}\label{Sec:Conclusion} 
In this work, we investigated the competition and balance between the spectral biases in NN representation and the differential operator when NNs are trained to solve elliptic PDEs. While strong low frequency bias of two-layer NNs will dominate differential operators as shown in~\cite{he2025can}, the situation can be quite different for multilayer NNs, which can have powerful representation capability and weak frequency bias when properly designed. In this work, it is shown that the inherent high frequency bias of a differential operator may dominate the overall formulation, which makes training inefficient even for smooth solutions. 

To address this issue, we propose an operator-aware preconditioning strategy that explicitly offsets the frequency bias of the differential operator and improves both efficiency and accuracy of the training process significantly. 
Extensive numerical experiments demonstrate the effectiveness of our strategy. 

A key take-home message from this study is that understanding spectral behaviors, which are strongly correlated with the efficiency and accuracy of NN representation and its training process, is essential in practice.
The operator-aware perspective proposed in this work suggests an effective problem-dependent approach, which also leads to interesting questions about other formulations, such as the deep Ritz method and other variational or residual-based neural solvers. Furthermore, in more complicated problems, such as operator learning, multiple competing loss terms, such as interior residuals, boundary errors, and operator evaluations, may exhibit distinct spectral behaviors, leading to complex training dynamics. The design of effective preconditioning strategies that can achieve an overall spectral balance will be essential.

%Although our analysis and preconditioning strategy in this work address the PINN formulation, the spectral bias induced by different optimization formulation or PDE types, such as the deep Ritz method and other variational or residual based neural solvers, is still not well understood. Moreover, in both operator learning and Neural PDE solvers, multiple competing loss terms, such as interior residuals, boundary penalties, and operator evaluations, may exhibit distinct spectral behaviors, leading to complex training dynamics. While recent works have explored balancing these terms at the level of gradient magnitudes, how to design preconditioning strategies that interact with the underlying spectral bias of the formulation and with operator-induced effects remains unclear. Addressing these issues will be essential for developing robust and scalable neural PDE solvers.

\section*{Acknowledgments}
Hongkai Zhao is partially supported by the National Science Foundation through grant DMS-2309551. Yimin Zhong is partially supported by the National Science Foundation through grant DMS-2309530. Roy Y. He is partially supported by NSFC grant 12501594, PROCORE-France/Hong Kong Joint Research Scheme by the RGC of Hong Kong and the Consulate General of France in Hong Kong (F-CityU101/24), StUp - CityU 7200779 from City University of Hong Kong, and the Hong Kong Research Grant
Council ECS grant 21309625.

\bibliographystyle{alpha}
\bibliography{sample}

\end{document}